\documentclass{amsart} 
\usepackage{amsmath}
\usepackage[mathscr]{eucal} 
\usepackage{amssymb}
\usepackage{latexsym}
\usepackage{amsthm} 
\theoremstyle{plain}
\newtheorem{theorem}{Theorem}[section]

\newtheorem{proposition}[theorem]{Proposition}

\newtheorem{lemma}[theorem]{Lemma}  
\newcommand{\supp}{\mathop{\mathrm{supp}}\nolimits} 
 
\newcommand{\card}{\mathop{\mathrm{card}}\nolimits} 
 
\newcommand{\re}{\mathop{\mathrm{Re}}\nolimits} 
\numberwithin{equation}{section}  
\theoremstyle{definition}

\theoremstyle{remark}
\newtheorem{remark}[theorem]{Remark}

\def\XXint#1#2#3{{\setbox0=\hbox{$#1{#2#3}{\int}$}
\vcenter{\hbox{$#2#3$}}\kern-.5\wd0}}

\title[Results in estimates for $k$-plane transforms]
{Results in estimates for $k$-plane transforms}  
\author{Shuichi Sato} 
 
\begin{document} 
\address{Department of Mathematics,
Faculty of Education,
Kanazawa University,
Kanazawa 920-1192,
Japan}
\email{shuichi@kenroku.kanazawa-u.ac.jp, shuichipm@gmail.com}
\begin{abstract}  
This is an expository paper. 
We give proofs of some results of M. Christ (1984) and S. W. Drury (1984)  
for $k$-plane transforms.  Also, we give proofs for 
some related results including that for the existence of invariant measures 
on certain homogeneous manifolds of Lie groups.   
  \end{abstract}
  \thanks{
  \endgraf
  }
\thanks{The author is partly supported by Grant-in-Aid
for Scientific Research (C) No. 20K03651, Japan Society for the  Promotion of 
Science.}

\maketitle

\section{Introduction }  
In this note, proofs will be presented in some detail  
for results of M. Christ \cite{Ch} in 1984 
(Theorem \ref{the1} below) and S. W. Drury \cite{D} in 1984 (Lemma \ref{L1.1} 
below; see also a result of Blaschke which 
can be found in \cite[Chap. 12]{Sant})  
related to  $k$-plane transforms $T_{k,n}$.  
For a function $f$ on $\Bbb R^n$ and an affine $k$-plane $p$, 
$1\leq k<n$, the value of $T_{k,n}f$ 
at $p$ is defined to be $\int_p f$, where the integration is with respect 
to the $k$-dimensional Lebesgue measure on $p$ (see \eqref{ee1.5} below 
for a more precise definition). 
When $k=1$, $T_{k,n}$ is called the $X$-ray transform and when $k=n-1$, it 
is called the Radon transform (see \cite{CDR} for applications 
 in harmonic analysis; also related results can be found in \cite{Sa}).
\par 
Let $G_{k,n}$ be the Grassmannian manifold of all $k$-planes in $\Bbb R^n$ 
passing through the origin ($1\leq k<n$). 
Theorem \ref{the1} and Lemma \ref{L1.1} are stated in terms of the 
${\rm SO}(n)$ invariant measure $d\sigma$ on $G_{k,n}$. 
To prove Lemma \ref{L1.1} we shall show 
Lemma \ref{L2.3}, which is also in \cite{D} and 
 can be regarded as a polar coordinates expression of 
the Lebesgue measure on $\Bbb R^n$ when $k=1$.   Lemma \ref{L2.3} can be 
used to prove  Lemma \ref{L2.6} too, which is an analogue of Lemma \ref{L2.3}, 
where $\Bbb R^n$ is replaced by $S^{n-1}$ (the unit sphere in $\Bbb R^n$) and 
the Lebesgue measure on $\Bbb R^n$ is replaced by the Lebesgue surface measure 
on $S^{n-1}$. 
Our proof of Lemma \ref{L2.3} is different from that of \cite{D}; it is 
based on straightforward computations 
concerning local coordinates on the Grassmannian manifolds 
and the Gram determinants (see Boothby \cite[pp. 63--65]{Bo} 
and Spivak \cite[Section 2 of Chapter 13]{Sp} for the local coordinates). 
\par 
We also consider an operator $S_{k,n}$ related to $T_{k,n}$,  
which maps a function on $S^{n-1}$ to a function on $G_{k,n}$ 
(see \eqref{Skn} below) and prove in Theorem \ref{the2} 
estimates for $S_{k,n}$ between the Lorentz space $L^{n/k, n}(S^{n-1})$  and  
the Lebesgue space $L^n(G_{k,n})$ 
with respect to the Lebesgue surface measure on $S^{n-1}$ and 
the measure $d\sigma$ on $G_{k,n}$, respectively.  
Theorem \ref{the1} follows from a more general multilinear estimates in 
Proposition \ref{prop1.3}. 
We shall prove Proposition \ref{prop1.3} by using 
Theorem \ref{the2} together with  Lemmas \ref{L1.1} and \ref{L2.6}. 
 Theorem \ref{the2} follows from Proposition \ref{p4.1} and from 
 more general results in Proposition \ref{p4.2} for $k\geq 2$, which 
are also stated as multilinear estimates.  
\par 
Proposition \ref{prop1.3} and Proposition \ref{p4.2}
will be shown by using the estimates in 
\eqref{e10} and  \eqref{e2.15} below, respectively, 
by applying interpolation arguments.  
In this note, we give proofs of interpolation results required 
in the arguments for the multilinear estimates 
(Sections \ref{sec5}--\ref{sec9}).  
For basic results on interpolation, we mainly refer to the book 
Bergh-Lofstrom \cite{BL} but we reproduce proofs 
of some results important for this note.  
To prove \eqref{e2.15} we also apply 
 induction arguments using Lemma \ref{L2.6} and  
Lemma \ref{L2.10} on the Gram determinants.  
\par 
Finally, we give a proof of the existence of invariant measures on 
certain homogeneous manifolds of Lie groups including the Grassmannian 
manifolds (Sections \ref{j8intro}--\ref{sec14}).  
The invariant measures on the  Grassmannian manifolds are used in stating 
 results mentioned above. 
The Grassmannian manifold is not always orientable but always admits an 
invariant measure with respect to ${\rm SO}(n)$. A condition for a homogeneous 
manifold $G/H$ to admit an 
invariant measure with respect to the action of $G$ 
will be stated in terms of the determinants of 
linear mappings in the adjoint representations of a Lie group $G$ and 
its closed subgroup $H$ on non-singular linear transformation groups of 
their Lie algebras, which will be applied to the case of the Grassmannian 
manifolds.  
\par 
For $\theta \in G_{k,n}$, define 
$$ \theta^\perp=\{x\in \Bbb R^n: x\perp \theta \}, $$ 
where 
$x\perp \theta$ means that $\langle x, y \rangle =0$ for all $y\in \theta$;  
$\langle x, y \rangle$ denotes the inner product in $\Bbb R^n$:  
$$\langle x, y \rangle= \sum_{j=1}^nx_jy_j; 
\quad x=(x_1, \dots , x_n), \quad y=(y_1, \dots , y_n). $$ 
Let 
\begin{equation}\label{def1}
\Bbb S=\{(x,\theta): x\in \theta^\perp, \theta \in G_{k,n}\}. 
\end{equation}  
This can be regarded as a parameterization of all affine $k$-planes in 
$\Bbb R^n$.  
We write $\pi=(x,\theta)$ for $(x,\theta)\in \Bbb S$. 
We define a measure $d\nu$ on $\Bbb S$ by 
\begin{equation}\label{def2} 
d\nu(\pi)=d\nu(x,\theta)= d\lambda_{\theta^\perp}(x)\, d\sigma(\theta), 
\end{equation} 
where 
\begin{enumerate} 
\item $d\lambda_{\theta^\perp}$ is the $n-k$ dimensional Lebesgue measure 
on the  hyperplane $\theta^\perp$, 
 which is considered as a singular measure on $\Bbb R^n$; 
\item $d\sigma$ denotes the ${\rm SO}(n)$-invariant measure on $G_{k,n}$, where ${\rm SO}(n)$ denotes the special orthogonal group on $\Bbb R^n$ 
(see Proposition \ref{j8p1} and Remark \ref{re14.2} in Section \ref{sec14}). 
\end{enumerate}   
\par 
Let $d\lambda_{P(x,\theta)}$ be a measure supported 
on $P(x,\theta)=\{x+y : y\in \theta\}$, $x\in \Bbb R^n$, $\theta\in G_{k,n}$, 
defined as 
\begin{equation}\label{def3}
\int_{\Bbb R^n} f(z)\ d\lambda_{P(x,\theta)}(z)
=\int_{\theta} f(x+y)\, d\lambda_\theta(y),    
\end{equation} 
where $d\lambda_\theta(y)$ is the $k$-dimensional Lebesgue measure on 
$\theta$ considered as a singular measure on $\Bbb R^n$: 
$d\lambda_\theta=d\lambda_{P(0, \theta)}$.  If we decompose 
$$x=x_\theta + x_{\theta^{\perp}}, \quad \text{where \quad 
$x_\theta\in \theta,  \quad x_{\theta^{\perp}}\in \theta^\perp$, } 
 $$  
 then it is easy to see that 
\begin{equation} \label{ee1.4}
 \int f(x+y)\, d\lambda_\theta(y)= \int f(x_{\theta^\perp} +y)
 \, d\lambda_\theta(y). 
\end{equation}

\par 
For $(x,\theta)\in \Bbb S$, let 

\begin{equation}\label{ee1.5}  
T_{k,n}(f)(x,\theta)
=\int f(z)\, d\lambda_{P(x,\theta)}(z)=\int f(x+y)\, d\lambda_\theta(y). 
\end{equation} 
For $\theta \in G_{k,n}$,  
let $S^{k-1}_\theta$ be the unit sphere in $\theta$: 
\begin{equation*} 
S^{k-1}_\theta =S^{n-1}\cap \theta =\left\{ y\in \theta : |y|=\langle y,y 
\rangle^{1/2}= 1\right\}. 
\end{equation*}
Define 
\begin{equation} \label{Skn}
S_{k,n}(f)(\theta)
=\int_{S^{k-1}_\theta} f(\omega)\, d\lambda_{\theta}(\omega), 
\end{equation}
where $ d\lambda_{\theta}(\omega)$ is the unique probability measure on 
$S^{k-1}_\theta$ invariant under the action of ${\rm SO}(k)$ on $\theta$; 
here we note that  
${\rm SO}(k)={\rm SO}_\theta(k)= \{O \in {\rm SO}(n): O(\theta) 
\subset \theta\}$. 
\par 
Let $C_0(\Bbb R^n)$ be the set of all continuous functions on $\Bbb R^n$ 
with compact support. Let $C(S^{n-1})$ be the set of all continuous functions 
on $S^{n-1}$. 
We have the following results (see \cite{Ch}). 
\begin{theorem}\label{the1} Let $1 \leq k<n$, $k\in \Bbb Z$ 
$($the set of integers$)$. 
For $f\in C_0(\Bbb R^n)$ we have 
$$ \|T_{k,n}f\|_{L^{n+1}(\Bbb S)}\leq  
C\|f\|_{L^{\frac{n+1}{k+1}, n+1}(\Bbb R^n)}, $$  
where 
$$\|T_{k,n}f\|_{L^{n+1}(\Bbb S)}
=\left(\int_{\Bbb S} \left|T_{k,n}(f)(x,\theta)\right|^{n+1}\, d\nu(x,\theta) 
\right)^{\frac{1}{n+1}}  $$  
and $L^{p,q}(\Bbb R^n)$ denotes the Lorentz space on 
$\Bbb R^n$ $($\!see \cite{H}, \cite[Chap. V]{SW}$)$. 
 
\end{theorem}

\begin{theorem}\label{the2} For $f\in C(S^{n-1})$ we have 
$$ \|S_{k,n}f\|_{L^n(G_{k,n})}= 
\left(\int_{G_{k,n}} \left|S_{k,n}(f)(\theta)\right|^n
\, d\sigma(\theta)\right)^{\frac{1}{n}}  
\leq  C\|f\|_{L^{\frac{n}{k}, n}(S^{n-1})}, $$  
where $1\leq k<n$, $k\in \Bbb Z$ and the Lorentz space 
$L^{\frac{n}{k}, n}(S^{n-1})$ is defined with respect to the Lebesgue 
surface measure on $S^{n-1}$. 
\end{theorem} 
We also write $\|f\|_{p,q}$ for $\|f\|_{ L^{p,q}}$. 
\par 
For a function $g$ on $\Bbb S$, we consider a mixed norm 
\begin{equation*} 
\|g\|_{L^q(L^r)} =\left(\int_{G_{k,n}} \left(\int |g(x,\theta)|^r \, 
d\lambda_{\theta^\perp}(x)\right)^{q/r}\, d\sigma(\theta)\right)^{1/q}. 
\end{equation*} 
Then the norm on the left hand side of the inequality in the conclusion of 
Theorem \ref{the1} can be expressed as $\|T_{k,n}f\|_{L^{n+1}(L^{n+1})}$. 
We recall two more theorems from \cite{Ch}.  
\begin{theorem}[Drury's conjecture]\label{thm1.3}  
Suppose that $np^{-1} - (n-k)r^{-1}= k, 1\leq p\leq (n+1)/(k+1)$ and 
$q\leq (n-k)p'$, where $p'$ denotes the exponent conjugate to $p$. Then 
$$\|T_{k,n}f\|_{L^{q}(L^{r})} \leq C\|f\|_p, $$   
where $\|f\|_p$ denotes the norm of $f$ in $L^p(\Bbb R^n)$. 
\end{theorem}  

\begin{theorem}\label{thm1.4}   
Suppose that $np^{-1} - (n-k)r^{-1}= k, p\leq 2, p<n/k$ and $q\leq (n-k)p'$. 
Then 
$$\|T_{k,n}f\|_{L^{q}(L^{r})} \leq C\|f\|_p.  $$    
\end{theorem}  
Theorem \ref{the1} implies 
Theorem \ref{thm1.3} for $p=(n+1)/(k+1)$, $r=n+1$ and $q=n+1$, from which 
and the obvious $L^1$ result  
the other estimates of Theorem \ref{thm1.3} follow 
 by interpolation. It is known that the following conditions are 
necessary. 
\begin{enumerate} 
\item $np^{-1} - (n-k)r^{-1}= k$ in Theorems \ref{thm1.3} and \ref{thm1.4}; 
\item $q\leq (n-k)p'$ in Theorems \ref{thm1.3} and \ref{thm1.4}; 
\item $p<n/k$ in Theorem \ref{thm1.4}.   
\end{enumerate}  
Here we give proofs for the necessity. 
\begin{proof}[Proof for part $(1)$]   Let $f^{(\delta)}(x)=f(\delta x)$ for 
$\delta>0$.  Then $T_{k,n}f^{(\delta)}(x, \theta)
=\delta^{-k}T_{k,n}f(\delta x, \theta)$,  and so  
$\|T_{k,n}f^{(\delta)}\|_{L^{q}(L^{r})}=\delta^{-k}\delta^{-(n-k)/r}
\|T_{k,n}f\|_{L^{q}(L^{r})}$. On the other hand,  $\|f^{(\delta)}\|_p=
\delta^{-n/p}\|f\|_p$. Thus, if $\|T_{k,n}f^{(\delta)}\|_{L^{q}(L^{r})}
\lesssim \|f^{(\delta)}\|_p$, we have 
$$\delta^{-k}\delta^{-(n-k)/r}\|T_{k,n}f\|_{L^{q}(L^{r})}\lesssim 
\delta^{-n/p}\|f\|_p $$   
for all $\delta>0$, which implies $np^{-1} - (n-k)r^{-1}= k$. 
\end{proof} 
\begin{proof}[Proof for part $(2)$]  
We write $x=(x^{(1)}, x^{(2)})$, $x^{(1)}\in \Bbb R^k$, $x^{(2)}\in 
\Bbb R^{n-k}$ for $x\in \Bbb R^n$. Define 
$$B^{(1)}=\{x^{(1)}\in \Bbb R^k: |x^{(1)}|\leq 1\}, \quad 
B^{(2)}_\epsilon =\{x^{(2)}\in \Bbb R^{n-k}: |x^{(2)}|\leq \epsilon\},  $$ 
where $\epsilon$ is a small positive number less than $1$.  
Let $T_\epsilon=B^{(1)}\times 
B^{(2)}_\epsilon$. Following \cite{Ch}, we consider the characteristic function  $$\chi_{T_\epsilon}(x^{(1)}, x^{(2)})= \chi_{B^{(1)}}(x^{(1)})
\chi_{B^{(2)}_\epsilon}(x^{(2)}).  $$
Let 
$$P_0=\{(x^{(1)},0): x^{(1)}\in \Bbb R^k\}. $$  
\par 
We can see that there exists an $\epsilon$-neighborhood $N_\epsilon$ of $P_0$ in $G_{k,n}$, which is a $k(n-k)$ dimensional manifold, such that  
if $\theta \in N_\epsilon$, $y\in \theta$ and $|y|\leq 1/2$, then  
$|y^{(1)}|\leq 1/2$, $y^{(2)}\in B^{(2)}_{\epsilon/2}$ and 
$\sigma(N_\epsilon) \gtrsim \epsilon^{k(n-k)}$   
(see \eqref{sur2.1} below in Section \ref{sec2} for the dimensionality of 
$G_{k,n}$; also, the formula in \eqref{eee2.3} below is helpful).  
For any $\theta\in G_{k,n}$, it is obvious that if $x\in \theta^\perp$ and $|x|\leq \epsilon/2$, 
then $|x^{(1)}|\leq 1/2$ and $x^{(2)}\in B^{(2)}_{\epsilon/2}$.  
Therefore, if $\theta \in N_\epsilon$, $y\in \theta$,  $|y|\leq 1/2$, 
$x\in \theta^\perp$ and $|x|\leq \epsilon/2$, then 
$x^{(1)}+ y^{(1)} \in B^{(1)}$ and $x^{(2)}+ y^{(2)} \in B^{(2)}_\epsilon$,  
and so $x+y \in T_\epsilon$.  Thus  
$$ \int \left|\int \chi_{T_\epsilon}(x+y)\, d\lambda_{\theta}(y)\right|^r \, 
d\lambda_{\theta^\perp}(x) \geq  
\int_{|x|\leq \epsilon/2} \left|\int_{|y|\leq 1/2} \, d\lambda_{\theta}(y)
\right|^r \, d\lambda_{\theta^\perp}(x)
\gtrsim \epsilon^{(n-k)}$$  
for all $\theta \in N_\epsilon$, and hence 
\begin{align}\label{sur1}
\|T_{k,n}\chi_{T_\epsilon}\|_{L^{q}(L^{r})}&\geq  
\left(\int_{N_\epsilon} \left(\left|\int \chi_{T_\epsilon}(x+y)\, 
d\lambda_{\theta}(y)\right|^r \, d\lambda_{\theta^\perp}(x)\right)^{q/r} 
\, d\sigma(\theta)\right)^{1/q} 
\\ 
&\gtrsim   \epsilon^{(n-k)/r} \sigma(N_\epsilon)^{1/q} 
\gtrsim   \epsilon^{(n-k)/r} \epsilon^{k(n-k)/q}.     \notag 
\end{align} 
On the other hand, we see that 
\begin{equation}\label{sur2}
\|\chi_{T_\epsilon}\|_p \sim \epsilon^{(n-k)/p}. 
\end{equation} 
Thus if 
$\|T_{k,n}\chi_{T_\epsilon}\|_{L^{q}(L^{r})} \lesssim \|\chi_{T_\epsilon}\|_p $, by \eqref{sur1} and \eqref{sur2} the quantity 
$$\epsilon^{\frac{n-k}{r} + \frac{k(n-k)}{q}-\frac{n-k}{p}}
=\epsilon^{\frac{k(n-k)}{q}-\frac{k}{p'}} $$
remains bounded as $\epsilon \to 0$, where we have used the equation in 
part (1). This implies that $k(n-k)q^{-1}-k(p')^{-1}\geq 0$, which is 
equivalent to what we need. 
\end{proof} 
\begin{proof}[Proof for part $(3)$]  
Let 
$$f(x)=(1+|x|)^{-k}\left(\log(2+|x|)\right)^{-\delta},  $$ 
where $1/p \leq k/n < \delta<1$.  Then, $f\in L^p(\Bbb R^n)$, since 
$kp\geq n$ and $p\delta>1$. Also, we have $T_{k,n}f(x,\theta)=\infty$ for 
all $x, \theta$, since $\delta <1$. 
So we need to have $p<n/k$ if $T_{k,n}$ is bounded from 
$L^p$ to the space $L^q(L^r)$.  

\end{proof} 
\par 
We shall prove Theorem \ref{the1} in Section \ref{sec3} 
assuming Theorem \ref{the2}, 
which will be shown in Section \ref{sec4}. 
In Section \ref{sec2}, we give some lemmas for the proofs of the theorems 
including a formula in Lemma \ref{L2.3} related to a result in \cite{D}. 
Theorems \ref{the1} and \ref{the2} follow from multilinear estimates in 
Propositions \ref{prop1.3} and \ref{p4.1} below, respectively. 
Results in interpolation arguments needed to prove Propositions \ref{prop1.3} 
and \ref{p4.1} will be provided in Sections \ref{sec5}--\ref{sec9}.  
In Sections \ref{j8intro}--\ref{j8cal}, 
we shall prove under a certain condition  
the existence of an invariant measure in a homogeneous space of a Lie group.  
An application of this to the Grassmann manifolds will be given in 
Section \ref{sec14}.

\section{Some results related to the proofs of Theorems \ref{the1} 
and \ref{the2} }\label{sec2}

Let $v_1, \dots, v_k$ be vectors in $\Bbb R^n$. The Gram matrix 
$G_0(v_1, v_2, \dots, v_k)$ is defined to be the $k\times k$ matrix whose 
$(\ell,m)$ component is given by $\langle v_\ell, v_m \rangle:$ 
$$G_0(v_1, v_2, \dots, v_k)=
\begin{pmatrix} 
\langle v_1, v_1\rangle & \dots       &  \langle v_1, v_k\rangle \\
\hdotsfor{3}      \\ 
\hdotsfor{3}      \\ 
\langle v_k, v_1\rangle & \dots  &    \langle v_k, v_k\rangle
\end{pmatrix}. 
$$    
Also, the Gram determinant $G(v_1, v_2, \dots, v_k)$ is defined as 
$$G(v_1,v_2 \dots, v_k)=\det G_0(v_1,v_2 \dots, v_k).   $$  
Then, $G\geq 0$ and $G^{1/2}$ is the $k$-dimensional Lebesgue measure of 
the parallelepiped determined by $v_1, \dots , v_k$ 
(see \cite[1.4, 1.5]{KP}). 
It is known that 
$$G(v_1, v_2, \dots, v_k)=G(v_{\tau(1)}, v_{\tau(2)}, \dots, v_{\tau(k)})
$$ 
for every permutation $\tau$ of $\{1,2,\dots, k\}$ and 
$$G(\alpha_1v_1, \alpha_2v_2, \dots, \alpha_kv_k)
=\alpha_1^2\alpha_2^2\dots \alpha_k^2G(v_{1}, v_{2}, \dots, v_{k})
$$ 
for any $\alpha_1, \alpha_2, \dots , \alpha_k\in \Bbb R$.  Also, 
$$G(Uv_1, Uv_2, \dots, Uv_k)
=G(v_{1}, v_{2}, \dots, v_{k})
$$ 
for every $U\in {\rm O}(n)$ (the orthogonal group).   
If $v_i=(v_{i1}, v_{i2} \dots  , v_{in})$, $1\leq i \leq k$,  
it is known that     
\begin{equation*} 
G(v_1, \dots, v_k)=\sum_{1\leq j_1<j_2< \dots <j_k\leq n}\begin{vmatrix} 
 v_{1,j_1} & v_{1,j_2}  &\dots       &  v_{1,j_k} \\ 
 v_{2,j_1} & v_{2,j_2}  &\dots       &  v_{2,j_k} \\ 
\hdotsfor{4}      \\ 
\hdotsfor{4}      \\ 
 v_{k,j_1} & v_{k,j_2}  &\dots       &  v_{k,j_k} \\ 
\end{vmatrix}^2, 
\end{equation*}
where  the summation is over all $J=\{j_1, j_2, \dots, j_k\}$ in the 
set $\{J: J\subset \{1, 2, \dots, n\}, \card J=k\}$ of cardinality 
$\binom{n}{k}=n!/(k!(n-k)!)$. 
\par 
We recall the following result, which I have learned from Drury 
\cite{D}; a closely related result of Blaschke 
can be found in \cite[Chap. 12]{Sant} (see 
\cite[pp. 371--372]{D2}). 

\begin{lemma}\label{L1.1}  We have  
$$d\lambda_{P(\pi)}(x_0) \dots  d\lambda_{P(\pi)}(x_k)\, d\nu(\pi)=
c|G(x_1-x_0, \dots, x_k-x_0)|^{(k-n)/2} \, dx_0\, dx_1 \dots dx_k   $$ 
with a positive constant $c$,   
where each of 
$dx_0, \dots, dx_k$ is the Lebesgue measure on $\Bbb R^n$, and the equation 
means that 
\begin{multline*} 
\int_{\Bbb R^{n(k+1)}\times \Bbb S}F(x_0, \dots, x_k, x,\theta) \, 
d\lambda_{P(x,\theta)}(x_0)\dots d\lambda_{P(x,\theta)}(x_k)\, 
d\nu(x,\theta)\\ 
=c \int_{\Bbb R^{n(k+1)}} F(x_0, \dots, x_k, 
(x_0)_{\theta(x_1-x_0, \dots, x_k-x_0)^\perp}, \theta(x_1-x_0, 
\dots, x_k-x_0))
\\ 
\times |G(x_1-x_0, \dots, x_k-x_0)|^{(k-n)/2} \, dx_0\, dx_1 \dots dx_k,  
\end{multline*} 
where $F$ is an appropriate function and 
$\theta(v_1, \dots, v_k)$ is the element in $G_{k,n}$ which contains  
$v_1, \dots, v_k$. 
\end{lemma}  
To prove this, we first show the following (see \cite[Theorem 3.4]{So}).  

\begin{lemma} \label{L1.2}
Let $\pi=(y,\theta)\in \Bbb S$, $x\in \Bbb R^n$. Define a measure 
$d\mu_x(\pi)$ on $\Bbb S$  by 
$$\int_{\Bbb S} F(\pi)\, d\mu_x(\pi)= \int_{\Bbb S} F(y,\theta)\, 
d\mu_x(y,\theta) =\int_{G_{k,n}} F(x_{\theta^\perp}, \theta)
\, d\sigma(\theta).  $$  
Then $ d\mu_x(\pi)\, dx=d\lambda_{P(\pi)}(x)\, d\nu(\pi)$ on 
$\Bbb R^n\times \Bbb S$, in the sense that 
 $$\int_{\Bbb S}\int_{\Bbb R^n} F(x,y,\theta)\, d\lambda_{P(y,\theta)}(x)\, 
d\nu(y,\theta)= 
\int_{\Bbb R^n}\int_{\Bbb S} F(x,y,\theta) \, d\mu_x(y,\theta)\, dx.  $$
\end{lemma}
\begin{proof}  Since $y=(y+z)_{\theta^\perp}$ if $y\in \theta^\perp$ and 
$z\in \theta$,  we have 
\begin{align*}  
&\int_{\Bbb R^n\times \Bbb S} F(x,y,\theta)\, d\lambda_{P(y,\theta)}(x)\, 
d\nu(y,\theta)
= \int_{\Bbb S}\int_{z\in \theta} F(y+z, y,\theta)\, d\lambda_\theta(z)
\, d\nu(y,\theta) 
\\ 
&=\int_{G_{k,n}}\int_{y\in  \theta^\perp}\int_{z\in \theta} 
F(y+z, (y+z)_{\theta^\perp}, \theta)\, d\lambda_\theta(z)
\, d\lambda_{\theta^\perp}(y)\, d\sigma(\theta) 
\\ 
&=\int_{G_{k,n}}\int_{\Bbb R^n} F(x, x_{\theta^\perp},\theta)
\ dx\, d\sigma(\theta)  
\\ 
&=\int_{\Bbb R^n}\int_{G_{k,n}} F(x,x_{\theta^\perp}, \theta) 
\, d\sigma(\theta)\, dx 
\\ 
&=\int_{\Bbb R^n}\int_{\Bbb S} F(x,y,\theta) \, d\mu_x(y,\theta)\, dx,   
\end{align*}
where the third equality holds since the Lebesgue measure $dx$ on $\Bbb R^n$ 
can be decomposed as $dx=d\lambda_\theta(z)\, d\lambda_{\theta^\perp}(y)$ 
for every $\theta\in G_{k,n}$: 
$$\int f(z+y)\, d\lambda_\theta(z)\, d\lambda_{\theta^\perp}(y)
=\int_{\Bbb R^n} f(x)\, dx.$$ 
\end{proof} 

\begin{proof}[Proof of Lemma $\ref{L1.1}$]  
Let $(x,\theta)\in \Bbb S$. 
\begin{align*}  
I&:=\int F(x_0,x_1,\dots, x_k,x,\theta) d\lambda_{P(x,\theta)}(x_1)
\dots \, d\lambda_{P(x,\theta)}(x_k)
\\ 
&=\int_{y_1\in \theta}\dots  \int_{y_k\in \theta} 
 F(x_0,x+ y_1,\dots, x+y_k,x,\theta)\, 
 d\lambda_\theta(y_1)\dots \, d\lambda_\theta(y_k) 
\end{align*}  
Therefore, by Lemma \ref{L1.2}, we have 
\begin{align*}  
&\int I\, d\lambda_{P(x,\theta)}(x_0)\, d\nu(x, \theta) =\int I
\, d\mu_{x_0}(x,\theta)\, dx_0 
\\ 
&=\int F(x_0, (x_0)_{\theta^\perp}+y_1, \dots,  (x_0)_{\theta^\perp}+y_k, 
 (x_0)_{\theta^\perp}, \theta)
\\ 
&\qquad\qquad\qquad\qquad\qquad\qquad\qquad \qquad\qquad\,  
 d\lambda_\theta(y_1)\dots \, d\lambda_\theta(y_k) \, d\sigma(\theta) dx_0 
\\ 
&=\int F(x_0, x_0+y_1, \dots,  x_0+y_k, (x_0)_{\theta^\perp}, \theta)\,  
 d\lambda_\theta(y_1)\dots \, d\lambda_\theta(y_k) \, d\sigma(\theta) dx_0
 \\ 
&=:I_1, 
\end{align*}  
where the penultimate equality follows from \eqref{ee1.4}.   
\par 
To complete the proof of Lemma \ref{L1.1}, we apply the following result.

\begin{lemma}\label{L2.3} Let $1\leq k< n$.  Then we have 
\begin{multline*} 
\int_{G_{k,n}}\int_{\theta^k} f(y_1, y_2, \dots, y_k, \theta)
\, d\lambda_{\theta}(y_1)
\, d\lambda_{\theta}(y_2)\dots \lambda_{\theta}(y_k)
\, d\sigma(\theta) 
\\ 
= c \int_{\Bbb R^{nk}}f(x_1, x_2, \dots, x_k, \theta(x_1, x_2, \dots , x_k)
)\left|G(x_1, x_2, \dots, x_k)\right|^{(k-n)/2} \, dx_1\, dx_2 
\dots dx_k.  
\end{multline*}
\end{lemma} 
By Lemma \ref{L2.3}, 
\begin{align*}  
I_1&=
c\int_{\Bbb R^{n(k+1)}} F(x_0, x_0+x_1, \dots,  x_0+x_k, 
 (x_0)_{\theta(x_1,\dots, x_k)^\perp}, \theta(x_1,\dots ,x_k)) 
 \\ 
 &\quad \times  \left|G(x_1, x_2, \dots, x_k)\right|^{(k-n)/2}
 \,  dx_1\dots \, dx_k\, dx_0  
 \\ 
&= c\int_{\Bbb R^{n(k+1)}} F(x_0, x_1, \dots, x_k, 
 (x_0)_{\theta(x_1-x_0,\dots, x_k-x_0)^\perp}, \theta(x_1-x_0,\dots ,x_k-x_0)) 
 \\ 
 &\quad \times  \left|G(x_1-x_0, x_2-x_0, \dots, x_k-x_0)
 \right|^{(k-n)/2}
 \, dx_0 \,  dx_1\dots \, dx_k,  
\end{align*}  
which completes the proof. 
\end{proof} 
\begin{proof}[Proof of Lemma $\ref{L2.3}$]  
Here we consider  the Grassmann manifold $G_{k,n}$ defined as in 
\cite[pp. 63--65]{Bo} (see also \cite[Section 2 of Chapter 13]{Sp}).  
The dimension of $G_{k,n}$ is $k(n-k)$ and the invariant measure on $G_{k,n}$ 
is realized as a measure based on 
 the absolute value of a $k(n-k)$ local differential form on 
$G_{k,n}$ (see Remark \ref{re14.2} below).  
By a suitable partition of unity we can write $d\sigma=\sum d\mu_j$, 
where $d\mu_j$ can be expressed by using local coordinates as 
$$d\mu_j(\theta)=\rho_j(x_1(\theta), \dots, x_{k(n-k)}(\theta))\, 
\left|(dx_1)_\theta \wedge \dots \wedge (dx_{k(n-k)})_\theta\right|,  
 $$ 
 where $\rho_j$ is compactly supported, non-negative, continuous function. 
We also write 
$$ d\mu_j= \rho_j(x_1, \dots, x_{k(n-k)})\, 
dx_1 \dots dx_{k(n-k)}. $$
 We rewrite this using $a_{1,k+1}, \dots, a_{k,k+1}, 
  a_{1,k+2}, \dots, a_{k,k+2}, \dots, a_{1,n}, \dots, a_{k,n}$ as  
$$ d\mu_j= \rho_j(a_{1,k+1}, \dots, a_{k,n})\, 
da_{1,k+1}\dots  da_{k,n}.$$  
Fix $j_0$ and consider $d\mu_{j_0}$.  
Let 
\begin{align*} 
v_1&=(1,0, \dots , 0,a_{1,k+1}, \dots, a_{1,n}), 
\\ 
v_2&=(0,1, \dots , 0,a_{2,k+1}, \dots, a_{2,n}), 
\\ 
&\, \, \,  \vdots 
\\ 
v_k&=(0, \dots , 0,1,a_{k,k+1}, \dots, a_{k,n}).   
\end{align*}  
We now assume that $\Bbb R^{k(n-k)}$ and  the local coordinates on $G_{k,n}$  
are related by the injection 
\begin{multline}\label{sur2.1} 
(a_{1,k+1}, \dots, a_{k,k+1}, 
  a_{1,k+2}, \dots, a_{k,k+2}, \dots, a_{1,n}, \dots, a_{k,n}) 
\\ 
  \longmapsto 
  \theta=\mathop{\mathrm{sp}}\{v_1, v_2, \dots, v_k\}\in G_{k,n},   
\end{multline} 
where 
$\mathop{\mathrm{sp}}\{v_1, v_2, \dots, v_k\}$ denotes the subspace of  
$\Bbb R^n$ generated by $\{v_1, v_2, \dots, v_k\}$.   
Put 
$$y_\ell=\sum_{i=1}^k s_i^\ell v_i=  
\left(s_1^\ell, s_2^\ell, \dots, s_k^\ell, \sum_{i=1}^k s_i^\ell a_{i, k+1}, 
\dots, \sum_{i=1}^k s_i^\ell a_{i, n}\right) $$
for $\ell=1, 2, \dots, k$.  
Let 
$$E_k=\begin{pmatrix} 
1 &  &        & 0 \\
& 1  &        &    &       \\ 
& & \ddots    &    &       \\ 
 0 &   &      &    1
\end{pmatrix}
 $$  
 be the $k\times k$ unit matrix and define a $k\times (n-k)$ matrix 
 $A_{k,n-k}$ by 
$$A_{k,n-k}=\begin{pmatrix} 
a_{1,k+1} & \dots       & a_{1,n} \\
\hdotsfor{3}      \\ 
\hdotsfor{3}      \\ 
 a_{k,k+1} & \dots  &       a_{k,n}
\end{pmatrix}. 
 $$
We note that 
$$\begin{pmatrix} 
v_1 \\
v_2 \\ 
\vdots \\ 
v_k
\end{pmatrix} 
= (E_k, A_{k,n-k}). 
 $$  
\par 
Let 
$$u^\ell=(s_1^\ell, s_2^\ell, \dots, s^\ell_k) $$  
be the $k$ dimensional row vector and 
$$0_k=(0,0, \dots, 0) $$ 
be the $k$ dimensional row zero vector. 
We define $n-k$ numbers of $k(n-k)$ row vectors by 
\begin{align*} 
S^\ell_1&=(u^\ell, 0_k, \dots, 0_k), 
\\ 
S^\ell_2&=(0_k, u^\ell, 0_k, \dots, 0_k), 
\\ 
&\vdots 
\\ 
S^\ell_{n-k}&=(0_k, 0_k, \dots, 0_k, u^\ell). 
\end{align*}
Define the $(n-k)\times k(n-k)$ matrix by 
$$B^\ell=\begin{pmatrix} 
S^\ell_1 \\
S^\ell_2 \\ 
\vdots \\ 
S^\ell_{(n-k)}  
\end{pmatrix}  
 $$  
for $\ell =1, 2, \dots, k$. 
Let 
$$a^j= (a_{1,k+j},\dots,  a_{k,k+j}), \quad j=1,2, \dots, n-k. $$ 
Consider the Jacobian 
\begin{equation}\label{eeee2.1}  
A= \frac{\partial(y_1, y_2, \dots, y_k)}
{\partial(u^1, u^2, \dots, u^k, a^1, a^2, \dots, a^{n-k})}.   
\end{equation}
Then $A$ is equal to the determinant of the $nk\times nk$ matrix: 
$$\begin{pmatrix} 
E_k & 0_{k,k}  & 0_{k,k} & \hdotsfor{3}& 0_{k,k} & 0_{k,k(n-k)} &  
\\
A_{k, n-k}^t &0_{n-k,k}& 0_{n-k,k} &\hdotsfor{3} & 0_{n-k,k}&B^1
\\ 
0_{k,k} &E_k  & 0_{k,k} & \hdotsfor{3}& 0_{k,k} & 0_{k,k(n-k)}  &  
\\
 0_{n-k,k}&A_{k, n-k}^t&0_{n-k,k}  &\hdotsfor{3} & 0_{n-k,k}&B^2
\\
\vdots &  \vdots &            &  &    & \ddots      &     &  \vdots \\ 
0_{k,k} & 0_{k,k}& 0_{k,k}   &   \hdotsfor{3}& E_k &0_{k,k( n-k)} \\ 
0_{n-k,k}& 0_{n-k,k}& 0_{n-k,k} & \hdotsfor{3}&
 A_{k,n-k}^t & B^k  
\end{pmatrix}, 
 $$ 
where $0_{j,m}$ denotes the $j\times m$ zero matrix and $A_{k,n-k}^t$ 
denotes the transpose.  We can see easily that $|A|$ is equal to the absolute 
value of the determinant of the matrix 

$$\begin{pmatrix} 
E_k & & &  &  
\\ 
 &E_k & &  &  
\\ 
    &  &\ddots & & 
\\ 
&          &  & E_k  & 
\\  
A_{k,n-k}^t & & & & B^1  
\\ 
 & A_{k,n-k}^t& & & B^2   
\\ 
   &   &\ddots & & \vdots   
\\ 
&           & & A_{k,n-k}^t & B^k  
\\  
\end{pmatrix}  
$$ 
(where the components not expressed explicitly are all $0$),   
which equals the absolute value of the determinant of the $k(n-k)\times k(n-k)$ matrix: 
$$B=\begin{pmatrix} 
B^1 \\
B^2 \\ 
\vdots \\ 
B^k  
\end{pmatrix}.  
 $$  
By inspection $|\det B|$ is equal to $|\det S|^{n-k}$, where    
\begin{equation}\label{eee2.1} 
S= \begin{pmatrix} 
 s^1_{1} & s^1_{2}  &\dots       &  s^1_{k} \\ 
 s^2_{1} & s^2_{2}  &\dots       &  s^2_{k} \\ 
\hdotsfor{4}      \\ 
\hdotsfor{4}      \\ 
 s^k_{1} & s^k_{2}  &\dots       &  s^k_{k} 
\end{pmatrix}.  
\end{equation}
\par 
We write 
$$y_\ell=\left(\sum_{i=1}^k s^\ell_i a_{i1}, \sum_{i=1}^k s^\ell_i a_{i2}, 
\dots, \sum_{i=1}^k s^\ell_i a_{in}  \right), $$  
where $a_{ij}= \delta_{ij}$ for $1\leq j\leq k$; $\delta_{ii}=1$, 
$\delta_{ij}=0$ if $i\neq j$.
\par 
Let $1\leq j_1<j_2< \dots <j_k\leq n$. To compute $G(y_1, \dots, y_k)$, 
we note that 
\begin{align*}
&\begin{vmatrix} 
\sum_{i_1=1}^k s^1_{i_1}a_{i_1,j_1} &\sum_{i_2=1}^k s^1_{i_2}a_{i_2,j_2} 
 &\dots       & \sum_{i_k=1}^k s^1_{i_k}a_{i_k,j_k} \\ 
\sum_{i_1=1}^k s^2_{i_1}a_{i_1,j_1} &\sum_{i_2=1}^k s^2_{i_2}a_{i_2,j_2} 
 &\dots       & \sum_{i_k=1}^k s^2_{i_k}a_{i_k,j_k} \\
\hdotsfor{4}      \\ 
\hdotsfor{4}      \\ 
\sum_{i_1=1}^k s^k_{i_1}a_{i_1,j_1} &\sum_{i_2=1}^k s^k_{i_2}a_{i_2,j_2} 
 &\dots       & \sum_{i_k=1}^k s^k_{i_k}a_{i_k,j_k} 
\end{vmatrix}
\\ 
&= \sum_{i_1, i_2, \dots, i_k=1}^ka_{i_1,j_1}a_{i_2,j_2}\dots a_{i_k,j_k} 
\begin{vmatrix} 
 s^1_{i_1} & s^1_{i_2}  &\dots       &  s^1_{i_k} \\ 
 s^2_{i_1} & s^2_{i_2}  &\dots       &  s^2_{i_k} \\ 
\hdotsfor{4}      \\ 
\hdotsfor{4}      \\ 
 s^k_{i_1} & s^k_{i_2}  &\dots       &  s^k_{i_k} 
\end{vmatrix} 
\\ 
&= \begin{vmatrix} 
 a_{1,j_1} & a_{1,j_2}  &\dots       &  a_{1,j_k} \\ 
 a_{2,j_1} & a_{2,j_2}  &\dots       &  a_{2,j_k} \\ 
\hdotsfor{4}      \\ 
\hdotsfor{4}      \\ 
 a_{k,j_1} & a_{k,j_2}  &\dots       &  a_{k,j_k} \\ 
\end{vmatrix}
\begin{vmatrix} 
 s^1_{1} & s^1_{2}  &\dots       &  s^1_{k} \\ 
 s^2_{1} & s^2_{2}  &\dots       &  s^2_{k} \\ 
\hdotsfor{4}      \\ 
\hdotsfor{4}      \\ 
 s^k_{1} & s^k_{2}  &\dots       &  s^k_{k} 
\end{vmatrix}.  
\end{align*}  
Thus 
\begin{align*} 
G(y_1, \dots, y_k)&=\sum_{1\leq j_1<j_2< \dots <j_k\leq n}\begin{vmatrix} 
 a_{1,j_1} & a_{1,j_2}  &\dots       &  a_{1,j_k} \\ 
 a_{2,j_1} & a_{2,j_2}  &\dots       &  a_{2,j_k} \\ 
\hdotsfor{4}      \\ 
\hdotsfor{4}      \\ 
 a_{k,j_1} & a_{k,j_2}  &\dots       &  a_{k,j_k} \\ 
\end{vmatrix}^2
\begin{vmatrix} 
 s^1_{1} & s^1_{2}  &\dots       &  s^1_{k} \\ 
 s^2_{1} & s^2_{2}  &\dots       &  s^2_{k} \\ 
\hdotsfor{4}      \\ 
\hdotsfor{4}      \\ 
 s^k_{1} & s^k_{2}  &\dots       &  s^k_{k} 
\end{vmatrix}^2 
\\ 
&= G(v_1, \dots v_k)\left|\det S\right|^2. 
\end{align*}
Therefore 
\begin{align*} 
&ds^1_1\dots ds^1_k\dots ds_1^k\dots ds_k^kd\mu_{j_0} 
\\ 
&= ds^1_1\dots ds^1_k\dots ds_1^k\dots ds_k^k 
\rho_{j_0}(a_{1,k+1}, \dots, a_{k,n})\, da_{1,k+1} \dots  da_{k,n} 
\\ 
&= \rho_{j_0}(a_{1,k+1}, \dots, a_{k,n})
\left|\det S\right|^{k-n}\, dy_1\, \dots \, dy_k 
\\ 
&= \rho_{j_0}(a_{1,k+1}, \dots, a_{k,n}) N(a_{ij})^{n-k} 
\left|G(y_1, \dots, y_k) \right|^{(k-n)/2} \, dy_1\, \dots \, dy_k, 
\end{align*}
where 
$$N(a_{ij})=G(v_1, \dots v_k)^{1/2}.  $$
\par 
Let $y_\ell=\sum_{i=1}^k s_i^\ell v_i$, where $v_1, \dots, v_k$ are as above. 
It can be shown that 
\begin{align*}  
&\int_\theta \dots \int_\theta g(z_1, z_2, \dots , z_k)\, 
d\lambda_\theta(z_1)\, d\lambda_\theta(z_2) \dots \, d\lambda_\theta(z_k) 
\\ 
&=\left|G(v_1, v_2, \dots, v_k)\right|^{k/2}
\int_{\Bbb R^k} \dots \int_{\Bbb R^k}
g(y_1, \dots, y_k)\, \prod_{\ell=1}^k 
ds_1^\ell \dots \, ds_k^\ell  
\\ 
&=N(a_{ij})^{k} 
\int_{\Bbb R^k} \dots \int_{\Bbb R^k}
g(y_1, \dots, y_k)\, \prod_{\ell=1}^k 
ds_1^\ell \dots \, ds_k^\ell   
\end{align*}
for an appropriate function $g$, where $\theta\in G_{k,n}$ is spanned by 
$v_1, v_2, \dots, v_k$.  
Therefore  

\begin{align}\label{e1} 
&\int \left(\int\dots \int f(z_1, \dots, z_k, \theta)\, d\lambda_\theta(z_1)\, 
 \dots \, d\lambda_\theta(z_k)\right)
\\ 
& \qquad \qquad \notag 
\rho_{j_0}(x_1(\theta), \dots, 
x_{k(n-k)}(\theta)) \, \left|(dx_1)_\theta \wedge \dots \wedge 
(dx_{k(n-k)})_\theta\right| 
\\ 
&=\int \int_{\theta(a_{1,k+1}, \dots, a_{k,n})^k} f(z_1, \dots, z_k, \theta(a_{1,k+1}, \dots, a_{k,n})) \notag 
\\ 
& \qquad \qquad \notag 
\, \prod_{j=1}^k d\lambda_{\theta(a_{1,k+1}, \dots, a_{k,n})}(z_j) 
\rho_{j_0}(a_{1,k+1}, \dots, a_{k,n})\, da_{1,k+1}\,  \dots \, da_{k,n}
\\ 
&=\int N(a_{ij})^{k} 
\int_{\Bbb R^k} \dots \int_{\Bbb R^k}
f(y_1, \dots, y_k, \theta(a_{1,k+1}, \dots, a_{k,n}))\, \prod_{\ell=1}^k 
ds_1^\ell \dots \, ds_k^\ell    \notag 
\\
& \quad\quad \times \rho_{j_0}(a_{1,k+1}, \dots, a_{k,n})\, da_{1,k+1}\,  
\dots \, da_{k,n}  \notag 
\\ 
&=\int_{\Bbb R^n} \dots \int_{\Bbb R^n}
f(y_1, \dots, y_k, \theta(a_{1,k+1}, \dots, a_{k,n}))  \notag 
\\ 
&\quad\quad \times \rho_{j_0}(a_{1,k+1}, \dots, a_{k,n}) N(a_{ij})^n
\left|G(y_1, \dots, y_k) \right|^{(k-n)/2} \, dy_1\, \dots \, dy_k,   
\notag 
\end{align}  
where $\theta(a_{1,k+1}, \dots, a_{k,n})=\theta(v_1, \dots , v_k)$. 
\par 
Here we note the following. Let $S$ be as in \eqref{eee2.1}. 
Then 
\begin{equation}\label{eee2.3}
[y] := 
\begin{pmatrix} 
y_1 \\
y_2 \\ 
\vdots \\ 
y_k
\end{pmatrix} 
= (S, SA_{k,n-k})= S(E_k, A_{k,n-k}). 
\end{equation} 
Thus 
$$ 
S=S(y_1, \dots, y_k)= \begin{pmatrix} 
 y^1_{1} & y^1_{2}  &\dots       &  y^1_{k} \\ 
 y^2_{1} & y^2_{2}  &\dots       &  y^2_{k} \\ 
\hdotsfor{4}      \\ 
\hdotsfor{4}      \\ 
 y^k_{1} & y^k_{2}  &\dots       &  y^k_{k} 
\end{pmatrix}, 
$$ 
$$ 
T=T(y_1, \dots, y_k):= SA_{k,n-k}= \begin{pmatrix} 
 y^1_{k+1} & y^1_{k+2}  &\dots       &  y^1_{n} \\ 
 y^2_{k+1} & y^2_{k+2}  &\dots       &  y^2_{n} \\ 
\hdotsfor{4}      \\ 
\hdotsfor{4}      \\ 
 y^k_{k+1} & y^k_{k+2}  &\dots       &  y^k_{n} 
\end{pmatrix}, 
$$ 
and hence 
$$ A_{k,n-k}=\begin{pmatrix} 
a_{1,k+1} & \dots       & a_{1,n} \\
\hdotsfor{3}      \\ 
\hdotsfor{3}      \\ 
 a_{k,k+1} & \dots  &       a_{k,n}
\end{pmatrix}=S^{-1}T, $$  
where we write $y_\ell=(y^\ell_1, \dots, y^\ell_n)$.  
Therefore $a_{ij}$ can be expressed by $y^\ell_m$: 
\begin{equation*} 
a_{ij}=a_{ij}(y_1, \dots , y_k)= 
\left(S(y_1, \dots, y_k)^{-1}T(y_1, \dots, y_k)\right)_{i,j-k} 
\end{equation*} 
for $1\leq i\leq k$, $k+1\leq j\leq n$.  
  So we can write 
\begin{multline}\label{e2}
\int_{\Bbb R^n} \dots \int_{\Bbb R^n}
f(y_1, \dots, y_k, \theta(a_{1,k+1}, \dots, a_{k,n})) 
\\ 
 \rho_{j_0}(a_{1,k+1}, \dots, a_{k,n}) N(a_{ij})^n
\left|G(y_1, \dots, y_k) \right|^{(k-n)/2} \, dy_1\, \dots \, dy_k  
\\ 
= 
\int_{\Bbb R^n} \dots \int_{\Bbb R^n}
f(y_1, \dots, y_k, \theta(y_{1}, \dots, y_{k})) 
 \widetilde{\rho}_{j_0}(y_{1}, \dots, y_{k}) 
\\ 
\times \left|G(y_1, \dots, y_k) \right|^{(k-n)/2} 
\, dy_1\, \dots \, dy_k,   
\end{multline} 
where 
\begin{multline*}
 \widetilde{\rho}_{j_0}(y_{1}, \dots, y_{k}) 
\\ 
=\rho_{j_0}\left(a_{1,k+1}(y_{1}, \dots, y_{k}), 
\dots, a_{k,n}(y_{1}, \dots, y_{k})\right) 
N\left(a_{ij}(y_{1}, \dots, y_{k})\right)^n   
\end{multline*} 
and we have used the relation:  
$$\theta\left(a_{1,k+1}(y_{1}, \dots, y_{k}), 
\dots, a_{k,n}(y_{1}, \dots, y_{k})\right)=\theta(v_1, \dots , v_k)
=\theta(y_{1}, \dots, y_{k}), 
$$ 
which follows by \eqref{eee2.3}.  
We note that each $a_{ij}$ is homogeneous of degree $0$: 
$$a_{ij}(\tau y_{1}, \dots, \tau y_{k})=a_{ij}(y_{1}, \dots, y_{k})
\qquad \text{for all $\tau>0$.}$$
\par 
Let $[y]$ be the $k\times n$ matrix as in \eqref{eee2.3}. 
We also write $S(y_1, \dots, y_k)=S([y])$, $T(y_1, \dots, y_k)=T([y])$, 
$a_{ij}(y_1, \dots, y_k)=a_{ij}([y])$ and 
$\widetilde{\rho}_{j_0}(y_{1}, \dots, y_{k})=\widetilde{\rho}_{j_0}([y])$. 
Then if $\alpha$ is a $k\times k$ matrix, 
\begin{equation*} 
S(\alpha[y])=\alpha S([y]), \quad T(\alpha[y])=\alpha T([y]).  
\end{equation*} 
Thus if $\alpha$ is non-singular, we have 
$$S(\alpha[y])^{-1}T(\alpha[y])=(\alpha S([y]))^{-1}\alpha T([y]) 
= S([y])^{-1}\alpha^{-1}\alpha T([y])=  S([y])^{-1}T([y]), $$ 
and hence $a_{ij}(\alpha[y])=a_{ij}([y])$ and 
\begin{equation}\label{e3}  
\widetilde{\rho}_{j_0}(\alpha[y])=\widetilde{\rho}_{j_0}([y]). 
\end{equation}  
Summing up in $j_0$, by \eqref{e1} and \eqref{e2}, we have 
\begin{align} \label{e4}
&\int_{G_{k,n}}\int_{\theta^k} f(z_1, z_2, \dots, z_k, \theta)
\, d\lambda_\theta(z_1)\, d\lambda_\theta(z_2)\dots 
d\lambda_\theta(z_k) \, d\sigma(\theta) 
\\
& = \int_{\Bbb R^n} \dots \int_{\Bbb R^n}
f(y_1, \dots, y_k, \theta(y_{1}, \dots, y_{k})) 
 \rho(y_{1}, \dots, y_{k})                       \notag 
\\ 
&\quad \times \left|G(y_1, \dots, y_k) \right|^{(k-n)/2} 
\, dy_1\, \dots \, dy_k,                         \notag 
\end{align} 
where 
$$\rho= \sum_{j_0} \widetilde{\rho}_{j_0}. 
$$ 
\par 
We can see that $\rho$ is a positive constant function as follows.  
Let $y_1, \dots, y_k$ be linearly independent in $\Bbb R^n$ and 
$O_1, \dots , O_k$ be orthonormal in the space spanned by  $y_1, \dots, y_k$. 
Then there exists a non-singular $k\times k$ matrix $\alpha$ such that 
$$\alpha [y]= \begin{pmatrix} 
O_1 \\
O_2 \\ 
\vdots \\ 
O_k
\end{pmatrix}.    $$
By \eqref{e3}, it follows that  
\begin{equation} \label{e5} 
\rho(y_{1}, \dots, y_{k})=\rho(O_{1}, \dots, O_{k}). 
\end{equation}
Also \eqref{e4} implies that $\rho$ is invariant under the action of 
${\rm SO}(n)$: 
\begin{equation} \label{e6}
\rho(y_{1}, \dots, y_{k})=\rho(U y_{1}, \dots, U y_{k})
\end{equation}
for $U\in {\rm SO}(n)$.  This can be seen as follows. 
 Let 
 $$I=\int_{G_{k,n}}\int_{\theta^k} f(Uz_1, Uz_2, \dots, Uz_k, U\theta)
\, d\lambda_\theta(z_1)\, d\lambda_\theta(z_2)\dots 
d\lambda_\theta(z_k) \, d\sigma(\theta),    
 $$  
where $U\theta= \theta(Ux_1, \dots, Ux_k)$ if 
$\theta=\theta(x_1, \dots, x_k)$. 
 Then  by ${\rm SO}(n)$ invariance of $d\sigma$ and \eqref{e4}, we have 
\begin{align*} 
I&=\int_{G_{k,n}}\int_{(U\theta)^k} f(z_1, z_2, \dots, z_k, U\theta)
\, d\lambda_{U\theta}(z_1)\, d\lambda_{U\theta}(z_2)\dots 
d\lambda_{U\theta}(z_k) \, d\sigma(\theta) 
\\
&=\int_{G_{k,n}}\int_{\theta^k} f(z_1, z_2, \dots, z_k, \theta)
\, d\lambda_{\theta}(z_1)\, d\lambda_{\theta}(z_2)\dots 
d\lambda_{\theta}(z_k) \, d\sigma(\theta) 
\\
& = \int_{\Bbb R^n} \dots \int_{\Bbb R^n}
f(y_1, \dots, y_k, \theta(y_{1}, \dots, y_{k})) 
 \rho(y_{1}, \dots, y_{k})                       \notag 
\\ 
&\quad \quad \times \left|G(y_1, \dots, y_k) \right|^{(k-n)/2} 
\, dy_1\, \dots \, dy_k.                          \notag 
\end{align*} 
On the other hand, using \eqref{e4}, we see that 
\begin{align*} 
I&= \int_{\Bbb R^n} \dots \int_{\Bbb R^n}
f(Uy_1, \dots, Uy_k, U\theta(y_{1}, \dots, y_{k})) 
 \rho(y_{1}, \dots, y_{k})                       \notag 
\\ 
&\quad \quad \times \left|G(y_1, \dots, y_k) \right|^{(k-n)/2} 
\, dy_1\, \dots \, dy_k  
\\
& = \int_{\Bbb R^n} \dots \int_{\Bbb R^n}
f(y_1, \dots, y_k, U\theta(U^{-1}y_{1}, \dots, U^{-1}y_{k})) 
 \rho(U^{-1}y_{1}, \dots, U^{-1}y_{k})                       \notag 
\\ 
&\quad \quad \times \left|G(U^{-1}y_1, \dots, U^{-1}y_k) \right|^{(k-n)/2} 
\, dy_1\, \dots \, dy_k,                         \notag 
\\
& = \int_{\Bbb R^n} \dots \int_{\Bbb R^n}
f(y_1, \dots, y_k, \theta(y_{1}, \dots, y_{k})) 
 \rho(U^{-1}y_{1}, \dots, U^{-1}y_{k})                       \notag 
\\ 
&\quad \quad \times \left|G(y_1, \dots, y_k) \right|^{(k-n)/2} 
\, dy_1\, \dots \, dy_k.                           \notag 
\end{align*} 
Comparing the two expressions of $I$ above, we can see that \eqref{e6} 
holds true. 
\par 
 Rearranging $O_1, \dots , O_k$,  if necessary, 
 we can find $U\in {\rm SO}(n)$ so that $U O_1=e_1$, \dots , 
$U O_k=e_k$, where $\{e_1, \dots, e_n\}$ denotes the standard basis of 
$\Bbb R^n$. Then 
by \eqref{e5} and \eqref{e6} we have 
$$
\rho(y_{1}, \dots, y_{k})=\rho(e_{1}, \dots, e_k). 
$$
This completes the proof of Lemma \ref{L2.3}.

\end{proof}  
\par 
We state a result analogous to Lemma \ref{L2.3}, which will be 
used in proving Theorem \ref{the2}.  
\begin{lemma}\label{L2.6}  We have 
$$d\lambda_{\theta}(\omega_1) \dots  d\lambda_{\theta}(\omega_k)
\, d\sigma(\theta) =
c|\det G(\omega_1, \dots , \omega_k)|^{(k-n)/2} \, d\omega_1 \dots 
d\omega_k,   
 $$  
$1\leq k < n$,  which means \eqref{e2.7} below,  where 
$d\omega_1, \dots, d\omega_k$ are the Lebesgue surface measure on the unit 
sphere $S^{n-1}$ and $d\lambda_{\theta}(\omega_1), \dots  , 
d\lambda_{\theta}(\omega_k)$ are as in \eqref{Skn}.  
\end{lemma}  

\begin{proof}
By Lemma \ref{L2.3}, we have  
\begin{multline*} 
\int_{G_{k,n}}\int_{y_1, \dots , y_k\in \theta} 
F(y_1, \dots, y_k,\theta) 
\, d\lambda_\theta(y_1)\dots \, d\lambda_\theta(y_k)\, d\sigma(\theta)
\\ 
= c\int_{\Bbb R^{nk}}F(x_1, \dots, x_k,\theta(x_1, \dots, x_k)) \,
 |G(x_1, \dots, x_k)|^{(k-n)/2} \, dx_1 \dots dx_k.  
\end{multline*} 
Thus, using polar coordinates, we see that 
\begin{multline*}
\int_{G_{k,n}}\int_{(0,\infty)^k}\int_{(S_\theta^{k-1})^k} 
F(r_1\omega_1, \dots, r_k\omega_k,\theta) (r_1\dots r_k)^{k-1} \, 
\prod_{j=1}^k dr_j\, d\lambda_\theta(\omega_j)\, d\sigma(\theta)
\\ 
= c\int_{(0,\infty)^k}\int_{(S^{n-1})^k} 
F(r_1\omega_1, \dots, r_k\omega_k,\theta(\omega_1, \dots, \omega_k)) 
|G(\omega_1, \dots, \omega_k)|^{(k-n)/2} 
\\ 
\times(r_1\dots r_k)^{k-1} \, dr_1 \dots\, dr_k 
\, d\omega_1 \dots \, d\omega_k.  
\end{multline*} 
Taking a function of the form 
$$H(|x_1|, \dots, |x_k|) F_0(x_1', \dots, x_k', \theta), 
\quad x_j'=x_j/|x_j| \quad (1\leq j\leq k) $$  
as $F$ and factoring out 
$\int_{(0,\infty)^k}H(r_1, \dots, r_k)(r_1\dots r_k)^{k-1}\, dr_1\dots dr_k$, 
we have
\begin{align}\label{e2.7} 
&\int_{G_{k,n}}\int_{(S_\theta^{k-1})^k} 
F_0(\omega_1, \dots, \omega_k,\theta) \, d\lambda_\theta(\omega_1)
\dots \, d\lambda_\theta(\omega_k)\, d\sigma(\theta) 
\\ 
&=c\int_{(S^{n-1})^k} 
F_0(\omega_1, \dots, \omega_k,\theta(\omega_1, \dots, \omega_k)) 
 |G(\omega_1, \dots, \omega_k)|^{(k-n)/2} 
  \, d\omega_1\dots \, d\omega_k.                 \notag 
\end{align}  
This implies what we need. 

\end{proof} 

\par      
\section{Proof of Theorem \ref{the1} }\label{sec3}  
In this section we prove Theorem \ref{the1} assuming Theorem \ref{the2}. 
Let 
$$ A_{k,n}(f_0, f_1, \dots, f_n)=
\int_{\Bbb S}  \left(\prod_{j=0}^n T_{k,n} f_j(x,\theta)\right)\, 
d\nu(x,\theta).$$ 
We note that 
\begin{equation*} 
\int_{\Bbb S}  (T_{k,n}f(x,\theta))^{n+1}\, d\nu(x,\theta)=
A_{k,n}(f, f, \dots, f).
\end{equation*}  
The following result obviously implies Theorem \ref{the1}. 
\begin{proposition} \label{prop1.3}  
Let $k$ be an integer such that  $1\leq k<n$. 
We have 
$$ \left|A_{k,n}(f_0, f_1, \dots, f_n)\right|\leq C \prod_{j=0}^n 
\|f_j\|_{(n+1)/(k+1), n+1}.$$  
In fact, we have more general estimates:  
$$ \left|A_{k,n}(f_0, f_1, \dots, f_n)\right|\leq C(p_0, \dots , p_n) 
\prod_{j=0}^n \|f_j\|_{p_j, n+1}, $$    
where $\sum_{j=0}^n \frac{1}{p_j}=k+1$ and $1< p_0, \dots, p_n <\frac{n}{k}$.  
\end{proposition}
\begin{proof} 
We have 
\begin{multline}  \label{e7}
A_{k,n}(f_0, f_1, \dots, f_n) 
\\ 
=c\int f_0(x_0)\dots f_k(x_k)\left(\prod_{j=k+1}^n 
\int_{\theta(x_1-x_0,\dots, x_k-x_0)} 
f_j(x_0+y)\, d\lambda_{\theta(x_1-x_0,\dots, x_k-x_0)}(y) \right) 
\\ 
\times \left|G(x_1-x_0, \dots, x_k-x_0)\right|^{\frac{k-n}{2}} 
\, dx_0 \dots\, dx_k.    
\end{multline} 
We can see this as follows. Applying Lemma \ref{L1.1}, we have 
\begin{align*}
& A_{k,n}(f_0, f_1, \dots, f_n)
\\ 
&=\int_{\Bbb S} \int\dots \int f_0(x_0)\dots f_k(x_k)
\left(\prod_{j=k+1}^n T_{k,n} f_j(x,\theta)\right)\, 
d\lambda_{P(x,\theta)}(x_0)\dots \, d\lambda_{P(x,\theta)}(x_k)\, 
d\nu(x,\theta) 
\\ 
&= c\int\dots \int f_0(x_0)\dots f_k(x_k) 
\left(\prod_{j=k+1}^n T_{k,n} f_j((x_0)_{\theta(x_1-x_0,\dots,x_k-x_0)^\perp},
\theta(x_1-x_0,\dots,x_k-x_0))\right) 
\\ 
&\quad \times \left|G(x_1-x_0,\dots, x_k-x_0)\right|^{\frac{k-n}{2}}
\, dx_0 \dots \, dx_k.  
\end{align*} 
Combining this with the observation:    
\begin{align*}
&T_{k,n}f_j((x_0)_{\theta(x_1-x_0,\dots,x_k-x_0)^\perp}, 
\theta(x_1-x_0,\dots,x_k-x_0)) 
\\ 
&= \int f_j((x_0)_{\theta(x_1-x_0,\dots,x_k-x_0)^\perp}+y)
\, d\lambda_{\theta(x_1-x_0,\dots, x_k-x_0)}(y)
\\
&= \int f_j(x_0+y)\, d\lambda_{\theta(x_1-x_0,\dots, x_k-x_0)}(y), 
\end{align*}
which follows by \eqref{ee1.4}, we get \eqref{e7}.
\par 
Let $\bar{\omega}_1= \frac{x_1-x_0}{|x_1-x_0|}$. When $k\geq 2$, define  
\begin{multline*} 
\Omega\left(x_0, \frac{x_1-x_0}{|x_1-x_0|}\right)
\\ 
=\int f_2(x_2)\dots f_k(x_k)\left(\prod_{j=k+1}^n \int 
f_j(x_0+y)\, d\lambda_{\theta(x_1-x_0,\dots, x_k-x_0)}(y) \right) 
\\ 
\times \left|G(\bar{\omega}_1, x_2-x_0,  \dots, 
x_k-x_0)\right|^{\frac{k-n}{2}} \, dx_2 \dots\, dx_k;     
\end{multline*} 
when $k=1$, let 
$$\Omega\left(x_0, \frac{x_1-x_0}{|x_1-x_0|}\right)=
\prod_{j=k+1}^n \int 
f_j(x_0+y)\, d\lambda_{\theta(x_1-x_0)}(y).  
$$
Let 
\begin{equation*} 
K(x_0, x_1)= 
|x_1-x_0|^{k-n}\Omega\left(x_0, \frac{x_1-x_0}{|x_1-x_0|}\right).  
\end{equation*}
Then, it is easy to see that 
\begin{equation*} 
A_{k,n}(f_0, f_1, \dots, f_n)=\iint f_0(x_0)f_1(x_1)K(x_0, x_1)
\, dx_0\, dx_1.  
\end{equation*} 
\par 
For $\omega_1\in S^{n-1}$ and $x_0\in \Bbb R^n$, when $k\geq 2$, define   
\begin{multline*} 
\Omega\left(x_0, \omega_1\right)
\\ 
=\int f_2(x_2+x_0)\dots f_k(x_k+x_0)\left(\prod_{j=k+1}^n 
\int_{\theta(\omega_1,x_2, \dots, x_k)} 
f_j(x_j+x_0)\, d\lambda_{\theta(\omega_1,x_2, \dots, x_k)}(x_j) \right) 
\\ 
\times \left|G(\omega_1, x_2,  \dots, x_k)\right|^{\frac{k-n}{2}} 
\, dx_2 \dots\, dx_k;  
\end{multline*} 
when $k=1$, let 
$$\Omega\left(x_0, \omega_1\right)=
  \prod_{j=k+1}^n \int_{\theta(\omega_1)} 
f_j(x_j+x_0)\, d\lambda_{\theta(\omega_1)}(x_j).  
$$
We note that $\Omega\left(x_0, \omega_1\right)
=|x_1-x_0|^{n-k}K(x_0,x_1)$ 
when $\omega_1=\bar{\omega}_1=(x_1-x_0)/|x_1-x_0|$.  
We show that 
\begin{equation} \label{e9}
\sup_{x_0} \|\Omega(x_0,\cdot)\|_{L^{n/(n-k)}(S^{n-1})}\leq C\prod_{j=2}^n 
\|f_j\|_{n/k,1}.  
\end{equation}    
Let $\omega_j=x_j/|x_j|$, $r_j=|x_j|$,   
$\widetilde{f}_j(x_j)=f_j(x_j+x_0)$, $j\geq 2$. Then, since 
$\theta(\omega_1,x_2, \dots, x_k)= \theta(\omega_1,\dots, \omega_k)$, 
$k\geq 2$, using the polar coordinates, we have 
\begin{multline*} 
\int_{\theta(\omega_1,x_2, \dots, x_k)} 
f_j(x_j+x_0)\, d\lambda_{\theta(\omega_1,x_2, \dots, x_k)}(x_j)=
\int_{\theta(\omega_1,\dots, 
\omega_k)} 
\widetilde{f}_j(x_j)\, d\lambda_{\theta(\omega_1,\dots, \omega_k)}(x_j) 
\\ 
=c\int_0^\infty \int_{S^{k-1}_{\theta(\omega_1,\dots, \omega_k)}} 
\widetilde{f}_j(r_j\omega_j)r_j^{k-1}\, 
d\lambda_{\theta(\omega_1,\dots, \omega_k)}(\omega_j)\, dr_j.   
\end{multline*} 
A formula similar to this holds for $k=1$. 
Let 
\begin{equation*} 
F_j(\omega_j)=\int_0^\infty \widetilde{f}_j(r_j\omega_j)r_j^{k-1}\, dr_j 
\end{equation*}
for $2\leq j\leq n$.   Then, if $k\geq 2$, 
\begin{multline*} 
\Omega\left(x_0, \omega_1\right)
\\ 
=c\int \widetilde{f}_2(r_2\omega_2)\dots 
\widetilde{f}_k(r_k\omega_k)\left(\prod_{j=k+1}^n 
\int_{S^{k-1}_{\theta(\omega_1,\dots, \omega_k)}} 
F_j(\omega_j)\, d\lambda_{\theta(\omega_1,\dots, \omega_k)}(\omega_j) \right) 
\\ 
\times (r_2\dots r_k)^{k-1}
\left|G(\omega_1, \omega_2,  \dots, \omega_k)\right|^{\frac{k-n}{2}} 
\, d\omega_2 \dots\, d\omega_k\, dr_2 \dots\, dr_k 
\\ 
=c\int_{(S^{n-1})^{k-1}}
 F_2(\omega_2)\dots F_k(\omega_k)\left(\prod_{j=k+1}^n 
\int_{S^{k-1}_{\theta(\omega_1,\dots, \omega_k)}} 
F_j(\omega_j)\, d\lambda_{\theta(\omega_1,\dots, \omega_k)}(\omega_j) \right) 
\\ 
\times 
\left|G(\omega_1, \omega_2,  \dots, \omega_k)\right|^{\frac{k-n}{2}} 
\, d\omega_2 \dots\, d\omega_k;      
\end{multline*}  
if $k=1$, 
$$\Omega\left(x_0, \omega_1\right)= \prod_{j=k+1}^n 
\int_{S^{k-1}_{\theta(\omega_1)}} 
F_j(\omega_j)\, d\lambda_{\theta(\omega_1)}(\omega_j)
=\prod_{j=2}^n \left(F_j(\omega_1)+F_j(-\omega_1)\right)/2.  
 $$
 By Lemma \ref{L2.6} we see that 

\begin{align*} 
&\int_{S^{n-1}} F_1(\omega_1)\Omega(x_0,\omega_1)\, d\omega_1 
\\ 
&=c\int_{G_{k,n}}\int_{(S^{k-1}_\theta)^{k}}
 F_1(\omega_1)F_2(\omega_2)\dots F_k(\omega_k) 
\\ 
&\qquad \quad \times\left(\prod_{j=k+1}^n \int_{S^{k-1}_\theta} 
F_j(\omega_j)\, d\lambda_{\theta}(\omega_j) \right) 
\, d\lambda_{\theta}(\omega_1)\dots \, d\lambda_{\theta}(\omega_k)\, 
d\sigma(\theta)
\\ 
&=c\int_{G_{k,n}}
\left(\prod_{j=1}^n \int_{S_\theta^{k-1}} 
F_j(\omega_j)\, d\lambda_{\theta}(\omega_j) \right) 
\, d\sigma(\theta)
\\ 
&=c\int_{G_{k,n}}
\left(\prod_{j=1}^n S_{k,n}(F_j)(\theta) \right) 
\, d\sigma(\theta).  
\end{align*}
Thus by H\"{o}lder's inequality we have 
$$\left|\int_{S^{n-1}} F_1(\omega_1)\Omega(x_0,\omega_1)\, d\omega_1\right| 
\leq c\prod_{j=1}^n \left\|S_{k,n}(F_j)\right\|_n.  
$$
So by Theorem \ref{the2} we have 
\begin{align*} 
\left|\int_{S^{n-1}} F_1(\omega_1)\Omega(x_0,\omega_1)\, d\omega_1 \right| 
&\leq C  \prod_{j=1}^n \|F_j\|_{n/k, n}
\\ 
& \leq   C\prod_{j=1}^n \|F_j\|_{n/k}
\\ 
& \leq  C \|F_1\|_{n/k} \prod_{j=2}^n \|\widetilde{f}_j\|_{n/k, 1}
\\ 
& =C\|F_1\|_{n/k} \prod_{j=2}^n \|f_j\|_{n/k, 1}.  
\end{align*} 
where the last inequality follows from Lemma \ref{L1.5} below.   
This proves \eqref{e9} by the converse of H\"{o}lder's inequality. 

\par 
Recall that 
\begin{align*} 
 A_{k,n}(f_0, f_1, \dots, f_n)&=\int f_0(u)f_1(v) K(u,v) \, du\, dv, 
\\  
 K(u,v)&=|v-u|^{k-n}\Omega\left(u,\frac{v-u}{|v-u|} \right).  
\end{align*} 
By \eqref{e9}  we can show that  
\begin{equation}\label{eq1}   
\sup_u\|K(u,\cdot)\|_{q,\infty}  \leq C\prod_{j=2}^n \|f_j\|_{n/k,1}, 
\quad q^{-1}=1- kn^{-1} 
\end{equation}   
as follows. 
\begin{align} \label{e3.5} 
&\left|\left\{v\in \Bbb R^n : |u-v|^{-n+k}
\left|\Omega\left(u, \frac{v-u}{|v-u|}\right)\right|>\lambda \right\}\right| 
\\ 
&=\left|\left\{v\in \Bbb R^n: |v|^{-n+k}    \notag 
\left|\Omega\left(u, \frac{v}{|v|}\right)\right|>\lambda \right\}\right| 
\\ 
&=\int_{\Bbb R^n}\chi_{[1,\infty)}\left(\lambda^{-1}|v|^{-n+k}
\left|\Omega\left(u, \frac{v}{|v|}\right)\right|
\right)\, dv                                           \notag 
\\ 
&=\int_{S^{n-1}}\int_0^\infty 
\chi_{[1,\infty)}\left(\lambda^{-1}r^{-n+k}|\Omega(u,\omega)|\right)
r^{n-1}\, dr
\, d\omega      \notag 
\\                                    
&=\int_{S^{n-1}}\int_0^{(\lambda^{-1}|\Omega(u,\omega)|)^{1/(n-k)}}r^{n-1} 
\, dr\, d\omega                                            \notag 
=\int_{S^{n-1}}(\lambda^{-1}|\Omega(u,\omega)|)^{n/(n-k)}\frac{1}{n} 
\, d\omega      
\\                                        
&=\frac{1}{n} \lambda^{-q} \int_{S^{n-1}} |\Omega(u,\omega)|^q \, 
d\omega,                                                  \notag 
\end{align} 
where $\lambda>0$.   
By this and \eqref{e9} we have  \eqref{eq1}. 
\par 
We note that \eqref{eq1} implies 
$$\sup_u\left|\int f_1(v)K(u,v)\, dv\right|\leq \|f_1\|_{n/k,1}
\sup_u \|K(u,\cdot)\|_{q,\infty} \leq 
C\|f_1\|_{n/k,1}\prod_{j=2}^n \|f_j\|_{n/k,1}. $$   
Therefore  
\begin{align}\label{e10} 
|A_{k,n}(f_0, f_1, \dots, f_n)|
&= \left|\int f_0(u)f_1(v) K(u,v)\, du\, dv\right|
\\ 
&= 
\left|\int f_0(u) \left(\int f_1(v)K(u,v)\, dv\right)\, du\right|  \notag 
\\ 
&\leq \|f_0\|_1 \left\|\int f_1(v)K(u,v)\, dv \right\|_{L^\infty(du)} \notag 
\\ 
&\leq C\|f_0\|_1 \|f_1\|_{n/k,1}\prod_{j=2}^n \|f_j\|_{n/k,1}. \notag  
\end{align}  
By interpolation arguments using \eqref{e10}, which will be given in 
Section 5, we have Proposition \ref{prop1.3}. 
\end{proof} 
\par 
Finally, we prove the following lemma used in the proof of \eqref{e9}.  
\begin{lemma}\label{L1.5} 
Let $1<p<\infty, \alpha=n/p, \omega\in S^{n-1}$. Define 
$$B_\alpha f(\omega)=\int_{0}^\infty f(t\omega)t^{\alpha-1}\, dt. $$  
Then 
$$ \|B_\alpha f\|_{L^p(S^{n-1})}\leq C\|f\|_{p,1}. $$ 
\end{lemma} 
\begin{proof} 
Let $g\in L^{p'}(S^{n-1})$. Then 
\begin{align*} 
\int_{S^{n-1}} g(\omega)B_\alpha f(\omega)\, d\omega&= 
\int_{\Bbb R^n} f(x)|x|^{\alpha-n} g(x')\, dx 
\\ 
&\leq \|f\|_{p,1}\||x|^{\alpha-n}g(x')\|_{p',\infty}
\\ 
&\leq C\|f\|_{p,1}\|g\|_{L^{p'}(S^{n-1})}, 
\end{align*}  
where the last inequality follows by arguing similarly to \eqref{e3.5}. 
This will imply the conclusion. 
\end{proof}  

\section{Proof of Theorem \ref{the2} } \label{sec4} 

In this section we give a proof of Theorem \ref{the2}. 
\par 
Let  
\begin{equation*} 
B_{k,n}( f_1, \dots, f_n) 
=\int_{G_{k,n}} \left(\prod_{j=1}^n 
\int_{S^{k-1}_\theta} f_j(\omega_j)\, d\lambda_{\theta}(\omega_j) \right)
\, d\sigma(\theta). 
\end{equation*} 

The following result implies Theorem \ref{the2}.
\begin{proposition}\label{p4.1} Let $1\leq k<n$, $k\in \Bbb Z$. Then we have 
\begin{equation*} \label{e2.16} 
\left|B_{k,n}( f_1, \dots, f_n)\right| \leq C\|f_1\|_{n/k,n} \|f_2\|_{n/k,n}
\dots \|f_n\|_{n/k,n}.   
\end{equation*} 
\end{proposition} 
We note that  this follows by H\"{o}lder's inequality for $k=1$.  
This can be described more precisely as follows. 
Let $\theta\in G_{1,n}$ and $\theta \cap S^{n-1}=\{\eta, -\eta\}$.  
Then we have 
\begin{equation*} \label{S1n}
S_{1,n}(f)(\theta)
=\int_{S^{0}_\theta} f(\omega)\, d\lambda_{\theta}(\omega)
=(f(\eta)+f(-\eta))/2.  
\end{equation*} 
Let $\beta : S^{n-1} \to G_{1,n}$ be defined by $\beta(\omega)=
\theta(\omega)$. We note that $\beta^{-1}(\{\theta(\omega)\})
=\{\omega, -\omega\}$. 
The measure $d\sigma$ on $G_{1,n}$ is defined as (see \cite[3.2]{KP}) 
$$\int_{G_{1,n}} F(\theta)\, d\sigma(\theta) =\int_{S^{n-1}}F(\beta(\omega))
\, d\omega.  $$  
Thus 
\begin{equation*} 
B_{1,n}( f_1, \dots, f_n) 
=\int_{G_{1,n}} \left(\prod_{j=1}^n S_{1,n}(f_j)(\theta) \right)
\, d\sigma(\theta)  
=\int_{S^{n-1}} \left(\prod_{j=1}^n (f_j(\omega)+f_j(-\omega))/2\right)
\, d\omega. 
\end{equation*} 
Therefore, H\"{o}lder's inequality implies Proposition \ref{p4.1} for $k=1$. 
\par 
As in Proposition $\ref{prop1.3}$, we have more general result, 
when $k\geq 2$.  
\begin{proposition} \label{p4.2} 
Suppose that $2\leq k<n$, $k\in \Bbb Z$. 
Let 
$$\frac{k-1}{n-1}<\frac{1}{p_j}<1, \quad\quad  1\leq j\leq n, 
\quad\quad \sum_{j=1}^n \frac{1}{p_j}=k.  $$  
Then 
\begin{equation*} 
\left|B_{k,n}( f_1, \dots, f_n)\right| 
\leq C\prod_{j=1}^n\|f_j\|_{p_j,n}.   
\end{equation*} 

\end{proposition} 

\par 
 For $k\geq 2$, we show that 
\begin{equation}\label{e2.15} 
\left|B_{k,n}( f_1, \dots, f_n)\right| \leq C\|f_1\|_1 
\prod_{j=2}^n \|f_j\|_{(n-1)/(k-1),1}, 
\end{equation}  
which implies Proposition \ref{p4.2} by interpolation. 
See Section 5 for the interpolation arguments.   
\par 
Let $1<\alpha<n$,  $\widetilde{\omega} \in S^{n-2}$ and 
\begin{equation*} 
C_\alpha f(\widetilde{\omega})=
\int_0^{\pi} f(\cos t, (\sin t)\widetilde{\omega})|\sin t|^{\alpha-2}
\, dt. 
\end{equation*} 
Suppose that a function $g$ satisfies 
\begin{equation} \label{ftng}
g(\cos t, (\sin t)\widetilde{\omega})=g(0,\widetilde{\omega}), \quad 
0\leq t\leq \pi, \quad \widetilde{\omega}\in S^{n-2}. 
\end{equation} 
Then we have 
\begin{align*} 
&\left|\int_{S^{n-2}} g(0,\widetilde{\omega})
C_\alpha f(\widetilde{\omega})\, d\widetilde{\omega} \right| 
\\ 
&= \left|\int_{S^{n-2}} 
\int_0^{\pi} g(\cos t,(\sin t)\widetilde{\omega})
f(\cos t, (\sin t)\widetilde{\omega})|\sin t|^{\alpha-2}
\, dt \, d\widetilde{\omega}\right|    
\\ 
&=c\left|\int_{S^{n-1}} g(\omega) f(\omega)|\omega'|^{\alpha-n} \, d\omega
\right| 
\\ 
&\leq c\|f\|_{p, 1}\|g(\omega)|\omega'|^{\alpha-n}\|_{p',\infty}, 
\end{align*}
where $\omega=(\omega^{(1)}, \omega')$, $\omega'\in \Bbb R^{n-1}$, 
$p=(n-1)/(\alpha-1)$.  
We need the inequality 
\begin{equation}\label{e2.17}  
\|C_\alpha f\|_{L^{(n-1)/(\alpha-1)}(S^{n-2})}
 \leq C\|f\|_{L^{(n-1)/(\alpha-1),1}(S^{n-1})},  
\end{equation} 
which can be shown by applying the following lemma in the estimates above. 

\begin{lemma}\label{L2.9} 
We have 
$$\|g(\omega)|\omega'|^{\alpha-n}\|_{(n-1)/(n-\alpha),\infty}\leq 
C\left(\int_{S^{n-2}} |g(0,\widetilde{\omega})|^{(n-1)/(n-\alpha)}
\, d\widetilde{\omega}\right)^{(n-\alpha)/(n-1)},   $$ 
where $g$ is assumed to satisfy \eqref{ftng} and $1<\alpha<n$. 
\end{lemma}
\begin{proof} Let $\lambda>0$.  We have 
\begin{align*} 
&\left|\left\{\omega\in S^{n-1}: |\omega'|^{\alpha-n}  
\left|g(\omega)\right|>\lambda \right\}\right| 
=\int_{S^{n-1}}\chi_{[1,\infty)}\left(\lambda^{-1}|\omega'|^{\alpha-n}
|g(\omega)|\right)\, d\omega  
\\ 
&=c\int_{S^{n-2}}\int_0^\pi  
\chi_{[1,\infty)}\left(\lambda^{-1}(\sin t)^{\alpha-n}|
g(\cos t,(\sin t)\widetilde{\omega})|\right)(\sin t)^{n-2}\, dt  
\, d\widetilde{\omega} 
\\ 
&=c\int_{S^{n-2}}\int_0^\pi  
\chi_{[1,\infty)}\left(\lambda^{-1}(\sin t)^{\alpha-n}|
g(0,\widetilde{\omega})|\right)(\sin t)^{n-2}\, dt
\, d\widetilde{\omega},  
\end{align*} 
which is equal to 
\begin{align*} 
&c\int_{S^{n-2}}\int_0^\pi
\chi_{\left(0, (\lambda^{-1}|g(0,\widetilde{\omega})|)^{1/(n-\alpha)}\right]}
(\sin t) (\sin t)^{n-2}  \, dt\, d\widetilde{\omega}
\\ 
& 
=2c\int_{S^{n-2}}\int_0^{\pi/2}
\chi_{\left(0, (\lambda^{-1}|
g(0,\widetilde{\omega})|)^{1/(n-\alpha)}\right]}(\sin t) 
(\sin t)^{n-2}  \, dt\, d\widetilde{\omega} =:I.  
\end{align*} 
Changing variables, we see that 
\begin{align*} 
&I=2c\int_{S^{n-2}}\int_0^{1}
\chi_{\left(0, (\lambda^{-1}|
g(0,\widetilde{\omega})|)^{1/(n-\alpha)}\right]}(u) 
u^{n-2} (1-u^2)^{-1/2} \, du\, d\widetilde{\omega} 
\\ 
&
= 2c\int_{S^{n-2}}
\int_0^{\min\left(1, (\lambda^{-1}|g(0,\widetilde{\omega})|)^{1/(n-\alpha)}
\right)} u^{n-2} (1-u^2)^{-1/2} \, du\, d\widetilde{\omega}
\\ 
&
\leq C\int_{S^{n-2}}(\lambda^{-1}|g(0,\widetilde{\omega})|)^{(n-1)/(n-\alpha)}
\, d\widetilde{\omega}.     
\end{align*} 
This completes the proof. 
\end{proof}
\par 
We assume Proposition \ref{p4.1} for $B_{k-1,m}$, $m>k-1$, and prove 
\eqref{e2.15} for $B_{k,n}$, $2 \leq k<n$.  This proves Proposition \ref{p4.2} 
for $B_{k,n}$, $2 \leq k<n$ by interpolation.  
Since Proposition \ref{p4.1} is true for $k=1$, this will give the proofs of 
Propositions \ref{p4.1} and \ref{p4.2} by induction. 
\par 
We note by Lemma \ref{L2.6} that 
\begin{align}\label{e4.1}  
&B_{k,n}(f_1, \dots, f_n) 
\\ 
&=\int f_1(\omega_1)\dots f_k(\omega_k)\left(\prod_{j=k+1}^n 
\int_{S^{k-1}_\theta} f_j(\omega_j)\, d\lambda_{\theta}(\omega_j) \right) 
 d\lambda_\theta(\omega_1) \dots \, d\lambda_\theta(\omega_k)\, 
d\sigma(\theta)                            \notag 
\\ 
&=c\int f_1(\omega_1)\dots f_k(\omega_k)\left(\prod_{j=k+1}^n 
\int_{S^{k-1}_{\theta(\omega_1, \dots, \omega_k)}} 
f_j(\omega_j)\, d\lambda_{\theta(\omega_1, \dots, \omega_k)}(\omega_j) \right) 
 \notag 
\\ 
&\quad\quad \times \left|G(\omega_1, \dots, \omega_k)\right|^{\frac{k-n}{2}} 
\, d\omega_1 \dots\, d\omega_k.      \notag 
\end{align} 
\par 
Let $\omega_1=e_1$ and write 
$\omega_\ell=(\cos t_\ell,(\sin t_\ell)\widetilde{\omega}_\ell)$, 
$0<t_\ell\leq \pi$, $\widetilde{\omega}_\ell\in S^{n-2}$ 
for $2\leq \ell\leq k$. 
Then, for $j\geq k+1$, 
\begin{multline*} 
\int_{S^{k-1}_{\theta(\omega_1, \dots, \omega_k)}} 
f_j(\omega_j)\, d\lambda_{\theta(\omega_1, \dots, \omega_k)}(\omega_j) 
\\ 
=c\int_{S^{k-2}_{\theta(\tilde{\omega}_2, \dots, \tilde{\omega}_k)}}
\int_0^\pi f_j(\cos t, (\sin t)\widetilde{\omega}_j)(\sin t)^{k-2}\, dt 
\, d\lambda_{\theta(\tilde{\omega}_2, \dots, \tilde{\omega}_k)}
(\widetilde{\omega}_j).  
\end{multline*}  
Let 
\begin{equation*}  
F(f_j)(\widetilde{\omega}_j)=C_k(f_j)(\widetilde{\omega}_j)
=\int_0^\pi f_j(\cos t, (\sin t)\tilde{\omega}_j)(\sin t)^{k-2}\, dt, 
\quad \tilde{\omega}_j \in S^{n-2}  
\end{equation*}  
for $2\leq j\leq n$. 
Define 
\begin{multline*}  
\Omega(\omega_1, \omega_2)=\Omega(\omega_1, \omega_2)(f_3, \dots, f_n)
\\ 
=
\int f_3(\omega_3)\dots f_k(\omega_k) 
\\ 
\times \prod_{j=k+1}^n 
\int_{S^{k-1}_{\theta(\omega_1,\dots, \omega_k)}} 
f_j(\omega_j)\, d\lambda_{\theta(\omega_1,\dots, \omega_k)}(\omega_j)
|G(\omega_1, \dots, \omega_k)|^{(k-n)/2}\, d\omega_3\dots \,d\omega_k 
\end{multline*}  
for $k\geq 3$; let 
\begin{multline*}  
\Omega(\omega_1, \omega_2)=\Omega(\omega_1, \omega_2)(f_3, \dots, f_n)
\\ 
= \prod_{j=3}^n 
\int_{S^{1}_{\theta(\omega_1,\omega_2)}} 
f_j(\omega_j)\, d\lambda_{\theta(\omega_1,\omega_2)}(\omega_j)
|G(\omega_1,\omega_2)|^{(2-n)/2} 
\end{multline*}  
when $k=2$.

We need the following result. 

\begin{lemma} \label{L2.10}
Let $\omega_1=e_1=(1, 0, \dots,0)$ and $\omega_j=(\cos t_j, 
(\sin t_j)\widetilde{\omega}_j)$, $\widetilde{\omega}_j\in S^{n-2}$, 
$2\leq j\leq k$. Then 
$$G(\omega_1,\omega_2, \dots, \omega_k)
=\sin^2 t_2 \dots \sin^2 t_k
G(\widetilde{\omega}_2, \dots, \widetilde{\omega}_k). $$ 
\end{lemma} 
\begin{proof} 
We have 
$$\langle \omega_j, \omega_\ell \rangle = \cos t_j\cos t_\ell + 
\sin t_j\sin t_\ell\langle \widetilde{\omega}_j, 
\widetilde{\omega}_\ell \rangle,   $$ 
where $\langle \widetilde{\omega}_j, \widetilde{\omega}_\ell \rangle$ denotes 
the inner product in $\Bbb R^{n-1}$. 
Thus $G(\omega_1,\omega_2, \dots, \omega_k)$ is equal to 
\begin{align*}
  &\begin{vmatrix} 
 1 & \cos t_2  & \cos t_3 &\dots       &  \cos t_k \\ 
  \cos t_2  & \langle \omega_2, \omega_2 \rangle & 
\langle \omega_2, \omega_3 \rangle   &\dots       & 
 \langle \omega_2, \omega_k \rangle \\ 
  \cos t_3 & \langle \omega_3, \omega_2 \rangle & 
\langle \omega_3, \omega_3 \rangle   &\dots       & 
 \langle \omega_3, \omega_k \rangle \\ 
\hdotsfor{5}      \\ 
\hdotsfor{5}      \\ 
\cos t_k & \langle \omega_k, \omega_2 \rangle  
&\langle \omega_k, \omega_3 \rangle &\dots       &  
\langle \omega_k, \omega_k \rangle \\ 
\end{vmatrix}  
\\ 
\vspace{5mm} 
\\ 
&=\begin{vmatrix} 
 1 & \cos t_2  & \cos t_3 &\dots       &  \cos t_k \\ 
 0  & 1-\cos^2 t_2 & 
\langle \omega_2, \omega_3 \rangle -\cos t_2\cos t_3  &\dots       & 
 \langle \omega_2, \omega_k \rangle - \cos t_2\cos t_k \\ 
0& \langle \omega_3, \omega_2 \rangle -\cos t_3\cos t_2 & 
1- \cos^2 t_3 &\dots       & 
 \langle \omega_3, \omega_k \rangle - \cos t_3\cos t_k
\\
\hdotsfor{5}      \\ 
\hdotsfor{5}      \\ 
0 & \langle \omega_k, \omega_2 \rangle -\cos t_k\cos t_2
&\langle \omega_k, \omega_3 \rangle -\cos t_k\cos t_3 &\dots       &  
1 -\cos^2 t_k\\ 
\end{vmatrix}   
\\ 
\vspace{5mm} 
\\ 
&=\begin{vmatrix} 
 \sin^2 t_2 & \sin t_2 \sin t_3\langle \widetilde{\omega}_2, 
\widetilde{\omega}_3 \rangle  & \dots       &  
\sin t_2 \sin t_k\langle \widetilde{\omega}_2, 
\widetilde{\omega}_k \rangle  \\ 
\sin t_3 \sin t_2\langle \widetilde{\omega}_3, 
\widetilde{\omega}_2 \rangle  
  & \sin^2 t_3 &  \dots       & 
\sin t_3 \sin t_k\langle \widetilde{\omega}_3, 
\widetilde{\omega}_k\rangle   
\\
\hdotsfor{4}      \\ 
\hdotsfor{4}      \\ 
\sin t_k \sin t_2\langle \widetilde{\omega}_k, 
\widetilde{\omega}_2\rangle
& \sin t_k \sin t_3\langle \widetilde{\omega}_k, 
\widetilde{\omega}_3\rangle &\dots      &  
\sin^2 t_k\\ 
\end{vmatrix}   
\\ 
\vspace{5mm} 
\\ 
&=\sin^2 t_2\sin^2 t_3 \dots \sin^2 t_k
\begin{vmatrix} 
1 & \langle \widetilde{\omega}_2, 
\widetilde{\omega}_3 \rangle  & \dots       &  
\langle \widetilde{\omega}_2, \widetilde{\omega}_k \rangle  \\ 
\langle \widetilde{\omega}_3, \widetilde{\omega}_2 \rangle  
  & 1 &  \dots       & 
\langle \widetilde{\omega}_3, \widetilde{\omega}_k\rangle   
\\
\hdotsfor{4}      \\ 
\hdotsfor{4}      \\ 
\langle \widetilde{\omega}_k, \widetilde{\omega}_2\rangle
& \langle \widetilde{\omega}_k, \widetilde{\omega}_3\rangle &\dots      &  
1\\ 
\end{vmatrix}   
\\ 
\vspace{5mm} 
\\ 
&=\sin^2 t_2\sin^2 t_3 \dots \sin^2 t_k 
 G(\widetilde{\omega}_2, \dots, \widetilde{\omega}_k). 
\end{align*} 
This completes the proof. 
\end{proof} 
\par 
Choose $O\in {\rm SO}(n)$ so that $O^{-1}\omega_1=e_1$. 
Then, by changing variables and applying Lemma \ref{L2.10}, we have 
\begin{multline*} 
\int f_2(\omega_2)\Omega(\omega_1, \omega_2)\, d\omega_2 
=c\int_{(S^{n-2})^{k-1}} F(f_2(O\cdot))(\widetilde{\omega}_2)\dots  
F(f_k(O\cdot))(\widetilde{\omega}_k) 
\\ 
\times \prod_{j=k+1}^n 
\int_{S^{k-2}_{\theta(\tilde{\omega}_2,\dots, \tilde{\omega}_k)}} 
F(f_j(O\cdot))(\widetilde{\omega}_j)\, 
d\lambda_{\theta(\widetilde{\omega}_2,\dots, \widetilde{\omega}_k)}
(\widetilde{\omega}_j) 
\\ 
\times |G(\widetilde{\omega}_2, \dots, \widetilde{\omega}_k)|^{(k-n)/2}\, 
d\widetilde{\omega}_2\dots \,d\widetilde{\omega}_k. 
\end{multline*}  
By Lemma \ref{L2.6} this is equal to 

\begin{multline*} 
c\int_{G_{k-1,n-1}}\int_{(S^{k-2}_{\theta})^{k-1}} 
F(f_2(O\cdot))(\widetilde{\omega}_2)\dots  
F(f_k(O\cdot))(\widetilde{\omega}_k) 
\\ 
\times \left(\prod_{j=k+1}^n 
\int_{S^{k-2}_{\theta}} 
F(f_j(O\cdot))(\widetilde{\omega}_j)\, 
d\lambda_{\theta}(\widetilde{\omega}_j)\right) 
d\lambda_{\theta}(\widetilde{\omega}_2) \dots 
d\lambda_{\theta}(\widetilde{\omega}_k)\, d\sigma(\theta) 
\\ 
= c\int_{G_{k-1,n-1}} \prod_{j=2}^n  S_{k-1,n-1}(F(f_j(O\cdot)))(\theta) 
\, d\sigma(\theta).  
\end{multline*}  
We have Proposition \ref{p4.1} for $B_{k-1,m}$, $m>k-1$, as the induction 
hypothesis and hence we have the inequality of Theorem \ref{the2} for 
$S_{k-1,m}$.
Thus by H\"older's inequality and the induction hypothesis, we have 

\begin{align*} 
\left|\int_{G_{k-1,n-1}} \prod_{j=2}^n  S_{k-1,n-1}(F(f_j(O\cdot)))(\theta) 
\, d\sigma(\theta)\right|
&\leq  \prod_{j=2}^n  \|S_{k-1,n-1}(F(f_j(O\cdot)))\|_{n-1}
\\  
&\leq  C\prod_{j=2}^n  \|F(f_j(O\cdot))\|_{(n-1)/(k-1), n-1} 
\\ 
&\leq  C\prod_{j=2}^n  \|F(f_j(O\cdot))\|_{(n-1)/(k-1)}.  
\end{align*} 
Applying \eqref{e2.17} with $\alpha=k$,  
\begin{equation*}  
 \|F(f_j(O\cdot))\|_{(n-1)/(k-1)}\leq C\|f_j\|_{(n-1)/(k-1),1}.  
\end{equation*}  
This estimate is uniform in $O$ and hence in $\omega_1$. 
Thus we have 
\begin{equation*}  
\left|\int f_1(\omega_1)f_2(\omega_2)\Omega(\omega_1, \omega_2)\, d\omega_1 
\, d\omega_2\right|  
\leq C \|f_1\|_1\prod_{j=2}^n  \|f_j\|_{(n-1)/(k-1), 1}. 
\end{equation*}
By \eqref{e4.1} this proves \eqref{e2.15} for $B_{k,n}$.

\section{Interpolation arguments and proofs of Propositions \ref{prop1.3}and 
\ref{p4.2}}\label{sec5}  

 By \eqref{e10} we can show the following result, which will be used 
 in proving Proposition \ref{prop1.3}. 
\begin{lemma}\label{L1.6} 
Suppose that $1\leq k<n$.  
Let $1<p_0, p_1, \dots, p_n<\frac{n}{k}$, $\sum_{j=0}^n \frac{1}{p_j}=
k+1$. Then 
$$\left|A_{k,n}(f_0, f_1, \dots , f_n)\right| 
\leq C \|f_0\|_{p_0, 1}\|f_1\|_{p_1, \infty}
\dots \|f_n\|_{p_n, \infty}.   $$ 
\end{lemma} 
When $p=\infty$, we consider only the case $q=\infty$ in $L^{p,q}$. 
We need the following interpolation results.   

\begin{lemma}\label{L1.7} 
Let $1\leq v, w \leq \infty$ and 
$$\frac{1}{v}+\frac{1}{w}=1.  $$ 
Let $0<\theta<1$.   
Let $1\leq s_i, u_i, a_i, b_i\leq \infty$, $i=0, 1$. 
We assume that $a_i=1$ if $s_i=1$ and that $b_i=1$ if $u_i=1$. 
Define 
$A_i=L^{s_i, a_i}(\Bbb R^n)$, $B_i=L^{u_i, b_i}(\Bbb R^n)$, $i=0, 1$, and 
$\bar{A}=(A_0, A_1)$, $\bar{B}=(B_0, B_1)$. 
 Then $\bar{A}_{\theta, v}=L^{s,v}(\Bbb R^n)$, where 
$$\frac{1}{s}=\frac{1-\theta}{s_0}+\frac{\theta}{s_1}$$  
and we require 
$$\frac{1}{v}=\frac{1-\theta}{a_0}+\frac{\theta}{a_1}$$ 
 if $s_0=s_1$ and also $\bar{B}_{\theta, w}
=L^{u,w}(\Bbb R^n)$, where 
$$\frac{1}{u}=\frac{1-\theta}{u_0}+\frac{\theta}{u_1}$$  
and we assume 
$$\frac{1}{w}=\frac{1-\theta}{b_0}+\frac{\theta}{b_1}$$
 if $u_0=u_1$.   $($See \cite[Chap. 3]{BL} for $\bar{A}_{\theta, v}$.$)$ 
 Suppose that $T: A_i\times B_i \to \Bbb C$ be 
a bilinear operator such that 
$$|T(f_1,f_2)|\leq M_i\|f_1\|_{A_i}\|f_2\|_{B_i}, \quad i=0, 1. $$  
We assume that 
$$T(f_1,f_2)= A_{k,n}(g_0, g_1, \dots , g_n)$$  
with $f_1=g_j$, $f_2=g_k$ for some fixed $j, k$, $j\neq k$; functions except
 for $g_j$, $g_k$ are fixed. Also,  
all functions $g_j$ are initially assumed to be continuous and 
compactly supported. 
Then 
$$|T(f_1,f_2)|\leq C M_1^{1-\theta}M_2^\theta
\|f_1\|_{\bar{A}_{\theta,v}}\|f_2\|_{\bar{B}_{\theta, w}}
=C M_1^{1-\theta}M_2^\theta
\|f_1\|_{L^{s,v}}\|f_2\|_{L^{u,w}}. $$ 
\end{lemma} 
See \cite{LP} and \cite{Z} for results relevant to Lemma \ref{L1.7}.  
\par 
We also need the following. 

\begin{lemma}\label{L1.9} 
Let $1\leq p, r, q, q_0, q_1\leq \infty$, $r\leq q_0, q_1$, 
 $0<\theta, \eta<1$, 
\begin{align*} 
\frac{1}{q}&=\frac{1-\eta}{q_0}+\frac{\eta}{q_1}, 
\\ 
\frac{1}{p}&=\frac{1-\theta}{r}, \quad r<p. 
\end{align*} 
Then 
$$(L^{p,q_0}, L^{p,q_1})_{\eta, q}=(L^r, L^\infty)_{\theta, q}
=L^{p,q}= (L^{p,q_0}, L^{p,q_1})_{[\eta]}.  $$
\end{lemma} 
In the conclusion of the lemma, the equality of spaces 
means that the spaces are equal with 
equivalent norms; we also have this rule for description in what follows. 
\par 
Here we recall $\bar{A}_{[\theta]}$, where $\bar{A}=(A_0,A_1)$ denotes a 
compatible pair of normed vector spaces (see 2.3 of \cite{BL}).  
Let $\mathscr F(\bar{A})$ be the space of all continuous functions $f$ from 
$S=\{z\in \Bbb C: 0\leq \re z\leq 1\}$ to $\Sigma(\bar{A})=A_0+A_1$ which are 
analytic in $S_0=\{z\in \Bbb C: 0< \re z< 1\}$ and functions $f_j(t)=f(j+it)$, 
$t\in \Bbb R$, are  $A_j$-valued continuous functions on $\Bbb R$ 
with respect to $A_j$-norm such that     
$\lim_{|t|\to \infty} \|f_j(t)\|_{A_j}=0$, $j=0,1$. 
\par 
Let $0<\theta<1$ and 
$$\bar{A}_{[\theta]}=\{a\in \Sigma(\bar{A}): a=f(\theta) \quad 
\text{for some $f\in \mathscr F(\bar{A})$}\}. $$ 
We define 
$$ \|a\|_{\bar{A}_{[\theta]}}=\inf\{\|f\|_{\mathscr F(\bar{A})}: 
a=f(\theta), \quad f\in \mathscr F(\bar{A})\},  $$  
where 
$$\|f\|_{\mathscr F(\bar{A})}=\max\left(\sup_{t\in \Bbb R}\|f(it)\|_{A_0}, 
\sup_{t\in \Bbb R}\|f(1+it)\|_{A_1}\right). 
$$ 
We also write  $\|a\|_{[\theta]}$ for $\|a\|_{\bar{A}_{[\theta]}}$
 when $\bar{A}$ is fixed. (See  \cite[Chap. 4]{BL}.) 
\begin{lemma}\label{L1.9+}
Let $0<\eta<1$.   
Let $1\leq  a_i, b_i\leq \infty$, $i=0,1$ and $1<s, u \leq \infty$.  
Define $v, w $ by 
$$\frac{1}{v}=\frac{1-\eta}{a_0}+\frac{\eta}{a_1},  \quad 
\frac{1}{w}=\frac{1-\eta}{b_0}+\frac{\eta}{b_1}. $$
  Let  $T: L^{s,a_i}\times L^{u,b_i} \to \Bbb C$ be as in Lemma $\ref{L1.7}$. 
Suppose that 
$$|T(f_1,f_2)|\leq 
M_i\|f_1\|_{L^{s, a_i}}\|f_2\|_{L^{u,b_i}}, \quad i=0, 1. $$ 
Then 
$$|T(f_1,f_2)|\leq C M_0^{1-\eta}M_1^\eta
\|f_1\|_{L^{s,v}}\|f_2\|_{L^{u, w}}. $$ 
\end{lemma} 
\begin{proof} 
By Theorem 4.4.1 of \cite{BL}, we have 
$$|T(f_1,f_2)|\leq M_0^{1-\eta}M_1^\eta
\|f_1\|_{\left(L^{s,a_0}, L^{s,a_1}\right)_{[\eta]}}
\|f_2\|_{\left(L^{u,b_0}, L^{u,b_1}\right)_{[\eta]}}. $$ 
From this and Lemma \ref{L1.9}, the conclusion follows. 
\end{proof} 

\begin{remark}\label{re5.5} 
We have analogues of Lemmas \ref{L1.7} and \ref{L1.9+} for the Lorentz 
spaces over $S^{n-1}$, where the operator $T$ is replaced with the one defined 
by using $B_{k,n}$ in Section \ref{sec4}.  
\end{remark} 

 We shall give proofs of Lemmas \ref{L1.7} and \ref{L1.9} in Sections 6 and 8. 
\begin{proof}[Proof of Lemma $\ref{L1.6}$]   
We may assume 
$$\frac{k}{n}<\frac{1}{p_0}\leq \frac{1}{p_1}\leq \dots \leq \frac{1}{p_n}<1, 
\qquad \sum_{j=0}^n \frac{1}{p_j}=k+1.  $$
Define $\theta_0\in (0,1)$ and $u_1$ by    
$$\frac{1}{p_0}=(1-\theta_0)\frac{k}{n}  +\theta_0, \qquad 
\frac{1}{u_1}=(1-\theta_0)\ +\theta_0\frac{k}{n}.  $$   
Then
$$\frac{1}{p_0}+\frac{1}{u_1}=1+ \frac{k}{n}    $$  
and 
$$\frac{1}{p_n}<\frac{1}{u_1}, $$  
since if $\frac{1}{p_n}\geq \frac{1}{u_1}$, then 
$$\frac{1}{p_0}+\frac{1}{p_n}\geq \frac{1}{p_0}+\frac{1}{u_1}=
\frac{n+k}{n}, $$  
and so 
$$\frac{1}{p_0}+\frac{1}{p_n} +\frac{1}{p_1}+\dots +\frac{1}{p_{n-1}}>
\frac{n+k}{n} +\frac{k}{n}(n-1)=k+1, $$  
which contradicts our assumption. 
Next, define $\theta_1\in (0,1)$ and $u_2$ by 
$$\frac{1}{p_1}=(1-\theta_1)\frac{k}{n}  +\theta_1\frac{1}{u_1}, \qquad   
\frac{1}{u_2}=(1-\theta_1)\frac{1}{u_1}\ +\theta_1\frac{k}{n}.  $$   
Then
$$\frac{1}{p_1}+\frac{1}{u_2}=\frac{k}{n}+ \frac{1}{u_1}   $$ 
and if $n\geq 3$, 
$$\frac{1}{p_n}<\frac{1}{u_2}, $$  
since if we would have $\frac{1}{p_n}\geq \frac{1}{u_2}$, then using 
$$\frac{1}{p_0}+\frac{1}{p_1}+\frac{1}{u_2}= \frac{1}{p_0}+\frac{1}{u_1}
+ \frac{k}{n}=\frac{n+2k}{n}, $$  
we would  have 
$$\frac{1}{p_0}+\frac{1}{p_1}+\frac{1}{p_n}\geq \frac{1}{p_0}+\frac{1}{p_1}+ 
\frac{1}{u_2}=\frac{n+2k}{n}, $$  
and hence 
$$\frac{1}{p_0}+\frac{1}{p_1}+\frac{1}{p_n} +\frac{1}{p_2}+\dots +
\frac{1}{p_{n-1}}>
\frac{n+2k}{n} +\frac{k}{n}(n-2)=k+1, $$  
which violates our assumption. Also, if $n=2$, $1/p_2= 1/u_2$. 
\par 
After defining $\theta_{j-1}$ and $u_j$, 
similarly, we can define $\theta_j\in (0,1)$ and  $u_{j+1}$ by 
$$\frac{1}{p_j}=(1-\theta_j)\frac{k}{n}  +\theta_j\frac{1}{u_j}, \qquad   
\frac{1}{u_{j+1}}=(1-\theta_j)\frac{1}{u_j}\ +\theta_j\frac{k}{n}  $$ 
for $j=1,2, \dots, n-1$.    
Then
$$\frac{1}{p_j}+\frac{1}{u_{j+1}}=\frac{k}{n}+\frac{1}{u_j}. 
$$ 
We have 
 $$\frac{1}{p_n}<\frac{1}{u_{j+1}}, \qquad 0\leq j \leq n-2. $$  
To see this we observe 
\begin{align*}
\frac{1}{p_0}+\frac{1}{p_1}+\dots +\frac{1}{p_j} +\frac{1}{u_{j+1}}
&= \frac{1}{p_0}+\dots +\frac{1}{p_{j-1}} +\frac{1}{u_{j}}
+\frac{k}{n}   
\\ 
&=\frac{1}{p_0}+\dots +\frac{1}{p_{j-2}} +\frac{1}{u_{j-1}}
+\frac{2k}{n} 
\\ 
&= \dots 
\\ 
&= \frac{1}{p_0}+\frac{1}{u_1}+\frac{jk}{n}.     
\end{align*} 
Thus if $\frac{1}{p_n}\geq \frac{1}{u_{j+1}}$, then 
\begin{align*}
\frac{1}{p_0}+\frac{1}{p_1}+\dots +\frac{1}{p_j} +\frac{1}{p_n}
&\geq 
\frac{1}{p_0}+\frac{1}{p_1}+\dots +\frac{1}{p_j} +\frac{1}{u_{j+1}}
\\ 
&= \frac{1}{p_0}+\frac{1}{u_1}+\frac{jk}{n} 
\\ 
&=\frac{n+(j+1)k}{n}
\end{align*} 
and hence 
\begin{align*}
&\frac{1}{p_0}+\frac{1}{p_1}+\dots +\frac{1}{p_j} +\frac{1}{p_n}
+ \frac{1}{p_{j+1}}+\dots  +\frac{1}{p_{n-1}}
\\ 
&>\frac{n+(j+1)k}{n}+\frac{k}{n}(n-j-1)
\\
&= k+1, 
\end{align*} 
which contradicts our assumption.  
Also, we have 
$$ \frac{1}{u_n}=\frac{1}{p_n}$$  
since 
$$\frac{1}{p_0}+\frac{1}{p_1}+\dots +\frac{1}{p_{n-1}}+
\frac{1}{u_n}=\frac{n+nk}{n}=k+1.   $$  
\par 
We write $A=|A_{k,n}(f_0, f_1, \dots, f_n)|$. 
By \eqref{e10} and symmetry, we have 
$$A
\leq C \|f_0\|_{n/k,1}\|f_1\|_1\|f_2\|_{n/k,1}\|f_3\|_{n/k,1}
\dots \|f_n\|_{n/k,1}, $$
$$A\leq C \|f_0\|_{1}\|f_1\|_{n/k,1}
\|f_2\|_{n/k,1}\|f_3\|_{n/k,1}\dots \|f_n\|_{n/k,1}. $$
Interpolating by using Lemma \ref{L1.7} and applying symmetry, 
$$A\leq C \|f_0\|_{p_0,\infty}\|f_1\|_{u_1,1}
\|f_2\|_{n/k,1}\|f_3\|_{n/k,1}\dots \|f_n\|_{n/k,1}, $$
$$A\leq C \|f_0\|_{p_0,\infty}\|f_1\|_{n/k,1}
\|f_2\|_{u_1,1}\|f_3\|_{n/k,1}\dots \|f_n\|_{n/k,1}. $$  
Using Lemma \ref{L1.7} and symmetry, 
$$A\leq C \|f_0\|_{p_0,\infty}\|f_1\|_{p_1,\infty}
\|f_2\|_{u_2,1}\|f_3\|_{n/k,1}\dots \|f_n\|_{n/k,1}, $$
$$A\leq C \|f_0\|_{p_0,\infty}\|f_1\|_{p_1,\infty}
\|f_2\|_{n/k,1}\|f_3\|_{u_2,1}\dots \|f_n\|_{n/k,1}. $$  
Continuing this procedure, 
$$A\leq C \|f_0\|_{p_0,\infty}\|f_1\|_{p_1,\infty}
\dots \|f_{n-2}\|_{p_{n-2},\infty}\|f_{n-1}\|_{u_{n-1},1}\|f_n\|_{n/k,1}, $$
$$A\leq C \|f_0\|_{p_0,\infty}\|f_1\|_{p_1,\infty}
\dots \|f_{n-2}\|_{p_{n-2},\infty}\|f_{n-1}\|_{n/k,1}\|f_n\|_{u_{n-1},1}.
$$  
Interpolating between these estimates, we have 
$$A \leq C \|f_0\|_{p_0,\infty}\|f_1\|_{p_1,\infty}
\dots \|f_{n-2}\|_{p_{n-2},\infty}\|f_{n-1}\|_{p_{n-1},\infty}
\|f_n\|_{u_{n},1}. 
$$  
This and symmetry complete the proof, since $u_n=p_n$.

\end{proof}  
\begin{proof}[Proof of Proposition $\ref{prop1.3}$]  
We also write $A=|A_{k,n}(f_0, f_1, \dots, f_n)|$. 
By Lemma \ref{L1.6}, 
$$A\leq C \|f_0\|_{p_0,1}\|f_1\|_{p_1,\infty}
\|f_1\|_{p_2,\infty}
\dots \|f_{n-2}\|_{p_{n-2},\infty}\|f_{n-1}\|_{p_{n-1},\infty}
\|f_n\|_{p_{n},\infty}, 
$$  
$$A\leq C \|f_0\|_{p_0,\infty}\|f_1\|_{p_1,1}
\|f_1\|_{p_2,\infty}
\dots \|f_{n-2}\|_{p_{n-2},\infty}\|f_{n-1}\|_{p_{n-1},\infty}
\|f_n\|_{p_{n},\infty}. 
$$  
Since 
$$\frac{1}{n+1}=\frac{1}{1}\frac{1}{n+1}  +\frac{1}{\infty}\frac{n}{n+1}, 
\quad 
\frac{n}{n+1}=\frac{1}{\infty}\frac{1}{n+1}+ \frac{1}{1}\frac{n}{n+1}, 
\quad 
\frac{1}{n+1}+ \frac{n}{n+1}=1, 
$$ 
by Lemma \ref{L1.7} in the case $s_0=s_1$ and $u_0=u_1$ or 
Lemma \ref{L1.9+} and symmetry, we have   
$$A\leq C \|f_0\|_{p_0,n+1}\|f_1\|_{p_1,\frac{n+1}{n}}
\|f_2\|_{p_2,\infty}\|f_3\|_{p_3,\infty}
\dots \|f_{n-2}\|_{p_{n-2},\infty}\|f_{n-1}\|_{p_{n-1},\infty}
\|f_n\|_{p_{n},\infty}, 
$$  
$$A\leq C \|f_0\|_{p_0,n+1}\|f_1\|_{p_1,\infty}
\|f_2\|_{p_2,\frac{n+1}{n}}\|f_3\|_{p_3,\infty}
\dots \|f_{n-2}\|_{p_{n-2},\infty}\|f_{n-1}\|_{p_{n-1},\infty}
\|f_n\|_{p_{n},\infty}. 
$$  
Since 
$$\frac{1}{n+1}=\frac{n}{n+1}\frac{1}{n}  +\frac{1}{\infty}\frac{n-1}{n}, 
\quad 
\frac{n-1}{n+1}=\frac{1}{\infty}\frac{1}{n}+ \frac{n}{n+1}\frac{n-1}{n}, 
$$ 
by Lemma \ref{L1.9+} and symmetry, we have 
$$A\leq C \|f_0\|_{p_0,n+1}\|f_1\|_{p_1,n+1}
\|f_2\|_{p_2,\frac{n+1}{n-1}}\|f_3\|_{p_3,\infty}
\dots \|f_{n-1}\|_{p_{n-1},\infty}
\|f_n\|_{p_{n},\infty}, 
$$  
$$A\leq C \|f_0\|_{p_0,n+1}\|f_1\|_{p_1,n+1}
\|f_2\|_{p_2,\infty}\|f_3\|_{p_3,\frac{n+1}{n-1}}
\dots \|f_{n-1}\|_{p_{n-1},\infty}
\|f_n\|_{p_{n},\infty}. 
$$  
(Note that $1/(n+1)+ (n-1)/(n+1)=n/(n+1)\neq 1$.) 
In general, since 
\begin{align*}   
\frac{1}{n+1}&=\frac{n-j+1}{n+1}\frac{1}{n-j+1} 
 +\frac{1}{\infty}\frac{n-j}{n-j+1},   
\\ 
\frac{n-j}{n+1}&=\frac{1}{\infty}\frac{1}{n-j+1}+ 
\frac{n-j+1}{n+1}\frac{n-j}{n-j+1}, 
\\ 
&\frac{1}{n-j+1}+\frac{n-j}{n-j+1}=1, 
\end{align*}   
interpolating by using Lemma \ref{L1.9+} between the estimates: 
\begin{multline*}
A\leq C \|f_0\|_{p_0,n+1}\dots \|f_{j-1}\|_{p_{j-1},n+1} 
\\ 
\times \|f_j\|_{p_j,\frac{n+1}{n-j+1}}\|f_{j+1}\|_{p_{j+1},\infty}
\|f_{j+2}\|_{p_{j+2},\infty}\dots \|f_n\|_{p_{n},\infty},  
\end{multline*} 
\begin{multline*}
A\leq C \|f_0\|_{p_0,n+1}\dots \|f_{j-1}\|_{p_{j-1},n+1} 
\\ 
\times \|f_j\|_{p_j,\infty}\|f_{j+1}\|_{p_{j+1},\frac{n+1}{n-j+1}}
\|f_{j+2}\|_{p_{j+2},\infty}
\dots 
\|f_n\|_{p_{n},\infty}    
\end{multline*} 
for $j\leq n-1$, we have 
\begin{multline*}
A\leq C \|f_0\|_{p_0,n+1}\dots \|f_{j-1}\|_{p_{j-1},n+1}
\\                                  
\times \|f_j\|_{p_j,n+1}\|f_{j+1}\|_{p_{j+1},\frac{n+1}{n-j}}
\|f_{j+2}\|_{p_{j+2},\infty}\dots \|f_n\|_{p_n,\infty}
\end{multline*} 
for $j\leq n-2$, and for $j=n-1$ this becomes 
$$A\leq C\|f_0\|_{p_0,n+1}\dots \|f_n\|_{p_n,n+1}, $$
which completes the proof of Proposition $\ref{prop1.3}$ by induction. 
\end{proof} 
\par 
\par 
Next, we prove Proposition \ref{p4.2} from \eqref{e2.15} by  interpolation 
arguments as above.   
Recall \eqref{e2.15}: 
\begin{equation*}
\left|B_{k,n}( f_1, \dots, f_n)\right| \leq C\|f_1\|_1 
\prod_{j=2}^n \|f_j\|_{(n-1)/(k-1),1}.   
\end{equation*}  
Let 
$$\frac{k-1}{n-1}<\frac{1}{p_j}<1, \quad 1\leq j\leq n, \qquad 
\sum_{j=1}^n\frac{1}{p_j}=k.  
$$ 
Then using Remark \ref{re5.5} with \eqref{e2.15} and arguing as in the proof of Lemma \ref{L1.6} from  \eqref{e10}, by taking $n-1$ and $k-1$ for $n$ and $k$, 
respectively, we have 
\begin{equation}\label{e3.1}
\left|B_{k,n}( f_1, \dots, f_n)\right| \leq C\|f_1\|_{p_0,\infty} \dots 
\|f_{n-1}\|_{p_{n-1},\infty} \|f_n\|_{p_n,1}. 
\end{equation}  
Similarly, as Lemma \ref{L1.6} proves Proposition \ref{prop1.3}, 
\eqref{e3.1} and Remark \ref{re5.5} imply  
\begin{equation*}
\left|B_{k,n}( f_1, \dots, f_n)\right| \leq C\prod_{j=1}^n \|f_j\|_{p_j,n}.  
\end{equation*}  
This completes the proof of Proposition \ref{p4.2}.

\section{Proof of Lemma $\ref{L1.7}$}\label{sec6}

Let $\bar{A}=(A_0, A_1)$, $\bar{B}=(B_0, B_1)$ and let 
$S(\bar{A}, v,\theta)=S(\bar{A}, (v,v),\theta)$,  
$S(\bar{B}, w,\theta)=S(\bar{B}, (w,w),\theta)$. 
Here  
$S(\bar{A}, (r_0,r_1),\theta)$ is the subspace of $\Sigma(\bar{A})$ 
consisting of all $a\in \Sigma(\bar{A})$ 
such that 
\begin{equation}\label{e1.7} 
a=\int_0^\infty u(t)\, \frac{dt}{t},  
\end{equation}  
with $u(t)\in \Delta(\bar{A})=A_0\cap A_1$ for all $t>0$ (see \cite[2.3]{BL}), 
 and the integral is taken in $\Sigma(\bar{A})$; further it is  assumed that 
$$\max\left(\left\|t^{-\theta}u(t) \right\|_{L^{r_0}(A_0, dt/t)}, 
\left\|t^{1-\theta}u(t) \right\|_{L^{r_1}(A_1, dt/t)} \right)<\infty,   $$   
where 
$$ \left\|U(t) \right\|_{L^{r_j}(A_j, dt/t)}= 
\left(\int_0^\infty\|U(t)\|_{A_j}^{r_j}\, \frac{dt}{t}\right)^{1/r_j}, $$
with the usual modification when $r_j=\infty$.  (See \cite[3.12]{BL}).)  
The norm is defined as 
\begin{multline*}
\left\|a\right\|_{S(\bar{A}, (r_0,r_1),\theta)}
\\
=\inf\left\{ \max\left(\left\|t^{-\theta}u(t) \right\|_{L^{r_0}(A_0, dt/t)}, 
\left\|t^{1-\theta}u(t) \right\|_{L^{r_1}(A_1, dt/t)} \right): 
 \text{$u$ is as in \eqref{e1.7}}\right\}.  
\end{multline*}
\par 
We assume that $u$ has the form 
\begin{equation}\label{ee6.2}
u(t)(x)=\sum_{j=1}^M F_j(x)m_j(t),  
\end{equation}  
where $F_j \in A_0\cap A_1$  
and $m_j$ is a bounded measurable 
function on $(0, \infty)$ supported on a compact subinterval of $(0, \infty)$. 
\par 
We assume that $a\in \Delta(\bar{A})$ is expressed as in \eqref{e1.7} with $u$ 
as in \eqref{ee6.2}: 
$$ a=\int_0^\infty u(t)\, \frac{dt}{t}=\sum_{j=1}^M c_jF_j, \quad   
c_j=\int_0^\infty m_j(t)\, \frac{dt}{t}. $$  
We note that $\|u(t)\|_{A_j}$ is measurable in $t$ for $j=0, 1$  
and define 
\begin{align*}
\left\|a\right\|_{S^*(\bar{A}, q,\theta)}
&=\inf\Big\{\max\left(\left\|t^{-\theta}u(t) \right\|_{L^{q}(A_0, dt/t)}, 
\left\|t^{1-\theta}u(t) \right\|_{L^{q}(A_1, dt/t)} \right): 
\\ 
&\qquad \qquad\qquad \qquad \quad \text{$a$ and $u$ are as in \eqref{e1.7}
 and $u$ is as in \eqref{ee6.2}}\Big\}.  
\end{align*}

\par 
Let $f_1$, $f_2$ be continuous functions on $\Bbb R^n$ with compact 
support. Let $u_1$, $u_2$ be as in \eqref{ee6.2} such that 
$$f_i=\int_0^\infty u_i(t)\, \frac{dt}{t}, \quad i=1, 2. $$ 
We choose an infinitely differentiable non-negative function $\varphi$ on 
$\Bbb R^n$ with compact support and with integral $1$. Put $\varphi_\epsilon(x)  =\epsilon^{-n}\varphi(\epsilon^{-1}x)$ with $\epsilon>0$.  
Let $f_i^{(\epsilon)}=f_i*\varphi_\epsilon$ 
and 
$$u_i^{(\epsilon)}(t)(x)=\sum_{j=1}^M \widetilde{F}_j^{(i)}*\varphi_\epsilon(x)
m_j(t),    
$$
where 
$\widetilde{F}_j^{(i)}=F_j^{(i)}\chi_K$ with $\chi_K$ denoting the 
characteristic function of a compact set $K$ in $\Bbb R^n$, 
if 
$$u_i(t)(x)=\sum_{j=1}^M F_j^{(i)}(x)m_j(t).     
$$ 
Then if $K$ is sufficiently large according to the supports of $f_1$, $f_2$, 
we see that 
$$f_i^{(\epsilon)}=\int_0^\infty u_i^{(\epsilon)}(t)\, \frac{dt}{t}, 
\quad i=1, 2. $$ 
Further, we see that $\widetilde{F}_j^{(i)}*\varphi_\epsilon$ is 
compactly supported and continuous. 
So we can define    
$$w_\epsilon(t)=\int_0^\infty T\left(u_1^{(\epsilon)}
\left(\frac{t}{t_1}\right), 
u_2^{(\epsilon)}(t_1)\right)\, \frac{dt_1}{t_1}, $$ 
which satisfies  
\begin{equation}\label{e1.8}
T(f_1^{(\epsilon)},f_2^{(\epsilon)})= \int_0^\infty w_\epsilon(t)\, 
\frac{dt}{t}.  
\end{equation}  
\par   
 For $\theta\in (0,1)$, we have 
 \begin{equation}\label{e1.9}
 \left|\int_0^\infty w_\epsilon(t)\, \frac{dt}{t}\right|\leq 
 C_\theta\|t^{-\theta}w_\epsilon(t)\|_\infty^{1-\theta}
 \|t^{1-\theta}w_\epsilon(t)\|_\infty^{\theta}.   
 \end{equation} 
To see this we evaluate the integral on the left hand side as follows: 
\begin{align*}  
\left|\int_0^\infty w_\epsilon(t)\, \frac{dt}{t}\right|& \leq  
\left|\int_0^A w_\epsilon(t)\, \frac{dt}{t}\right| +
\left|\int_A^\infty w_\epsilon(t)\, \frac{dt}{t}\right| 
\\ 
& \leq  
  \|t^{-\theta}w_\epsilon(t)\|_\infty\left|\int_0^A t^{\theta-1}\, dt\right| + 
\|t^{1-\theta}w_\epsilon(t)\|_\infty\left|\int_A^\infty t^{-2+\theta}
\, dt\right| 
\\ 
&= \frac{1}{\theta}A^\theta \|t^{-\theta}w_\epsilon(t)\|_\infty 
+ \frac{1}{1-\theta}A^{\theta-1}\|t^{1-\theta}w_\epsilon(t)\|_\infty 
\\ 
&= 2\theta^{\theta-1}(1-\theta)^{-\theta}
\|t^{-\theta}w_\epsilon(t)\|_\infty^{1-\theta} 
\|t^{1-\theta}w_\epsilon(t)\|_\infty^\theta, 
\end{align*} 
where $A=\theta(1-\theta)^{-1}\|t^{-\theta}w_\epsilon(t)\|_\infty^{-1} 
\|t^{1-\theta}w_\epsilon(t)\|_\infty$.  Now, by H\"{o}lder's inequality,  
\begin{align*}
|t^{-\theta}w_\epsilon(t)|
&\leq M_0\int_0^\infty t^{-\theta}\|u_1^{(\epsilon)}(tt_1^{-1})\|_{A_0} 
\|u_2^{(\epsilon)}(t_1)\|_{B_0} \, \frac{dt_1}{t_1} 
\\ 
&\leq M_0\|t^{-\theta}u_1^{(\epsilon)}(t)\|_{L^{v}(A_0,dt/t)}
\|t^{-\theta}u_2^{(\epsilon)}(t)\|_{L^{w}(B_0,dt/t)}. 
\end{align*} 
Similarly, 
$$|t^{1-\theta}w_\epsilon(t)|\leq 
M_1\|t^{1-\theta}u_1^{(\epsilon)}(t)\|_{L^{v}(A_1,dt/t)}
\|t^{1-\theta}u_2^{(\epsilon)}(t)\|_{L^{w}(B_1,dt/t)}. 
$$  
Thus by \eqref{e1.8} and \eqref{e1.9} we have 
\begin{align*}   
|T(f_1^{(\epsilon)},f_2^{(\epsilon)})|&\leq C_\theta M_0^{1-\theta}M_1^\theta
\|t^{-\theta}u_1^{(\epsilon)}(t)\|_{L^{v}(A_0,dt/t)}^{1-\theta}
\|t^{1-\theta}u_1^{(\epsilon)}(t)\|_{L^{v}(A_1,dt/t)}^\theta  
\\ 
&\quad \times 
\|t^{-\theta}u_2^{(\epsilon)}(t)\|_{L^{w}(B_0,dt/t)}^{1-\theta} 
\|t^{1-\theta}u_2^{(\epsilon)}(t)\|_{L^{w}(B_1,dt/t)}^\theta 
\\ 
&\leq C_\theta M_0^{1-\theta}M_1^\theta
\|t^{-\theta}u_1(t)\|_{L^{v}(A_0,dt/t)}^{1-\theta}
\|t^{1-\theta}u_1(t)\|_{L^{v}(A_1,dt/t)}^\theta  
\\ 
&\quad \times 
\|t^{-\theta}u_2(t)\|_{L^{w}(B_0,dt/t)}^{1-\theta} 
\|t^{1-\theta}u_2(t)\|_{L^{w}(B_1,dt/t)}^\theta, 
\end{align*} 
where the last inequality follows since 
$$\|u_1^{(\epsilon)}(t)\|_{A_i} \leq C\|u_1(t)\|_{A_i}, 
\|u_2^{(\epsilon)}(t)\|_{B_i} \leq C\|u_2(t)\|_{B_i}, \quad i=1, 2,  $$ 
with a constant $C$ independent of $\epsilon>0$.  
Letting $\epsilon \to 0$, we have  
\begin{align*}   
|T(f_1,f_2)| 
&\leq C_\theta M_0^{1-\theta}M_1^\theta
\|t^{-\theta}u_1(t)\|_{L^{v}(A_0,dt/t)}^{1-\theta}
\|t^{1-\theta}u_1(t)\|_{L^{v}(A_1,dt/t)}^\theta  
\\ 
&\quad \times 
\|t^{-\theta}u_2(t)\|_{L^{w}(B_0,dt/t)}^{1-\theta} 
\|t^{1-\theta}u_2(t)\|_{L^{w}(B_1,dt/t)}^\theta, 
\end{align*} 
and hence, 
taking  $u_1, u_2$  suitably, we see that  
$$|T(f_1,f_2)|\leq C_\theta M_0^{1-\theta}M_1^\theta 
\|f_1\|_{S^*(\bar{A}, v,\theta)}\|f_2\|_{S^*(\bar{B}, w,\theta)}. 
$$  
This and Lemma \ref{L1.8} below imply 
$$|T(f_1,f_2)|\leq C_\theta M_0^{1-\theta}M_1^\theta 
\|f_1\|_{\bar{A}_{\theta,v}}\|f_2\|_{\bar{B}_{\theta,w}}. 
$$  
The relation $\bar{A}_{\theta,v}=L^{s,v}$, $\bar{B}_{\theta,w}=
L^{u,w}$ claimed in the 
lemma follows from Theorem 5.3.1 of \cite{BL}.

\begin{lemma}\label{L1.8}   Let $1\leq q \leq\infty$ and $0<\theta<1$.  Then 
$\|a\|_{S^*(\bar{A}, q,\theta)} \sim \|a\|_{\bar{A}_{\theta,q}}$ for 
$a\in \Delta(\bar{A})$.    
\end{lemma}   
See \cite[Theorem 3.12.1]{BL} for the case $q<\infty$ 
and \cite[3.14.12]{BL} for the case $q=\infty$ .  
We give a proof of Lemma \ref{L1.8} 
in the following section.  

\section{Proof of Lemma $\ref{L1.8}$} \label{sec7}  

Let $\underline{S}(\bar{A}, (r_0,r_1),\theta)$ be 
the subspace of $\Sigma(\bar{A})$ of all $a\in \Sigma(\bar{A})$ 
such that 
\begin{multline}\label{e1.10}
a=a_0(t)+a_1(t) \quad \text{for every $t>0$, with}  
\\ 
\left\|t^{-\theta}a_0(t) \right\|_{L^{r_0}(A_0, dt/t)}<\infty, \quad 
 \left\|t^{1-\theta}a_1(t) \right\|_{L^{r_1}(A_1, dt/t)} <\infty.
\end{multline}  
The norm is defined as 
\begin{multline*}
\left\|a\right\|_{\underline{S}(\bar{A}, (r_0,r_1),\theta)}
\\
=\inf\left\{\left\|t^{-\theta}a_0(t) \right\|_{L^{r_0}(A_0, dt/t)} 
+\left\|t^{1-\theta}a_1(t) \right\|_{L^{r_1}(A_1, dt/t)}: 
 \text{$a_0$, $a_1$ are as in \eqref{e1.10}}\right\}.  
\end{multline*}
Let $\underline{S}(\bar{A}, q,\theta)=\underline{S}(\bar{A}, (q,q),\theta)$. 
\par We assume that 
\begin{align} \label{ee7.2} 
a_0(t)(x)&=\sum_{j=1}^M G_j(x)g_j(t), 
\\  
a_1(t)(x)&=\sum_{j=1}^M H_j(x)h_j(t),   \notag 
\end{align}  
where  $G_j, H_j \in A_0\cap A_1$ and $g_j, h_j$ are bounded measurable 
functions supported on $[\epsilon, \infty)$ and $(0, \tau]$, respectively, 
for some $\epsilon, \tau>0$.    
\par 
Let $a\in \Delta(\bar{A})$ and 
\begin{multline*}
\left\|a\right\|_{\underline{S}^*(\bar{A}, q,\theta)}
\\
=\inf\left\{\left\|t^{-\theta}a_0(t) \right\|_{L^{q}(A_0, dt/t)} 
+\left\|t^{1-\theta}a_1(t) \right\|_{L^{q}(A_1, dt/t)}: 
 (a_0, a_1) \in \mathscr{G}(a, \bar{A})\right\}, 
\end{multline*}
where 
$$
\mathscr G(a,  \bar{A})=\left\{(a_0,a_1): 
\text{$a_0$, $a_1$ are as in \eqref{ee7.2} and  
$a=a_0(s)+a_1(s)$ for all $s>0$} \right\}.   
$$ 
The conclusion of Lemma \ref{L1.8} follows from the next two results.  

\begin{equation}\label{e1.11}  
\|a\|_{S^*(\bar{A},q, \theta)}\sim 
\|a\|_{\underline{S}^*(\bar{A}, q,\theta)},    
\end{equation} 
\begin{equation}\label{e1.12}  
\|a\|_{\underline{S}^*(\bar{A}, q,\theta)}\sim \|a\|_{\bar{A}_{\theta,q}},    
\end{equation}  
where $a\in \Delta(\bar{A})$.  
\par 
We note that $\bar{A}_{\theta,q}$ is as in \cite[Chap.3]{BL}, although 
the norms of $S^*(\bar{A},q, \theta)$ and $\underline{S}^*(\bar{A}, q,\theta)$ 
are stated in expressions slightly different from those of 
$S(\bar{A},q, \theta)$ and $\underline{S}(\bar{A}, q,\theta)$, respectively.

\begin{proof}[Proof of \eqref{e1.11}]
We first prove $\|a\|_{S^*(\bar{A},q, \theta)} \gtrsim  
\|a\|_{\underline{S}^*(\bar{A}, q,\theta)}$.  
Let $a=\int_0^\infty u(s)ds/s$  
 with $u$ satisfying \eqref{ee6.2}.   If we define 
$$a_0(t)=\int_0^1u(ts)\, \frac{ds}{s}, \quad  
a_1(t)=\int_1^\infty u(ts)\, \frac{ds}{s}, $$    
then $(a_0, a_1)\in \mathscr G(a,  \bar{A})$ and 
$$\left\|t^{-\theta}a_0(t) \right\|_{L^{q}(A_0, dt/t)}
\leq \theta^{-1}\left\|t^{-\theta}u(t) \right\|_{L^{q}(A_0, dt/t)}, $$  
$$\left\|t^{1-\theta}a_1(t) \right\|_{L^{q}(A_1, dt/t)}
\leq (1-\theta)^{-1}\left\|t^{1-\theta}u(t) \right\|_{L^{q}(A_1, dt/t)}. $$  
This implies that 
$\|a\|_{\underline{S}^*(\bar{A}, q,\theta)}\leq C_\theta 
\|a\|_{S^*(\bar{A}, q,\theta)}$.  
\par 
Next, let $(a_0, a_1)\in \mathscr G(a,  \bar{A})$. 
Take $\varphi\in C_0^\infty(\Bbb R)$ such that $\supp(\varphi) 
\subset [1,2]$ and $\int_0^\infty \varphi(s)\, ds/s=1$.  Define 
$$b_j(t)=\int_0^\infty \varphi(s)a_j(ts^{-1})\, \frac{ds}{s}
=\int_0^\infty \varphi(ts^{-1})a_j(s)\, \frac{ds}{s}.  $$  
Then $a=b_0(t)+b_1(t)$, 
$b_0(t)=0$ if $t$ is small, $b_1(t)=0$ if $t$ is large.  
Also we have 
\begin{gather}\label{ee1.13}
\|t^{-\theta}tb_0'(t)\|_{L^{q}(A_0, dt/t)}\leq C 
\|t^{-\theta} a_0(t)\|_{L^{q}(A_0, dt/t)},  
\\   
\label{ee1.14}
\|t^{1-\theta}tb_1'(t)\|_{L^{q}(A_1, dt/t)}\leq C 
\|t^{1-\theta} a_1(t)\|_{L^{q}(A_1, dt/t)}.   
\end{gather}
Let 
$$u(t)=tb_0'(t)=-tb_1'(t) \in \Delta(\bar{A}). $$   
Then, $u$ is supported in a compact subinterval of $(0, \infty)$ 
and $u$ is as in \eqref{ee6.2}. 
We note that 
$$\int_0^\infty u(t)\frac{dt}{t}=\int_0^1 b_0'(t)\, dt 
-\int_1^\infty b_1'(t)\, dt=b_0(1)+b_1(1)=a.  $$  
Thus 
$$\|a\|_{S^*(\bar{A}, q,\theta)}\leq 
C\max\left(\|t^{-\theta}tb_0'(t)\|_{L^{q}(A_0, dt/t)}, 
\|t^{1-\theta}tb_1'(t)\|_{L^{q}(A_1, dt/t)}\right). $$  
By this and \eqref{ee1.13}, \eqref{ee1.14} we have 
$\|a\|_{S^*(\bar{A}, q,\theta)}\leq 
C\|a\|_{\underline{S}^*(\bar{A}, q,\theta)}$.  
This completes the proof of \eqref{e1.11}. 
\end{proof}
\begin{proof}[Proof of \eqref{e1.12}]  We first consider the case $q<\infty$. 
We easily see that 
\begin{equation*} 
\|a\|_{\underline{S}^*(\bar{A}, q,\theta)} 
\sim 
\inf_{(a_0, a_1)\in \mathscr G(a,  \bar{A})}
\left(\left\|t^{-\theta}a_0(t) \right\|_{L^{q}(A_0, dt/t)}^q 
+\left\|t^{1-\theta}a_1(t) \right\|_{L^{q}(A_1, dt/t)}^q\right)^{1/q}.  
\end{equation*}  
This implies 
\begin{align} \label{ee1.15}
\|a\|_{\underline{S}^*(\bar{A}, q,\theta)}^q 
&\gtrsim
\int_0^\infty\inf_{(a_0, a_1)\in \mathscr G(a,  \bar{A})}
\left(t^{-q\theta}\left\|a_0(t) \right\|_{A_0}^q 
+t^{q(1-\theta)}\left\|a_1(t) \right\|_{A_1}^q\right)
\,\frac{dt}{t} 
\\  \notag 
&\gtrsim \int_0^\infty \left(t^{-\theta} K(t, a; A_0, A_1) \right)^q 
\,\frac{dt}{t},   
\end{align}  
where $K$ is the functional as in \cite[Chap. 3]{BL}.  
It follows that 
\begin{equation}\label{e2.12+}
\|a\|_{(A_0, A_1)_{\theta,q}}\leq C\|a\|_{\underline{S}^*(\bar{A}, q,\theta)}. 
\end{equation}   
\par 
Next we prove the reverse inequality.  We note the following. 
\begin{align*} 
&\int_0^\infty t^{-q\theta} K(t, a; A_0, A_1)^q\, \frac{dt}{t} 
\sim \sum_{k=-\infty}^\infty 2^{-kq\theta}K(2^k, a; A_0, A_1)^q 
\\ 
& \gtrsim \sum_{k=-\infty}^\infty 2^{-kq\theta}\left(\|a_0(2^k)\|_{A_0}^q +
2^{kq}\|a_1(2^k)\|_{A_1}^q\right) -\epsilon  
\\ 
& \gtrsim \sum_k\int_0^\infty t^{-q\theta}\left(\left\|a_0(2^k)
\chi_{[2^k, 2^{k+1})}(t)\right\|_{A_0}^q 
+ t^q\left\|a_1(2^k)\chi_{[2^k, 2^{k+1})}(t)\right\|_{A_1}^q\right)
\, \frac{dt}{t} -c\epsilon 
\\ 
& \gtrsim \int_0^\infty t^{-q\theta}\left(\left\|\sum_k a_0(2^k)
\chi_{[2^k, 2^{k+1})}(t)\right\|_{A_0}^q 
+ t^q\left\|\sum_k a_1(2^k)\chi_{[2^k, 2^{k+1})}(t)\right\|_{A_1}^q\right)
\, \frac{dt}{t} -c\epsilon, 
\end{align*}  
for any $\epsilon>0$ with some $a_0(2^k)\in A_0$, $a_1(2^k)\in A_1$  
 such that $a=a_0(2^k)+ a_1(2^k)$. 
We note that $a_0(2^k), a_1(2^k) \in A_0\cap A_1$ since $a\in 
 A_0\cap A_1$. 
 Thus we have 
\begin{align*} 
&\int_0^\infty t^{-q\theta} K(t, a; A_0, A_1)^q\, \frac{dt}{t} 
\\ 
& \gtrsim \int_0^\infty \left(\left\|t^{-\theta}\sum_k a_0(2^k)
\chi_{[2^{k}, 2^{k+1})}(t)\right\|_{A_0}^q 
+ \left\|t^{1-\theta}\sum_k a_1(2^k)\chi_{[2^{k}, 2^{k+1})}(t)
\right\|_{A_1}^q\right)\, \frac{dt}{t} -c\epsilon 
\\ 
& = \lim_{M\to \infty}\int_0^\infty \left(\left\|t^{-\theta}
\sum_{k=-M}^M a_0(2^k) \chi_{[2^{k}, 2^{k+1})}(t)\right\|_{A_0}^q \right.  
\\ 
&\quad\qquad \qquad \qquad \qquad \left. + \left\|t^{1-\theta}
\sum_{k=-M}^M a_1(2^k)\chi_{[2^{k}, 2^{k+1})}(t) \right\|_{A_1}^q\right)
\, \frac{dt}{t} -c\epsilon 
\\ 
& =: I.  
\end{align*} 
To comply with the definition of the norm of 
$\underline{S}^*(\bar{A}, q,\theta)$ in \eqref{e1.12} 
(see \eqref{ee7.2}),   this  may be modified as follows. 
\begin{align*} 
I & = \lim_{M\to \infty}\int_0^\infty \left(\left\|t^{-\theta}
\left(\sum_{k=-M}^M a_0(2^k)
\chi_{[2^{k}, 2^{k+1})}(t) + a\chi_{[2^{M+1}, \infty)}(t)\right)
\right\|_{A_0}^q 
\right.  
\\ 
&\quad\left. + \left\|t^{1-\theta}\left(
\sum_{k=-M}^M a_1(2^k)\chi_{[2^{k}, 2^{k+1})}(t)
+a\chi_{(0, 2^{-M})}(t)\right) \right\|_{A_1}^q\right)\, \frac{dt}{t} 
-c\epsilon.  
\\ 
&\gtrsim \inf_{(a_0, a_1)\in \mathscr G(a,  \bar{A})}
\left(\left\|t^{-\theta}a_0(t) \right\|_{L^{q}(A_0, dt/t)}^q 
+\left\|t^{1-\theta}a_1(t) \right\|_{L^{q}(A_1, dt/t)}^q \right) -c\epsilon    
\\ 
&\gtrsim \|a\|_{\underline{S}^*(\bar{A}, q,\theta)}^q -c\epsilon,  
\end{align*}   
which implies that 
\begin{equation} \label{ee.7.9}
\int_0^\infty t^{-q\theta} K(t, a; A_0, A_1)^q\, \frac{dt}{t} 
\gtrsim \|a\|_{\underline{S}^*(\bar{A}, q,\theta)}^q.  
\end{equation}
Combining this with \eqref{e2.12+}, we have 
\begin{equation*} \label{equiv}
\|a\|_{\underline{S}^*(\bar{A}, q,\theta)}\sim \|a\|_{(A_0, A_1)_{\theta,q}}. 
\end{equation*}   
\par 
The case $q=\infty$ can be handled by an obvious modification of the 
arguments for the case $q<\infty$ as follows. 
As in \eqref{ee1.15}, we have 
\begin{align} \label{eee1.15}
\|a\|_{\underline{S}^*(\bar{A}, \infty,\theta)}  
&= \inf_{(a_0, a_1)\in \mathscr G(a,  \bar{A})}
\left(\sup_{t>0}t^{-\theta}\left\|a_0(t) \right\|_{A_0} 
+\sup_{t>0}t^{1-\theta}\left\|a_1(t) \right\|_{A_1}\right) 
\\ \notag 
&\geq 
\inf_{(a_0, a_1)\in \mathscr G(a,  \bar{A})} 
\sup_{t>0}t^{-\theta}\left(\left\|a_0(t) \right\|_{A_0}
+t\left\|a_1(t) \right\|_{A_1}\right) 
\\  \notag 
&\geq \sup_{t>0} t^{-\theta} K(t, a; A_0, A_1)
\\ 
&= \|a\|_{(A_0, A_1)_{\theta,\infty}}.   \notag  
\end{align}  
Next, we prove the reverse inequality.  As in the proof of \eqref{ee.7.9}
we have 
\begin{align*} 
&\sup_{t>0} t^{-\theta} K(t, a; A_0, A_1)  
\sim \sup_{k\in \Bbb Z} 2^{-k\theta}K(2^k, a; A_0, A_1)  
\\ 
&  \gtrsim \sup_{k\in \Bbb Z} 2^{-k\theta}\left(\|a_0(2^k)\|_{A_0} +
2^{k}\|a_1(2^k)\|_{A_1}\right) -\epsilon  
\\ 
& \gtrsim \sup_{t>0}  t^{-\theta}\left(\left\|\sum_{k\in \Bbb Z} a_0(2^k)
\chi_{[2^k, 2^{k+1})}(t)\right\|_{A_0} 
+ t\left\|\sum_{k\in \Bbb Z} 
a_1(2^k)\chi_{[2^k, 2^{k+1})}(t)\right\|_{A_1}\right) -c\epsilon. 
\end{align*}  
We modify this as follows to comply with the definition of the norm of 
$\underline{S}^*(\bar{A}, q,\theta)$ in \eqref{e1.12}. 
\begin{align*} 
&\sup_{t>0} t^{-\theta} K(t, a; A_0, A_1)  
\\ 
& \gtrsim  \lim_{M \to \infty}\sup_{t>0} \left(\left\|t^{-\theta}
\left(\sum_{k=-M}^M a_0(2^k)
\chi_{[2^{k}, 2^{k+1})}(t) + a\chi_{[2^{M+1}, \infty)}(t)\right)\right\|_{A_0}
\right.  
\\ 
&\quad\left. + \left\|t^{1-\theta}\left(
\sum_{k=-M}^M a_1(2^k)\chi_{[2^{k}, 2^{k+1})}(t)
+a\chi_{(0, 2^{-M})}(t)\right) \right\|_{A_1}\right) -c\epsilon 
\\ 
& \geq  \frac{1}{2}\lim_{M \to \infty}\left(\sup_{t>0}\left\|t^{-\theta}
\left(\sum_{k=-M}^M a_0(2^k)
\chi_{[2^{k}, 2^{k+1})}(t) + a\chi_{[2^{M+1}, \infty)}(t)\right)\right\|_{A_0}
\right.
\\ 
&\quad\left. + \sup_{t>0}\left\|t^{1-\theta}\left(
\sum_{k=-M}^M a_1(2^k)\chi_{[2^{k}, 2^{k+1})}(t)
+a\chi_{(0, 2^{-M})}(t)\right) \right\|_{A_1}\right) -c\epsilon 
\\ 
&\gtrsim \inf_{(a_0, a_1)\in \mathscr G(a,  \bar{A})}
\left(\left\|t^{-\theta}a_0(t) \right\|_{L^{\infty}(A_0, dt/t)}
+\left\|t^{1-\theta}a_1(t) \right\|_{L^{\infty}(A_1, dt/t)}\right) 
-c\epsilon    
\\ 
&= \|a\|_{\underline{S}^*(\bar{A}, \infty,\theta)} -c\epsilon   
\end{align*}   
for all $\epsilon>0$, where the first inequality holds since $0<\theta<1$. 
Thus it follows that 
\begin{equation} \label{ee.7.11}
\left\|t^{-\theta} K(t, a; A_0, A_1)\right\|_{L^\infty(dt/t)} 
\gtrsim \|a\|_{\underline{S}^*(\bar{A}, \infty,\theta)}.  
\end{equation}
By \eqref{eee1.15} and \eqref{ee.7.11} we have 
\begin{equation*} \label{equiv}
\|a\|_{\underline{S}^*(\bar{A}, \infty,\theta)}\sim 
\|a\|_{(A_0, A_1)_{\theta,\infty}}. 
\end{equation*}   
This proves \eqref{e1.12} for $q=\infty$. 

\end{proof} 
This completes the proof of Lemma \ref{L1.8}.  We refer to \cite{LP} and 
\cite[3.12]{BL} for relevant results.  

\section{Proof of Lemma  $\ref{L1.9}$}   \label{sec8}  

To prove Lemma \ref{L1.9},  we need the next two results.  
\begin{lemma}\label{L1.10} 
Let $f\in L^p+L^\infty, 1\leq p<\infty$.  Then 
\begin{equation} 
K(t, f;L^p, L^\infty) \sim \left(\int_0^{t^p}(f^*(s))^p\, ds\right)^{1/p},  
\end{equation} 
where $f^*$ denotes the nonincreasing rearrangement of $f$. 
\end{lemma} 
\begin{lemma}\label{L1.11} 
Suppose that $1\leq p \leq q \leq\infty$ and $ p <\infty$. 
Let $1/r=(1-\theta)/p$, $0<\theta<1$. Then 
\begin{equation} 
(L^p, L^\infty)_{\theta,q}= L^{r,q}, \quad \text{with equivalent norms.}
\end{equation}
\end{lemma} 
Lemma \ref{L1.10} is found in Theorem 5.2.1 of \cite{BL}. Lemma \ref{L1.11} 
is also almost in the same theorem (the case $p=q<\infty$ is not stated 
there). Here we give a proof for completeness. 
\begin{proof}[Proof of Lemma $\ref{L1.11}$]  We first consider the case 
$q<\infty$. 
Since 
$$\|f\|_{(L^{p}, L^\infty)_{\theta,q}}=\left(\int_0^\infty 
\left(t^{-\theta}K(t,f;L^p, L^\infty)\right)^q\, \frac{dt}{t}\right)^{1/q}, 
\quad p\leq q<\infty, $$ 
by Lemma \ref{L1.10} we have 
\begin{align*}  
\|f\|_{(L^{p}, L^\infty)_{\theta,q}}&\sim 
\left(\int_0^\infty 
\left(t^{-\theta p}\int_0^{t^{p}}(f^*(s))^{p}\, ds\right)^{q/p}
\, \frac{dt}{t}\right)^{1/q} 
\\ 
&= \left(\int_0^\infty 
\left(t^{-\theta p+p}\int_0^{1}(f^*(st^{p}))^{p}
\, ds\right)^{q/p}\, \frac{dt}{t}\right)^{1/q}. 
\end{align*}  
Since $q/p\geq 1$, Minkowski's inequality with changing variables implies that 
\begin{align*}    
\|f\|_{(L^{p}, L^\infty)_{\theta,q}}
&\leq C\left(\int_0^1\left(\int_0^{\infty}t^{(-\theta+1)q}(f^*(st^{p}))^{q}
\, \frac{dt}{t}\right)^{p/q}\, ds\right)^{1/p} 
\\ 
&= C\left(\int_0^1 p^{-p/q}s^{-p/r}
\left(\int_0^{\infty}t^{q/r}(f^*(t))^{q}
\, \frac{dt}{t}\right)^{p/q}\, ds\right)^{1/p} 
\\ 
\\ 
&= Cp^{-1/q}\left(\int_0^1 s^{\theta-1}\, ds\right)^{1/p}
\left(\int_0^{\infty}t^{q/r}(f^*(t))^{q}
\, \frac{dt}{t}\right)^{1/q}  
\\ 
&= Cp^{-1/q}\theta^{-1/p} \|f\|_{r,q}, 
\end{align*}   
where we have used the relation $1/r=(1-\theta)/p$. 
\par 
To prove the reverse inequality, we simply apply that $0\leq f^*(t^{p})\leq 
f^*(s)$, if $0<s \leq t^p$,  to get 
\begin{align*}  
\|f\|_{(L^{p}, L^\infty)_{\theta,q}}
&\geq C\left(\int_0^\infty 
\left(t^{-\theta p}\int_0^{t^{p}}(f^*(s))^{p}\, ds\right)^{q/p}
\, \frac{dt}{t}\right)^{1/q}  
\\ 
&\geq 
C\left(\int_0^\infty 
\left(t^{-\theta p}t^{p}(f^*(t^{p}))^{p}\right)^{q/p}
\, \frac{dt}{t}\right)^{1/q} 
\\ 
&\geq C \|f\|_{r,q}. 
\end{align*}  
\par Next we consider the case $q=\infty$. By Lemma \ref{L1.10} we see that 
\begin{align*} 
\|f\|_{(L^{p}, L^\infty)_{\theta,\infty}}&=\sup_{t>0} 
t^{-\theta}K(t,f;L^p, L^\infty)  
\\ 
&\leq C\sup_{t>0} 
t^{-\theta}\left(\int_0^{t^p}(f^*(s))^p\, ds\right)^{1/p}
\\ 
&\leq C\sup_{s>0} s^{1/r}f^*(s) \sup_{t>0} 
t^{-\theta}\left(\int_0^{t^p}s^{-p/r}\, ds\right)^{1/p}=:I.  
\end{align*}  
Using the relation $1/r=(1-\theta)/p$, we have 
\begin{align*} 
I &= C\sup_{s>0} s^{1/r}f^*(s) \sup_{t>0} 
t^{-\theta}\left(\int_0^{t^p}s^{\theta-1}\, ds\right)^{1/p}
\\ 
&= C\theta^{-1/p}\sup_{s>0} s^{1/r}f^*(s)
\\ 
&= C\theta^{-1/p}\|f\|_{r,\infty}. 
\end{align*}  
Also, 
\begin{align*} 
\|f\|_{(L^{p}, L^\infty)_{\theta,\infty}} 
&\geq C\sup_{t>0} 
t^{-\theta}\left(\int_0^{t^p}(f^*(s))^p\, ds\right)^{1/p} 
\\
&\geq C\sup_{t>0} t^{-\theta}tf^*(t^p) 
\\ 
&= C\sup_{t>0} t^{1/r}f^*(t) 
\\ 
&= C\|f\|_{r,\infty}. 
\end{align*}  
This completes the proof of the case $q=\infty$. 
  
\end{proof}

\begin{proof}[Proof of Lemma $\ref{L1.9}$]  
By Lemma \ref{L1.11}, since $r\leq q_j$, we have 
\begin{equation}\label{e1.21} 
 L^{p,q_j}=(L^r, L^\infty)_{\theta, q_j}, \quad j=0, 1. 
\end{equation} 
By Theorem 5.3.1 of \cite{BL} and Lemma \ref{L1.11}, we have  
\begin{equation}\label{e1.22} 
(L^{p,q_0}, L^{p,q_1})_{\eta, q}= L^{p,q}=(L^r,L^\infty)_{\theta, q}. 
\end{equation} 
Also, by Theorem 4.7.2 of \cite{BL}, 
\begin{equation}\label{e1.24} 
\left((L^r, L^\infty)_{\theta, q_0}, (L^r, L^\infty)_{\theta,q_1}
\right)_{[\eta]}=(L^r,L^\infty)_{\theta, q}. 
\end{equation} 
Although Theorem 4.7.2 is stated with $\theta_0, \theta_1\in (0,1)$ 
such that $\theta_0 \neq \theta_1$, we easily see that we may assume that 
$\theta_0=\theta_1=\theta$ to get the result above.  
Combining \eqref{e1.21}, \eqref{e1.22} and \eqref{e1.24},  we see that 
$$(L^{p,q_0}, L^{p,q_1})_{\eta, q}= 
\left(L^{p,q_0}, L^{p,q_1}\right)_{[\eta]}=L^{p,q}. $$     
This completes the proof of Lemma $\ref{L1.9}$.  
\end{proof} 
\par 
We give a proof of \eqref{e1.24} in Section 9 below.  

\section{Proof of \eqref{e1.24}}  \label{sec9} 
Let 
$$\ell^q_\eta=\left\{(\alpha_k)_{k\in \Bbb Z}: \alpha_k\in \Bbb C, 
k\in \Bbb Z, \|(\alpha_k)\|_{\ell^q_\eta}=\left(\sum_{k=-\infty}^\infty 
\left(2^{-\eta k}|\alpha_k|\right)^q\right)^{1/q}<\infty \right\},   $$  
where 
$1\leq q\leq \infty$,  $0< \eta<1$, with the usual modification when 
$q=\infty$.  
Let $1\leq q, q_0, q_1\leq \infty$, $1\leq r<\infty$, 
 $0<\theta, \eta<1$, 
$$\frac{1}{q}=\frac{1-\eta}{q_0}+\frac{\eta}{q_1}.$$  
Let $X_j=\bar{A}_{\theta, q_j}$,  $j=0, 1$, with $\bar{A}=(A_0, A_1)$, 
$A_0=L^r$, $A_1=L^\infty$. Let $\bar{X}=(X_0, X_1)$.   
\par 
Let $a \in \bar{A}_{\theta, q}$, $a\neq 0$. 
We show that $a \in \bar{X}_{[\eta]}$.  
By Lemma 3.2.3 and Theorem 3.3.1 of \cite{BL},  there exists  a sequence 
$(u_\nu)_{\nu\in \Bbb Z}$ in $\Delta(\bar{A})$ such that 
$a=\sum_\nu u_\nu$ in $\Sigma(\bar{A})$ and 
\begin{equation*}\label{} 
\left\|\left(J(2^\nu, u_\nu;\bar{A})\right)_\nu\right\|_{\ell_{\theta}^q} 
\leq C \|a\|_{\bar{A}_{\theta, q}}. 
\end{equation*}  
\par 
We first assume that $q<\infty$.  
For $\delta>0$ and $z\in \Bbb C$ with $0\leq \re z \leq 1$, let 
\begin{equation*}
f_\nu(z)=\left(2^{-\theta\nu}J(2^\nu,u_\nu;\bar{A})
\|a\|_{\bar{A}_{\theta, q}}^{-1}\right)^{q(1/q_1-1/q_0)(z-\eta)} u_\nu, 
\end{equation*} 
where $1/q_i=0$ if $q_i=\infty$,   
and 
$$f(z)=\exp(\delta(z-\eta)^2)\sum_\nu f_\nu(z). $$  
Then 
$$\left|\exp(-\delta(it-\eta)^2)\right|\|f(it)\|_{\bar{A}_{\theta, q_0}} 
\leq C\left\|\left(J(2^\nu, f_\nu(it);\bar{A})\right)_\nu\right\|_
{\ell_{\theta}^{q_0}} \leq C \|a\|_{\bar{A}_{\theta, q}}. 
 $$
Lemma 3.2.3 of \cite{BL} implies the first inequality. 
The second inequality can  
be seen as follows. First note that  $-\eta q(1/q_1-1/q_0)=q/q_0 -1$.  Thus, 
if $q_0<\infty$,  
\begin{align*} 
&\left\|\left(J(2^\nu, f_\nu(it);\bar{A})\right)_\nu\right\|_
{\ell_{\theta}^{q_0}}^{q_0} 
\\ 
&= \sum_\nu  2^{-\theta\nu q_0} 
2^{\theta\nu(q_0-q)}J(2^\nu, u_\nu;\bar{A})^{q-q_o}
\|a\|_{\bar{A}_{\theta, q}}^{q_0-q} J(2^\nu, u_\nu;\bar{A})^{q_o} 
\\
&= \sum_\nu 2^{-\theta q\nu}J(2^\nu, u_\nu;\bar{A})^{q}\|a\|_{\bar{A}_
{\theta, q}}^{q_0-q} 
\\ 
&\leq C\|a\|_{\bar{A}_{\theta, q}}^{q_0}.   
\end{align*}  
If $q_0=\infty$, then  $-\eta q(1/q_1-1/q_0)= -1$ and 
$$\left\|\left(J(2^\nu, f_\nu(it);\bar{A})\right)_\nu\right\|_
{\ell_{\theta}^{q_0}} = \sup_\nu 2^{-\theta\nu}2^{\theta\nu}
J(2^\nu, u_\nu;\bar{A})^{-1} \|a\|_{\bar{A}_{\theta, q}} J(2^\nu, u_\nu;\bar{A})=\|a\|_{\bar{A}_{\theta, q}}.  $$
\par 
Also, we have 
$$\left|\exp(-\delta(1+it-\eta)^2)\right|
\|f(1+it)\|_{\bar{A}_{\theta, q_1}} \leq 
C\left\|\left(J(2^\nu, f_\nu(1+it);\bar{A})\right)_\nu\right\|_
{\ell_{\theta}^{q_1}} \leq C \|a\|_{\bar{A}_{\theta, q}},  
 $$
since  $(1-\eta)q(1/q_1 - 1/q_0)=q/q_1-1$ and  
\begin{enumerate} 
\item[(a)]  if $q_1<\infty$,  
\begin{align*} 
&\left\|\left(J(2^\nu, f_\nu(1+it);\bar{A})\right)_\nu\right\|_{\ell_
{\theta}^{q_1}}^{q_1} 
\\ 
&= \sum_\nu  2^{-\theta\nu q_1} 
2^{\theta\nu(q_1-q)}J(2^\nu, u_\nu;\bar{A})^{q-q_1}
\|a\|_{\bar{A}_{\theta, q}}^{q_1-q} J(2^\nu, u_\nu;\bar{A})^{q_1} 
\\
&= \sum_\nu 2^{-\theta q\nu}J(2^\nu, u_\nu;\bar{A})^{q}
\|a\|_{\bar{A}_{\theta, q}}^{q_1-q} 
\\ 
&\leq C\|a\|_{\bar{A}_{\theta, q}}^{q_1};    
\end{align*} 
\item[(b)] 
if $q_1=\infty$, then  $(1-\eta) q(1/q_1-1/q_0)= -1$ and 
\begin{align*} 
\left\|\left(J(2^\nu, f_\nu(1+it);\bar{A})\right)_\nu\right\|_
{\ell_{\theta}^{q_1}} &= \sup_\nu 2^{-\theta\nu}2^{\theta\nu}
J(2^\nu, u_\nu;\bar{A})^{-1} \|a\|_{\bar{A}_{\theta, q}} J(2^\nu, u_\nu;\bar{A})\\ 
&=\|a\|_{\bar{A}_{\theta, q}}.  
\end{align*} 
\end{enumerate}   
Thus $f\in \mathscr F(\bar{A}_{\theta, q_0}, \bar{A}_{\theta, q_1})$ and 
$f(\eta)=a$; also $\|a\|_{\bar{X}_{[\eta]}}\leq C\|a\|_{\bar{A}_{\theta, q}}$ 
by letting $\delta \to 0$.  
\par 
If $q=\infty$, then $q_0=q_1=\infty$. Let $f_\nu(z)=u_\nu$. Then we can argue 
similarly (more directly) to the case $q<\infty$ to have the same conclusion.  
\par  
Next, we show that $a \in \bar{A}_{\theta, q}$ assuming 
$a \in \bar{X}_{[\eta]}$.  Take $f\in \mathscr F(\bar{A}_{\theta, q_0}, 
\bar{A}_{\theta, q_1})$ such that $f(\eta)=a$.  Define 
$$g(z)=2^{(z-\eta)\gamma}f(z). $$
 Then $g\in \mathscr F(\bar{A}_{\theta, q_0}, \bar{A}_{\theta, q_1})$ and 
$g(\eta)=a$. Thus $a$ can be expressed by the Poisson integral in 
$\Sigma(\bar{A})$ (see \cite[Chap. 4]{BL}):  
$$a=\int_{-\infty}^\infty P_0(\eta,t)g(it)\, dt + 
\int_{-\infty}^\infty P_1(\eta,t)g(1+it)\, dt. $$  
This proves 
\begin{align}\label{e1.30} 
&K(2^\nu, a; \bar{A})
\\ 
&\leq 2^{-\eta\gamma}
\int P_0(\eta, t)K(2^\nu, f(it);\bar{A})\, dt 
+ 2^{(1-\eta)\gamma}\int P_1(\eta, t)K(2^\nu, f(1+it);\bar{A})\, dt  \notag 
\\ 
&=2 \left(\int P_0(\eta, t)K(2^\nu, f(it);\bar{A})\, dt
\right)^{1-\eta}  \left(\int P_1(\eta, t)K(2^\nu, f(1+it);\bar{A})
\, dt\right)^{\eta},                  \notag 
\end{align} 
if $\gamma$ is chosen to satisfy 
$$2^\gamma=\left(\int P_0(\eta, t)K(2^\nu, f(it);\bar{A})\, dt\right)
\left(\int P_1(\eta, t)K(2^\nu, f(1+it);\bar{A})\, dt\right)^{-1}.       $$
Putting 
\begin{gather*}
C_\nu=2^{-\nu\theta}\int P_0(\eta, t)K(2^\nu, f(it);\bar{A})\, dt, 
\\ 
D_\nu= 2^{-\nu\theta}\int P_1(\eta, t)K(2^\nu, f(1+it);\bar{A})\, dt, 
\end{gather*} 
by Lemma 3.1.3 of \cite{BL}, 
 \eqref{e1.30}, H\"{o}lder's inequality and Minkowski's inequality, we have
\begin{align*} 
\|a\|_{\bar{A}_{\theta, q}}
&\leq 
C\left\|\left(K(2^\nu,a;\bar{A}) \right)_\nu\right\|_{\ell_{\theta}^q} 
\\ 
&\leq C\left(\sum_\nu C_\nu^{q_0}\right)^{(1-\eta)/q_0}
\left(\sum_\nu D_\nu^{q_1}\right)^{\eta/q_1} 
\\ 
&\leq C\left(\int P_0(\eta, t)\left\|\left(K(2^\nu,f(it);\bar{A}) 
\right)_\nu\right\|_{\ell_{\theta}^{q_0}} dt\right)^{1-\eta} 
\\ 
&\quad \quad \times 
\left(\int P_1(\eta, t)\left\|\left(K(2^\nu,f(1+it);\bar{A}) \right)_\nu
\right\|_{\ell_{\theta}^{q_1}} dt\right)^\eta 
\\ 
&\leq C\left(\int P_0(\eta, t)\left\|f(it)\right\|_{\bar{A}_{\theta, q_0}} dt
\right)^{1-\eta} 
\left(\int P_1(\eta, t)\left\|f(1+it)\right\|_{\bar{A}_{\theta, q_1}} dt
\right)^\eta
\\ 
&\leq C\|f\|_{\mathscr F(\bar{A}_{\theta, q_0}, \bar{A}_{\theta, q_1})}. 
\end{align*} 
This implies $\|a\|_{\bar{A}_{\theta, q}}\leq C\|a\|_{\bar{X}_{[\eta]}}$, 
completing the proof of \eqref{e1.24}.

\section{Invariant measures on homogeneous manifolds}\label{j8intro}  

Let $H$ be a closed Lie subgroup of a Lie Group $G$. We assume that $G$ 
has a countable base of open sets 
(see \cite[p. 6, p. 81]{Bo}, \cite[Chap. IV, \S 6]{M})).  
We consider the homogeneous space (the quotient manifold) 
$G/H$ (see \cite[Chap. IV]{Bo},  
\cite[Chap. IV]{M}). 
Let $\frak{g}=L(G)$ and $\frak{h}=L(H)$ be the Lie algebras of $G$ and $H$, 
respectively. We assume that $1\leq \dim L(H)< \dim L(G)$.    
Let $Ad(g)=Ad_G(g)\in GL(L(G))$ (the general linear group on $L(G)$), 
$g\in G$, be the adjoint 
representation of $G$ in $GL(L(G))$ and $Ad(h)|L(H)= Ad_H(h)\in GL(L(H))$ that 
of $H$ in $GL(L(H))$ with $h\in H$.  
For $k\in G$
$$\tau(k): G/H\to G/H$$ 
 is defined by $\tau(k)(gH)=kgH$.   
We prove the following. 
\begin{proposition}\label{j8p.1.0}
Suppose that 
$$\left|\det Ad_{G}(h)\right|= \left|\det Ad_{H}(h)\right| \quad 
\text{for all $h\in H$.}$$  
Then, there exists a $\tau(g)$-invariant Borel measure $\mu$ on 
$G/H$ $(g\in G);$ by the $\tau(g)$-invariance we mean that 
$$\mu(\tau(g)A)=\mu(A) $$ 
for every $g\in G$ and every Borel set $A$ in $G/H$. 
\end{proposition}
\par 
Let $(d\tau(h))_o$, $h\in H$, be an element of 
the linear isotropy group  of $G$ at $o=\pi(e)$, 
where $\pi : G \to G/H$ is the canonical mapping, $e$ is the identity element 
 and $$(d\tau(h))_o: T_o(G/H) \to T_o(G/H)$$ 
is the differential of $\tau(h)$ at $o$ 
with $T_o(G/H)$ denoting the tangent vector space at $o$. 
Then, under the assumption of Proposition \ref{j8p.1.0} we have 
$|\det(d\tau(h))_o|=1$ for all $h\in H$ (see Lemma \ref{j8inv.l7} below). 
We can find relevant results in \cite[Chap. 6]{KO}.  

\section{Some results for the proof of 
Proposition \ref{j8p.1.0}}\label{j8Some} 
To prove Proposition \ref{j8p.1.0} we need some preliminaries.  
Let $L(G)=\mathcal M + L(H)$ be a direct sum of vector spaces, i.e., 
$\mathcal M \cap L(H)=\{0\}$. Let $m=\dim \mathcal M$.   
 There exists an open set $U \subset \mathcal M$ 
such that $0\in U$ and,  if $V=\exp(U)$, then 
$$\pi|V : V\to G/H$$ 
is a diffeomorphism from the submanifold $V$ onto $W=\pi(V)$ 
(see \cite[Chap. IV]{M}, \cite[Chap. 3]{S}, \cite[Chap. 2, \S 4]{Hel}).   
Let $p\in \tau(k)W$, $p=kvH, v\in V$, where $v$ is uniquely determined 
by $p$ and $k$: $v=(\pi|V)^{-1}\circ \tau(k^{-1})p$. We define 
$$\varphi^{(k,p)} : \tau(k)W \to W_o=
\{\tau(v^{-1}v_0)(o): v_0\in V\}, \quad o=eH=H,  $$ 
by 
$$\varphi^{(k,p)}(r)=\tau(v^{-1}k^{-1})(r) = 
v^{-1}k^{-1}(kv_0)H=v^{-1}v_0H, $$
if $r=kv_0H \in \tau(k)W$ $(v_0\in V)$. 
We note that $\varphi^{(k,p)}(p)=eH=o$. 
\par 
Take a non-zero alternating $m$ form $\omega_o$ on $T_o(G/H)$. Define 
\begin{equation}\label{e11.1} 
\omega^{(k)}_p =\omega_o\circ \left(d\varphi^{(k,p)}\right)_p^\otimes,    
\end{equation} 
where $\left(d\varphi^{(k,p)}\right)_p^\otimes$ is defined by 
\begin{multline}\label{e11.2} 
\left(d\varphi^{(k,p)}\right)_p^\otimes (X_p^{(1)}), \dots, X_p^{(m)}) 
\\ 
=\left(\left(d\varphi^{(k,p)}\right)_p(X_p^{(1)}), \dots, 
\left(d\varphi^{(k,p)}\right)_p(X_p^{(m)}) \right)  
\end{multline}  
for $X^{(j)}_p \in T_p(G/H)$, $1\leq j \leq m$.  
A similar notation will be used in what follows. 
\par 
Let $F_k: \tau(k)W \to U$ be defined by 
$$F_k(q)= (\pi\circ\exp)^{-1}\tau(k^{-1})(q),   $$ 
where we simply write $\exp$ for $\exp|U$.    
Then we can choose $U$ so that 
$$\left\{(\tau(k)W, F_k): k\in G \right\} $$ 
is a $C^\infty$ coordinate system of $G/H$ (see \cite[Chap. 3]{S}). 
Let $F_k(r)=(x_1(r), \dots , x_m(r))$ be the local coordinates expression of 
$F_k$.  We consider, near $(x_1(p), \dots , x_m(p))$, 
$$F_e\circ \varphi^{(k,p)}\circ F_k^{-1}(x_1, \dots , x_m)= 
\left(\varphi_1^{(k,p)}(x_1, \dots, x_m), \dots , \varphi_m^{(k,p)}(x_1, \dots, x_m) \right). $$  
Then we have the following.  

\begin{proposition}\label{j8prop.1.1} 
For $p\in \tau(k)W$, $k\in G$, we have 
\begin{equation*}
\omega^{(k)}_p =J\left(\varphi_1^{(k,p)}, \dots , \varphi_m^{(k,p)}\right)
(F_k(p))\omega_o\left(\left(\frac{\partial}{\partial y_1}\right)_o, \dots, 
\left(\frac{\partial}{\partial y_m}\right)_o \right)
(dx_1)_p\wedge \dots \wedge (dx_m)_p,  
\end{equation*} 
where $J(p)
=J\left(\varphi_1^{(k,p)}, \dots , \varphi_m^{(k,p)}\right)(F_k(p))$ 
is the Jacobian$:$ 
    
\begin{equation*}
J(p)=
\det\begin{pmatrix} 
\frac{\partial\varphi_1^{(k,p)}}{\partial x_1}(F_k(p)) & \dots       &  
\frac{\partial\varphi_1^{(k,p)}}{\partial x_m}(F_k(p)) \\
\hdotsfor{3}      \\ 
\hdotsfor{3}      \\ 
\frac{\partial\varphi_m^{(k,p)}}{\partial x_1}(F_k(p)) & \dots  &    
\frac{\partial\varphi_m^{(k,p)}}{\partial x_m}(F_k(p))
\end{pmatrix}, 
\end{equation*}   
and $(y_1, \dots , y_m)$ is the local coordinates defined by $F_e(r)=
(y_1(r), \dots, y_m(r))$, $r\in W$. 
$($\!We note that $J(p)$ is $C^\infty$ on $\tau(k)W$.$)$
\end{proposition}  

\begin{proof} 
We observe that 
$$\left(d\varphi^{(k,p)}\right)_p(X_p)f
=X_p\left(f\left(\varphi^{(k,p)}\right) \right)= \sum_{i=1}^m a_i(p)
\left(\frac{\partial}{\partial x_i}\right)_pf\left(\varphi^{(k,p)}\right), 
$$ 
where $X_p \in T_p(G/H)$ with the expression 
$$X_p=  \sum_{i=1}^m a_i(p)\left(\frac{\partial}{\partial x_i}\right)_p.  $$  
We note that 
\begin{align*} 
\left(\frac{\partial}{\partial x_i}\right)_pf\left(\varphi^{(k,p)}\right)&= 
 \left[\left(\frac{\partial}{\partial x_i}(f\circ \varphi^{(k,p)}\circ F_k^{-1})  \right)\circ F_k \right](p) 
\\ 
&= \left[\left(\frac{\partial}{\partial x_i}\left((f\circ F_e^{-1})\circ(F_e    \circ \varphi^{(k,p)}\circ F_k^{-1})\right)\right)\circ F_k \right](p). 
\end{align*}
Thus by the chain rule we see that 
\begin{equation*} 
\left(\frac{\partial}{\partial x_i}\right)_pf\left(\varphi^{(k,p)}\right)= 
\sum_{l=1}^m (\frac{\partial}{\partial y_l} (f\circ F_e^{-1}))
(F_e\circ \varphi^{(k,p)})(p) 
\frac{\partial\varphi_l^{(k,p)}}{\partial x_i}(F_k(p)).  
\end{equation*}  
So we have 
\begin{align*}  
\left(d\varphi^{(k,p)}\right)_p(X_p)f 
&= \sum_{i=1}^m a_i(p) \sum_{l=1}^m \left(\frac{\partial}{\partial y_l} 
(f\circ F_e^{-1})\right)(F_e\circ \varphi^{(k,p)})(p) 
\frac{\partial\varphi_l^{(k,p)}}{\partial x_i}(F_k(p))  
\\ 
&= \sum_{l=1}^m \left(\sum_{i=1}^m a_i(p)  
\frac{\partial\varphi_l^{(k,p)}}{\partial x_i}(F_k(p))\right)
\left(\frac{\partial}{\partial y_l}\right)_{\varphi^{(k,p)}(p)}f
\\ 
\\ 
&= \sum_{l=1}^m \left(\sum_{i=1}^m a_i(p)  
\frac{\partial\varphi_l^{(k,p)}}{\partial x_i}(F_k(p))\right) 
\left(\frac{\partial}{\partial y_l}\right)_{o}f.  
\end{align*}
Therefore 
$$\left(d\varphi^{(k,p)}\right)_p(X_p)= 
 \sum_{l=1}^m \left(\sum_{i=1}^m a_i(p)  
\frac{\partial\varphi_l^{(k,p)}}{\partial x_i}(F_k(p))\right) 
\left(\frac{\partial}{\partial y_l}\right)_{o}. $$  
\par  
Let 
$$X_p^{(j)}=  \sum_{i_j=1}^m a_{i_j}^{(j)}(p)
\left(\frac{\partial}{\partial x_{i_j}}\right)_p,   $$  
$1\leq j\leq m$, be in $T_p(G/H)$.  
Then 
\begin{align*} 
&\omega_o\left(\left(d\varphi^{(k,p)}\right)_p(X_p^{(1)}), \dots, 
\left(d\varphi^{(k,p)}\right)_p(X_p^{(m)}) \right)
\\  
&= \sum_{l_1=1}^m \sum_{i_1=1}^m \dots \sum_{l_m=1}^m \sum_{i_m=1}^m 
a_{i_1}^{(1)}(p)\dots a_{i_m}^{(m)}(p)
\frac{\partial\varphi_{l_1}^{(k,p)}}{\partial x_{i_1}}(F_k(p))\dots 
\frac{\partial\varphi_{l_m}^{(k,p)}}{\partial x_{i_m}}(F_k(p)) 
\\ 
&\qquad \times \omega_o\left(\left(\frac{\partial}{\partial y_{l_1}}
\right)_{o},\dots, \left(\frac{\partial}{\partial y_{l_m}}\right)_{o} \right)
\\ 
&= \sum_{i_1=1}^m \dots\sum_{i_m=1}^m a_{i_1}^{(1)}(p)\dots a_{i_m}^{(m)}(p)
\det\begin{pmatrix} 
\frac{\partial\varphi_1^{(k,p)}}{\partial x_{i_1}}(F_k(p)) & \dots       &  
\frac{\partial\varphi_1^{(k,p)}}{\partial x_{i_m}}(F_k(p)) \\
\hdotsfor{3}      \\ 
\hdotsfor{3}      \\ 
\frac{\partial\varphi_m^{(k,p)}}{\partial x_{i_1}}(F_k(p)) & \dots  &    
\frac{\partial\varphi_m^{(k,p)}}{\partial x_{i_m}}(F_k(p))
\end{pmatrix} 
\\ 
&\qquad \times \omega_o\left(\left(\frac{\partial}{\partial y_1}\right)_o, 
\dots, \left(\frac{\partial}{\partial y_m}\right)_o \right) 
\\ 
&=\det\begin{pmatrix} 
a_1^{(1)}(p) & \dots       &   a_1^{(m)}(p)\\
\hdotsfor{3}      \\ 
\hdotsfor{3}      \\ 
a_m^{(1)}(p) & \dots  &   a_m^{(m)}(p)
\end{pmatrix} 
J\left(\varphi_1^{(k,p)}, \dots , \varphi_m^{(k,p)}\right)(F_k(p))
\\ 
&\qquad\qquad\qquad\qquad\qquad\qquad\qquad\qquad\qquad \times 
\omega_o\left(\left(\frac{\partial}{\partial y_1}\right)_o, 
\dots, \left(\frac{\partial}{\partial y_m}\right)_o \right).   
\end{align*}
On the other hand, we also see that 
\begin{align*}
&\left((dx_1)_p\wedge \dots \wedge (dx_m)_p\right)
\left(X_p^{(1)}, \dots, X_p^{(m)}\right)
\\ 
&=\det\begin{pmatrix} 
a_1^{(1)}(p) & \dots       &   a_1^{(m)}(p)\\
\hdotsfor{3}      \\ 
\hdotsfor{3}      \\ 
a_m^{(1)}(p) & \dots  &   a_m^{(m)}(p)
\end{pmatrix} 
\left((dx_1)_p\wedge \dots \wedge (dx_m)_p\right)
\left(\left(\frac{\partial}{\partial x_1}\right)_p, 
\dots, \left(\frac{\partial}{\partial x_m}\right)_p \right)
\\ 
&=\det\begin{pmatrix} 
a_1^{(1)}(p) & \dots       &   a_1^{(m)}(p)\\
\hdotsfor{3}      \\ 
\hdotsfor{3}      \\ 
a_m^{(1)}(p) & \dots  &   a_m^{(m)}(p)
\end{pmatrix}.  
\end{align*}
Thus we have 
\begin{align*} 
&\omega_o\circ \left(d\varphi^{(k,p)}\right)_p^\otimes\left(X_p^{(1)}, \dots, 
X_p^{(m)}\right)  
\\
&=\omega_o\left(\left(d\varphi^{(k,p)}\right)_p(X_p^{(1)}), \dots, 
\left(d\varphi^{(k,p)}\right)_p(X_p^{(m)}) \right)
\\  
&=J\left(\varphi_1^{(k,p)}, \dots , \varphi_m^{(k,p)}\right)(F_k(p))
\\ 
&\qquad \times \omega_o\left(\left(\frac{\partial}{\partial y_1}\right)_o, 
\dots, \left(\frac{\partial}{\partial y_m}\right)_o \right)
\left((dx_1)_p\wedge \dots \wedge (dx_m)_p\right)
\left(X_p^{(1)}, \dots, X_p^{(m)}\right) 
\end{align*} 
for all  $X^{(j)}_p $, $1\leq j\leq m$.    
This completes the proof of Proposition \ref{j8prop.1.1}. 
\end{proof} 
\par 
We write 
$$\omega_p^{(k)}= h^{(k)}(p)\left((dx_1)_p\wedge \dots \wedge (dx_m)_p\right), 
\quad p\in \tau(k)W.   $$  
\par 
Suppose that $p \in (\tau(k_1)W)\cap (\tau(k_2)W)$, $p=k_1v_1H$, $p=k_2v_2H$, 
$k_1, k_2 \in G$, $v_1, v_2\in V$.   
Let $g_1=k_1v_1$,  $g_2=k_2v_2$.   Then  $g_2^{-1}g_1\in H$ and 
for  $r \in (\tau(k_1)W)\cap (\tau(k_2)W)$ we have 
$$\tau(g_2^{-1}g_1)\circ  \varphi^{(k_1,p)}(r) = \varphi^{(k_2,p)}(r). $$
This is true  since 
$$\tau(g_2^{-1}g_1)\circ  \varphi^{(k_1,p)}(r)= 
\tau(g_2^{-1}g_1)(\tau(g_1^{-1})(r))=\tau(g_2^{-1})(r)= \varphi^{(k_2,p)}(r). $$From this we have the following.  

\begin{lemma}\label{j8inv.l1} 
Let $p, g_1, g_2$ be as above and 
$h=g_2^{-1}g_1 \in H$. Then 
$$(d\varphi^{(k_2,p)})_p=(d\tau(h))_o\circ(d\varphi^{(k_1,p)})_p. $$ 
\end{lemma}  
 
Lemma \ref{j8inv.l1} implies 
\begin{multline*} 
\omega_o\left(\left(d\varphi^{(k_2,p)}\right)_p(X_p^{(1)}), \dots, 
\left(d\varphi^{(k_2,p)}\right)_p(X_p^{(m)}) \right)
\\ 
=\det((d\tau(h))_o)
\omega_o\left(\left(d\varphi^{(k_1,p)}\right)_p(X_p^{(1)}), \dots, 
\left(d\varphi^{(k_1,p)}\right)_p(X_p^{(m)}) \right) 
\end{multline*} 
for  $p \in (\tau(k_1)W)\cap (\tau(k_2)W)$, $h=g_2^{-1}g_1$, 
$g_1=k_1(\pi|V)^{-1}\tau(k_1^{-1})p$, $g_2=k_2(\pi|V)^{-1}\tau(k_2^{-1})p$, 
where  
$\det((d\tau(h))_o)$ is the determinant of an $m\times m$ matrix expressing the linear transformation $(d\tau(h))_o$ on $T_o(G/H)$.  
\par 
Thus we have the following.  

\begin{lemma}\label{j8inv.l2} 
Let $F_{k_2}(r)=(y_1(r), \dots, y_m(r))$, $F_{k_1}(r)=(x_1(r), \dots, x_m(r))$. Then for $p \in (\tau(k_1)W)\cap (\tau(k_2)W)$ we have 
$$h^{(k_2)}(p)\left((dy_1)_p\wedge \dots \wedge (dy_m)_p\right) =
\det((d\tau(h))_o)h^{(k_1)}(p)\left((dx_1)_p\wedge \dots 
\wedge (dx_m)_p\right).
$$
\end{lemma} 

Let  $U_1=F_{k_1}((\tau(k_1)W)\cap (\tau(k_2)W))$, $U_2=F_{k_2}((\tau(k_1)W)\cap (\tau(k_2)W))$.  
We consider $\psi : U_2 \to U_1$ defined by $\psi= F_{k_1}\circ F_{k_2}^{-1}$. 
We write 
$$\psi(y_1, \dots, y_m)=(\psi^{(1)}(y_1, \dots, y_m), \dots , 
\psi^{(m)}(y_1, \dots, y_m)).  $$ 
Also, we write $J(\psi)(F_{k_2}(p))$ for the Jacobian 
$J(\psi^{(1)}, \dots, \psi^{(m)})(F_{k_2}(p))$. 
The following holds. 

\begin{lemma}\label{j8inv.l3} 
Let $p \in ((\tau(k_1)W)\cap (\tau(k_2)W))$. Then, with the notation of 
Lemma $\ref{j8inv.l2}$, we have  
$$ (dx_1)_p\wedge \dots \wedge (dx_m)_p =J(\psi)(F_{k_2}(p))
(dy_1)_p\wedge \dots \wedge (dy_m)_p. $$  
\end{lemma} 
Lemmas \ref{j8inv.l2} and \ref{j8inv.l3} imply the following. 

\begin{lemma}\label{j8inv.l4} Using the notations of Lemmas $\ref{j8inv.l2}$  
and $\ref{j8inv.l3}$, we have 
$$h^{(k_2)}(p) =\det((d\tau(h))_o)h^{(k_1)}(p)J(\psi)(F_{k_2}(p))
$$
on $((\tau(k_1)W)\cap (\tau(k_2)W))$.  

\end{lemma} 

\section{Proof of Proposition $\ref{j8p.1.0}$}  \label{sec12}
We define a measure $\mu$ on $M=G/H$. 
Let $K$ be a compact set in $M$. Let $\frak M_K$ be the collection 
of the Borel sets of $M$ contained in $K$. 
We first define a measure $\mu_K$ on $\frak M_K$.  
We decompose $K=\cup_{j=1}^l E_j$, where $\{E_j\}$ is a 
family of mutually disjoint Borel sets in $\frak M_K$ such that 
$E_j\subset \tau(k_j)W$ for some $k_j\in G$. 
This is possible since $K$ is compact and hence we can select
 a finite subcovering from the covering $\{\tau(k)W\}_{k\in K}$ of $K$.  
 For $A \in \frak M_K$, define $A_j=A\cap E_j$, $1\leq j\leq l$, and 
\begin{equation}\label{j8e3.0}
\mu_K(A)=\sum_{j=1}^l \int_{A_j}|\omega_p^{(k_j)}|
=\sum_{j=1}^l \int_{A_j}\left| h^{(k_j)}(p)\right| \left|(dx_1)_p\wedge \dots 
\wedge (dx_m)_p\right|,  
\end{equation}
where  $F_{k_j}(p)=(x_1(p), \dots, x_m(p))$ and 
$$I:=\int_{A_j}\left| h^{(k_j)}(p)\right| \left|(dx_1)_p\wedge \dots 
\wedge (dx_m)_p\right|
=\int_{F_{k_j}(A_j)}\left| h^{(k_j)}(F_{k_j}^{-1}(x))\right| \, dx_1 \dots 
dx_m. $$ 
\par 
We see that $\mu_K$ is well-defined. Let 
$K=\cup_{u=1}^r D_u$, where $\{D_u\}$ is another family of mutually disjoint 
Borel sets in $K$ such that $D_u\subset \tau(s_u)W$ for some $s_u\in G$. 
Then  
\begin{equation*}
I=\sum_{u=1}^r 
\int_{F_{k_j}(A_j\cap D_u)}\left| h^{(k_j)}(F_{k_j}^{-1}(x))\right| \, dx_1 
\dots dx_m
\end{equation*}
for each $j$.  By change of variables $x=F_{k_j}\circ F_{s_u}^{-1}(y)$, writing $\psi=F_{k_j}\circ F_{s_u}^{-1}$,  we see that 
\begin{align*}
I&=\sum_{u=1}^r 
\int_{F_{s_u}(A_j\cap D_u)}\left| h^{(k_j)}(F_{s_u}^{-1}(y))\right| 
|J(\psi)(y)|  \, dy_1 \dots dy_m 
\\ 
&=\sum_{u=1}^r 
\int_{F_{s_u}(A_j\cap D_u)}\left| h^{(s_u)}(F_{s_u}^{-1}(y))\right| 
 \, dy_1 \dots dy_m, 
\end{align*}
where the last equality follows from Lemma \ref{j8inv.l4} and the fact  
$|\det((d\tau(h))_o)|=1$, which follows from Lemma \ref{j8inv.l7} below and 
 the assumption of the proposition.  
Thus 
\begin{align*}
&\sum_{j=1}^l \int_{A_j}\left| h^{(k_j)}(p)\right| \left|(dx_1)_p\wedge \dots 
\wedge (dx_m)_p\right|
= \sum_{u=1}^r\sum_{j=1}^l \int_{F_{s_u}(A_j\cap D_u)}
\left| h^{(s_u)}(F_{s_u}^{-1}(y))\right| \, dy_1 \dots dy_m  
\\ 
&= \sum_{u=1}^r \int_{F_{s_u}(A\cap D_u)}
\left| h^{(s_u)}(F_{s_u}^{-1}(y))\right| \, dy_1 \dots dy_m  
\\ 
&= \sum_{u=1}^r \int_{A\cap D_u}
\left| h^{(s_u)}(p)\right| \, 
\left|(dy_1)_p\wedge \dots \wedge (dy_m)_p\right|.   
\end{align*} 
This implies that $\mu_K$ is well-defined.  
\par 
We see the countable additivity of 
$\mu_K$. 
Let $A=\cup_{a=1}^\infty B_a$ 
be a union of disjoint sets $B_a$ in  $\frak M_K$.  
By the definition \eqref{j8e3.0} and the countable additivity 
of the Lebesgue measure,   we have 
\begin{align*} 
\mu_K(A)&=\sum_{j=1}^l \int_{A_j}|\omega_p^{(k_j)}| 
\\ 
&=\sum_{j=1}^l\int_{F_{k_j}(A_j)}\left| h^{(k_j)}(F_{k_j}^{-1}(x))\right| \, dx_1 \dots dx_m 
\\ 
&=\sum_{j=1}^l\sum_{a=1}^\infty \int_{F_{k_j}(B_a\cap E_j)}
\left| h^{(k_j)}(F_{k_j}^{-1}(x))\right| 
\, dx_1 \dots dx_m  
\\ 
&=\sum_{a=1}^\infty\sum_{j=1}^l \int_{F_{k_j}(B_a\cap E_j)}
\left| h^{(k_j)}(F_{k_j}^{-1}(x))\right| 
\, dx_1 \dots dx_m 
\\ 
&=\sum_{a=1}^\infty \mu_K(B_a),    
\end{align*} 
which is what we need to prove.   
\par 
Let $\frak M$ be the Borel $\sigma$ algebra on $G/H$. We now define a measure 
$\mu$ on $\frak M$. For $E\in \frak M$, let 
\begin{equation}\label{j8e3.2+}  
\mu(E)=\sup_K \mu_K(E\cap K), 
\end{equation}  
where the supremum is taken over all compact sets $K$.   
To show that $\mu$ is a measure on $\frak M$, we prove that $\mu$ 
is countably additive on  $\frak M$.  Let $E=\cup_{j=1}^\infty E_j$ 
be a union of disjoint sets $E_j$ in  $\frak M$.  We prove 
\begin{equation} \label{j8e3.1} 
\mu(E)= \sum_{j=1}^\infty \mu(E_j).  
\end{equation} 
By the countable additivity of $\mu_K$, we easily see that 
\begin{align}\label{j8e3.2} 
\mu(E)&=\sup_K\mu_K(E\cap K)= \sup_K \sum_{j=1}^\infty \mu_K(E_j\cap K)
\\ 
&\leq  \sum_{j=1}^\infty \sup_K\mu_K(E_j\cap K)= \sum_{j=1}^\infty \mu(E_j). 
\notag 
\end{align}
To prove the reverse inequality, without loss of generality, 
we may assume that $\mu(E_j)<\infty$ for every $j$. 
For any $\epsilon>0$ and a positive 
integer $j$, there exists a compact set $K_j$ such that 
$$\mu(E_j) - 2^{-j}\epsilon < \mu_{K_j}(E_j\cap K_j).  $$  
Thus for any positive integer $N$ we have 
\begin{align}\label{j8e3.3} 
\mu(E)&= \sup_K \sum_{j=1}^\infty \mu_K(E_j\cap K)
\\ 
&\geq \sup_K\sum_{j=1}^N \mu_K(E_j\cap K) \geq \sum_{j=1}^N
\mu_{\cup_{l=1}^N K_l}(E_j\cap  \cup_{l=1}^N K_l)
\notag 
\\ 
&\geq \sum_{j=1}^N
\mu_{\cup_{l=1}^N K_l}(E_j\cap K_j)=\sum_{j=1}^N
\mu_{K_j}(E_j\cap K_j) 
\notag 
\\ 
&\geq \sum_{j=1}^N\mu(E_j) -\epsilon, 
\notag 
\end{align}
where we have used an easily observable fact that 
if $K$, $L$ are compact sets and 
$K\subset L$, then 
\begin{equation}\label{j8e3.5}
\mu_K(E)=\mu_{L}(E) \qquad \text{for $E\in \frak M_K$.} 
\end{equation}  
Letting $N\to \infty$ and $\epsilon \to 0$ in \eqref{j8e3.3}, 
we have the reverse inequality of \eqref{j8e3.2}, which completes 
the proof of \eqref{j8e3.1}. 
\par 
\begin{remark}\label{re12.1} 
Here we note that 
\begin{equation}\label{j8e3.6}
\mu(E)=\mu_K(E) \qquad \text{if $E\in \frak M_K$. }
\end{equation} 
This can be seen as follows. By the definition of $\mu(E)$, 
if $E\in \frak M_K$, we have 
\begin{equation}\label{j8e3.7}
\mu(E)=\sup_{K'} \mu_{K'}(E\cap K')\geq \mu_K(E\cap K)=\mu_K(E),  
\end{equation}  
where in the supremum $K'$ ranges over the collection of all compact sets.  
On the other hand, by \eqref{j8e3.5}  
$$\mu_{K'}(E\cap K')=\mu_{K'\cup K}(E\cap K')=\mu_{K}(E\cap K')\leq 
\mu_{K}(E).  $$ 
Thus, taking supremum in $K'$, we have the reverse inequality of 
\eqref{j8e3.7} and hence we have \eqref{j8e3.6}.  
\end{remark}  
\par
Next, we show that $\mu$ is $\tau(g)$-invariant.  
\begin{lemma}\label{j8inv.l5}
If $p\in \tau(k)W$, then 
$$\omega_p^{(k)}=\omega_{\tau(g)p}^{(gk)}\circ (d\tau(g))_p^\otimes,  $$ 
where $(d\tau(g))_p^\otimes$ is defined from $(d\tau(g))_p$ as 
$\left(d\varphi^{(k,p)}\right)_p^\otimes$ is defined by 
$\left(d\varphi^{(k,p)}\right)_p$ in \eqref{e11.2}. 
\end{lemma}

\begin{proof}  
By the definition \eqref{e11.1} we have 
$$\omega_{\tau(g)p}^{(gk)}=\omega_o\circ 
(d\varphi^{(gk, \tau(g)p)})_{\tau(g)p}^\otimes.  $$
If $p=kvH$, $v\in V$, 
$$ \varphi^{(gk, \tau(g)p)}(r)=\tau(v^{-1}k^{-1}g^{-1})(r)
=\tau(v^{-1}k^{-1})\tau(g^{-1})(r), $$  
and hence 
$$ (\varphi^{(gk, \tau(g)p)}\circ \tau(g))(r)=\tau(v^{-1}k^{-1})(r)
=\varphi^{(k, p)}(r),  $$ 
which implies 
 $$ (d\varphi^{(gk, \tau(g)p)})_{\tau(g)p}\circ (d\tau(g))_p 
=(d\varphi^{(k, p)})_p.     $$ 
Therefore, we easily see that  
$$\omega_{p}^{(k)}=\omega_o\circ 
(d\varphi^{(k,p)})_{p}^\otimes = \omega_o\circ 
(d\varphi^{(gk, \tau(g)p)})_{\tau(g)p}^\otimes\circ (d\tau(g))_p^\otimes
=\omega_{\tau(g)p}^{(gk)}\circ (d\tau(g))_p^\otimes.  $$

\end{proof}
Let  $U_1=F_{k}(\tau(k)W)$, $U_2=F_{gk}(\tau(gk)W)$, $F_{k}(p)=(x_1(p), \dots, x_m(p))$, $F_{gk}(q)=(y_1(q), \dots, y_m(q))$. 
Let $\Psi : U_1 \to U_2$ be defined by 
$$\Psi(x)= \left(F_{gk}\circ \tau(g)\circ F_{k}^{-1}\right)(x)= 
\left(\Psi^{(1)}(x), \dots, \Psi^{(m)}(x)\right).  
$$  
We recall that 
$$\omega_p^{(k)}=h^{(k)}(p)\left((dx_1)_p\wedge \dots \wedge (dx_m)_p\right), 
\qquad 
\omega_q^{(gk)}=h^{(gk)}(q)\left((dy_1)_q\wedge \dots \wedge (dy_m)_q\right)
   $$  
for $p\in \tau(k)W$, $q\in \tau(gk)W$.   

\begin{lemma}\label{j8inv.l6}   
For $p\in \tau(k)W$, we have 
$$ 
h^{(k)}(p)\left((dx_1)_p\wedge \dots \wedge (dx_m)_p\right) 
=h^{(gk)}(\tau(g)p)J(\Psi)(F_k(p))
\left((dx_1)_p\wedge \dots \wedge (dx_m)_p\right).  
$$ 
\end{lemma} 
\begin{proof} 
Since 
$$ 
\omega_{\tau(g)p}^{(gk)}=h^{(gk)}(\tau(g)p)
\left((dy_1)_{\tau(g)p}\wedge \dots \wedge (dy_m)_{\tau(g)p}\right), 
$$ 
we see that 
\begin{align*} 
&\left(\omega_{\tau(g)p}^{(gk)}\circ (d\tau(g))_p^\otimes\right)
\left(\left(\frac{\partial}{\partial x_1}\right)_p, \dots , 
\left(\frac{\partial}{\partial x_m}\right)_p\right)
\\ 
&= \omega_{\tau(g)p}^{(gk)}
\left((d\tau(g))_p\left(\frac{\partial}{\partial x_1}\right)_p, \dots , 
(d\tau(g))_p\left(\frac{\partial}{\partial x_m}\right)_p\right)
\\ 
&= \omega_{\tau(g)p}^{(gk)}
\left(\sum_{l_1=1}^m 
\frac{\partial \Psi^{(l_1)}}{\partial x_1}(F_k(p))
\left(\frac{\partial}{\partial y_{l_1}}\right)_{\tau(g)p}, \dots , 
\sum_{l_m=1}^m 
\frac{\partial \Psi^{(l_m)}}{\partial x_m}(F_k(p))
\left(\frac{\partial}{\partial y_{l_m}}\right)_{\tau(g)p}
\right) 
\\ 
&= J(\Psi)(F_k(p)) h^{(gk)}(\tau(g)p).  
\end{align*}
This implies 
$$\omega_{\tau(g)p}^{(gk)}\circ (d\tau(g))_p^\otimes
=J(\Psi)(F_k(p)) h^{(gk)}(\tau(g)p)
\left((dx_1)_p\wedge \dots \wedge (dx_m)_p\right).  $$
Combining this with Lemma \ref{j8inv.l5}, we get the conclusion.  
\end{proof} 

\par 
Using Lemma \ref{j8inv.l6}, we can prove that $\mu$ is $\tau(g)$-invariant.  
Let $K=\cup_{j=1}^l E_j$, where $K$ is a compact set in $M$ and 
 $\{E_j\}$ is a family of mutually disjoint 
Borel sets such that $E_j\subset \tau(k_j)W$.  Then $\tau(g)K$ is compact and 
 $\tau(g)K=\cup_{j=1}^l \tau(g)E_j$, $\tau(g)E_j\subset \tau(gk_j)W$ and 
$\{\tau(g)E_j\}$ is a family of mutually disjoint Borel sets. 
Let  $U_1=F_{k_j}(\tau(k_j)W)$, $U_2=F_{gk_j}(\tau(gk_j)W)$, 
$F_{k_j}(p)=(x_1(p), \dots, x_m(p))$, $F_{gk_j}(q)=(y_1(q), \dots, y_m(q))$. 
Let $\Psi_j : U_1 \to U_2$ be defined by 
$$\Psi_j= F_{gk_j}\circ \tau(g)\circ F_{k_j}^{-1}.  
$$   
Let $A\in \frak M_K$. We first prove that 
\begin{equation} \label{j8e3.9}  
\mu_{\tau(g)K}(\tau(g)A)= \mu_{K}(A).  
\end{equation}  
By the definition of the measure $\mu_{\tau(g)K}$, we have 
\begin{align*} 
\mu_{\tau(g)K}(\tau(g)A)
&=\sum_{j=1}^l \int_{(\tau(g)E_j)\cap(\tau(g)A)}|\omega_q^{(gk_j)}| 
\\ 
&=\sum_{j=1}^l \int_{\tau(g)(E_j\cap A)}\left| h^{(gk_j)}(q)\right| \left|(dy_1)_q\wedge \dots \wedge (dy_m)_q\right| 
\\ 
&=\sum_{j=1}^l 
\int_{F_{gk_j}(\tau(g)(E_j\cap A))}
\left| h^{(gk_j)}(F_{gk_j}^{-1}(y))\right| 
\, dy_1 \dots dy_m. 
\end{align*} 
Applying change of variables $y=\Psi_j(x)$, this is equal to  
\begin{align*} 
&\sum_{j=1}^l \int_{F_{k_j}(E_j\cap A)}\left| h^{(gk_j)}(\tau(g)\circ 
F_{k_j}^{-1}(x))\right| |J(\Psi_j)(x)|\, dx_1 \dots dx_m 
\\ 
&= \sum_{j=1}^l \int_{F_{k_j}(E_j\cap A)}\left| h^{(k_j)}(F_{k_j}^{-1}(x))
\right| 
\, dx_1 \dots dx_m 
\\ 
&= \sum_{j=1}^l \int_{E_j\cap A}\left| h^{(k_j)}(p)\right| 
\, \left|(dx_1)_p\wedge \dots \wedge (dx_m)_p\right|,  
\end{align*} 
where the first equality follows from Lemma \ref{j8inv.l6}.  
Thus 
$$ \mu_{\tau(g)K}(\tau(g)A)=\sum_{j=1}^l 
\int_{(\tau(g)E_j)\cap (\tau(g)A)}|\omega_q^{(gk_j)}| 
= \sum_{j=1}^l \int_{E_j\cap A}|\omega_p^{(k_j)}| 
=\mu_K(A). 
$$ 
This proves \eqref{j8e3.9}.  
\par 
Let $A\in \frak M$. By the definition of $\mu$ in \eqref{j8e3.2+}, since 
the mapping $K\to \tau(g)K$ is a bijection from the set of all compact sets in 
$G/H$ to itself,  we have 
\begin{align*}  
\mu(\tau(g)A) &= \sup_K \mu_K((\tau(g)A)\cap K) 
\\ 
&= \sup_K \mu_{\tau(g)K}((\tau(g)A)\cap (\tau(g)K)) 
\\ 
&= \sup_K \mu_{\tau(g)K}(\tau(g)(A\cap K)) 
\\ 
&= \sup_K \mu_{K}(A\cap K) 
\\ 
&= \mu(A),     
\end{align*} 
where the penultimate equality follows from \eqref{j8e3.9}. 
This means that $\mu$ is $\tau(g)$-invariant, which completes the proof of 
Proposition \ref{j8p.1.0}.  

\section{Calculation of $\det (d\tau(h))_o$}\label{j8cal}

Here we state and prove a result used above.  Some relevant results can be 
found in \cite[Chap. IV]{M}. 
\begin{lemma}\label{j8inv.l7}    
Let $(d\tau(h))_o$, $h\in H$,  be an element of the linear isotropy group 
of $G$ at $o=\pi(e)$, 
where $\pi : G \to G/H$, $\tau(h): G/H\to G/H$ and  $(d\tau(h))_o: T_o(G/H) \to T_o(G/H)$ are as in Section $\ref{j8intro}$.  
Then 
$$\det \left((d\tau(h))_o\right) 
= \det (dA_h)_e/ \det \left((dA_h)_e|L(H)_e\right), 
$$ 
where $A_h:G\to G$ is defined by $A_h(u)=huh^{-1}$ and $(dA_h)_e: T_e(G) \to 
T_e(G)$, $(dA_h)_e|L(H)_e: T_e(H) \to T_e(H)$ are the differential and its 
restriction to $L(H)_e$; we can 
identify  $(dA_h)_e$ and $(dA_h)_e|L(H)_e$ with $Ad_G(h)$ and $Ad_H(h)$ of 
Proposition $\ref{j8p.1.0}$ by the relations 
$(Ad_G(h)X)_e=(dA_h)_e(X_e)$ and 
$(Ad_H(h)Y)_e=((dA_h)_e|L(H)_e)(Y_e)$, respectively 
$($\!we can identify $T_e(G)$ 
and $T_e(H)$ with $L(G)_e$ and $L(H)_e$, respectively$)$.   
\end{lemma} 
\begin{proof} 
Consider 
$$(d\pi)_e : T_e(G) \to T_o(G/H). $$  
We see that $\mathrm{ker}  (d\pi)_e = L(H)_e=\{(di)_e Y_e: Y_e\in T_e(H)\}$,  
where  $i: H\to G$ is defined by $i(x)=x$. 
\par 
Let $X_e\in L(H)_e$.  Then $X_e= (di)_e Y_e$ with $Y\in L(H)$. Thus 
$$(d\pi)_eX_e=(d\pi)_e((di)_e Y_e)= ((d\pi)_e\circ (di)_e)(Y_e)= 
(d(\pi\circ i))_e(Y_e)=0, $$
since $\pi\circ i : H \to \{o\}$, and hence $(d(\pi\circ i))_e=0$ on 
$T_e(H)$.  Therefore $L(H)_e\subset \mathrm{ker}  (d\pi)_e$.  
\par 
Next, let $X_e \in T_e(G)$ and $(d\pi)_eX_e=0$.  
Let $L(G)= \mathcal M + L(H)$ be a direct sum as in Section \ref{j8Some}. 
We take a basis  $\{X_j\}_{j=1}^l$, $l=m+n$, of $L(G)$ 
such that $\{X_{j}\}_{j=1}^m$ and $\{X_{m+j}\}_{j=1}^n$ are  bases of 
$\mathcal M$ and $L(H)$, respectively.  
Write $X=\sum_{j=1}^l c_j X_j$. 
 Let $U$ be a small open neighborhood in $\mathcal M$ of $0$ 
and $V=\exp(U)$ in $G$ as above. Let $f$ be a $C^\infty$ function on 
a neighborhood of $o$ in $G/H$.   
Then 
\begin{equation}\label{j8e.1} 
 ((d\pi)_eX_e)f= \sum_{j=1}^l c_j (X_j)_e (f\circ \pi) = 
\sum_{j=1}^l c_j\left[\frac{d}{dt}f(\pi(\exp(tX_j)))\right]_{t=0}. 
\end{equation}  
If $j\geq m+1$, $\exp(tX_j)\in H$ and so $\pi(\exp(tX_j))=o$ and if 
$1\leq j\leq m$, $\pi(\exp(tX_j))=(\pi|V)(\exp(tX_j))$ if $|t|$ is 
small enough.  
Using these observations in \eqref{j8e.1}, we see that 
\begin{equation}\label{j8e.e.2} 
 ((d\pi)_eX_e)f= 
\sum_{j=1}^m c_j\left[\frac{d}{dt}f((\pi|V)(\exp(tX_j)))\right]_{t=0}. 
\end{equation}  
Let $(x^1, \dots , x^l)$ be the canonical coordinate about $e$ 
with respect to $\{X_j\}_{j=1}^l$, then 
$$ x^{k}(\exp(tX_k))=tc_k, \quad x^{k}(\exp(tX_j))=0 \quad (j\neq k), $$  
if $|t|$ is small enough. 
Thus, taking $x^k\circ (\pi|V)^{-1}$ for $f$ in \eqref{j8e.e.2}, we have 
\begin{equation*}\label{j8e1}
0= ((d\pi)_eX_e)\left(x^k\circ (\pi|V)^{-1}\right)=
\sum_{j=1}^m c_j\left[\frac{d}{dt}x^k(\exp(tX_j))\right]_{t=0}=c_k 
\end{equation*}  
for $1\leq k\leq m$. 
Therefore we have $X_e=\sum_{j=m+1}^l c_j (X_j)_e \in L(H)_e$. 
\par 
We have established that $\mathrm{ker}  (d\pi)_e = L(H)_e$.   
Using this, we can define the mapping 
$(d\pi)_e^\dagger : T_e(G)/L(H)_e \to T_o(G/H)$ which is 
an isomorphism between vector spaces, by $(d\pi)_e^\dagger(X_e+L(H)_e) 
= (d\pi)_e(X_e)$. 
\par  
Let $h\in H$.  We see that $(dA_h)_eY_e \in L(H)_e$ 
if $Y_e\in L(H)_e$.   
 Therefore we can define 
$$(dA_h)_e^\dagger: L(G)_e/L(H)_e \to  L(G)_e/L(H)_e$$
 by $(dA_h)_e^\dagger(X_e+ L(H)_e)= (dA_h)_e(X_e) + L(H)_e$.  
\par 
 Define $P_e : L(G)_e \to L(G)_e/L(H)_e$, 
by $P_e(X_e)= X_e+ L(H)_e$.  Then, obviously, 
\begin{equation}\label{j8e.2}
(dA_h)_e^\dagger\circ P_e =P_e\circ (dA_h)_e, 
\end{equation} 
 since for $X_e\in L(G)_e$ we have 
$$(dA_h)_e^\dagger\circ P_e(X_e)= (dA_h)_e^\dagger(X_e+L(H)_e) 
=(dA_h)_e(X_e) +L(H)_e = P_e\circ (dA_h)_e(X_e). $$
Also we note that 
\begin{equation}\label{j8e.3}
(d\tau(h))_o\circ (d\pi)_e = (d\pi)_e\circ (dA_h)_e, 
\end{equation}
since $\tau(h)\circ \pi= \pi\circ A_h$. Let 
$I=((d\pi)_e^\dagger)^{-1}: T_o(G/H) \to T_e(G)/L(H)_e$.  Then 
$P_e=I\circ (d\pi)_e$. Thus by \eqref{j8e.2} we have 
\begin{equation*}
(dA_h)_e^\dagger\circ I\circ (d\pi)_e =I\circ (d\pi)_e\circ (dA_h)_e, 
\end{equation*} 
and hence 
\begin{equation}\label{j8e.4}
I^{-1}\circ (dA_h)_e^\dagger\circ I\circ (d\pi)_e =(d\pi)_e\circ (dA_h)_e.  
\end{equation} 
Equations \eqref{j8e.3} and \eqref{j8e.4} imply 
$$(d\tau(h))_o\circ (d\pi)_e  = 
I^{-1}\circ (dA_h)_e^\dagger\circ I\circ (d\pi)_e, $$ 
from which it follows that 
\begin{equation}\label{j8e.e.5}
(d\tau(h))_o = 
I^{-1}\circ (dA_h)_e^\dagger\circ I.  
\end{equation} 
On the other hand, by a result of linear algebra, we have  
$$\det \left((dA_h)_e^\dagger\right) 
= \det (dA_h)_e/ \det \left((dA_h)_e|L(H)_e\right). 
$$ 
This and \eqref{j8e.e.5} imply  
$$\det \left((d\tau(h))_o\right) = \det \left((dA_h)_e^\dagger\right) 
= \det (dA_h)_e/ \det \left((dA_h)_e|L(H)_e\right). 
$$ 
This completes the proof of Lemma \ref{j8inv.l7}. 

\end{proof}

\section{Application to Grassmann manifolds} \label{sec14} 

We apply Proposition \ref{j8p.1.0} to the Grassmann manifolds.  

\begin{proposition}\label{j8p1}
Let $G_{k,n}$, $1\leq k<n$, 
 be the Grassmann manifold of $k$ dimensional subspaces in 
$\Bbb R^n$ 
$($see \cite[pp. 63--65]{Bo}, \cite[Section 2 of Chap. 13]{Sp}$)$. 
We identify $G_{k,n}$ with the homogeneous manifold 
${\rm SO}(n)/H$, where $H$ is the closed subgroup of ${\rm SO}(n)$ consisting 
of the elements of ${\rm SO}(n)$ having the form: 
$$ \begin{pmatrix} 
A  &  0
 \\
0  &   B 
\end{pmatrix} 
\qquad A\in {\rm O}(k), B\in {\rm O}(n-k), \quad \det A \det B =1$$  
$($\!see \cite[pp. 362--363]{Bo}, \cite[Chap. IV, \S 18]{M}$)$. 
Then there exists an ${\rm SO}(n)$-invariant measure on $G_{k,n}$. 
\end{proposition} 
Since compact Lie groups are unimodular, Proposition \ref{j8p.1.0} 
immediately implies the existence of an invariant measure on 
${\rm SO}(n)/H$. Here we give a proof according to \cite[Chap. V]{M}, 
which also provides an information on a condition when ${\rm SO}(n)/H$ 
is orientable. 
\begin{proof}[Proof of Proposition $\ref{j8p1}$]  
Let $\frak{o}(n)$ be the set of skew symmetric $n\times n$ 
real matrices.  Let $G={\rm SO}(n)$.  
Then $\frak{g}=L(G)=\frak{o}(n)$ and 
$$ \frak{h}=L(H)= \left\{\begin{pmatrix} 
C  &  0
 \\
0  &   D
\end{pmatrix}: C\in \frak{o}(k), D\in \frak{o}(n-k)\right\}.  $$
 Let $M(k,n-k)$ be the set of $k\times (n-k)$ real matrices and 
$$ \frak{m}=\left\{\begin{pmatrix} 
0 &  F
 \\
-F^t  &  0  
\end{pmatrix}:  
 F \in M(k,n-k)\right\},  $$  
where we recall that $F^t$ denotes the transpose of $F$. 
Then, we have a direct sum 
$$ \frak{g}=\frak{m}+\frak{h}. $$
Let $h\in H$. 
It is known that $Ad_G(h) : \frak{g} \to \frak{g}$ is defined by 
$Ad_G(h)X=hXh^{-1}$ and that 
 $Ad_H(h) : \frak{h} \to \frak{h}$ is defined by $Ad_H(h)Y=hYh^{-1}$.  
Define $T_h: \frak{m} \to \frak{m}$ by 
$T_h X=hXh^{-1}$, more specifically, 
$$T_h X=\begin{pmatrix} 
0 &  AFB^{-1} 
 \\
-BF^t A^{-1} &  0  
\end{pmatrix},  $$  
$$h=\begin{pmatrix} 
A  &  0
 \\
0  &   B 
\end{pmatrix}, 
\quad   X=\begin{pmatrix} 
0 &  F
 \\
-F^t  &  0  
\end{pmatrix}.   $$ 
Then,  the equality 
$$\left|\det Ad_G(h) \right|=\left|\det Ad_H(h) \right|   $$  
follows from 
\begin{equation}\label{j8e.5}
\left|\det T_h\right| =1. 
\end{equation}  
Thus, if we prove \eqref{j8e.5}, by Proposition \ref{j8p.1.0} 
we have the conclusion of Proposition \ref{j8p1}.   
\par
Define an isomorphism $J: \frak{m} \to  M(k,n-k)$ by $J(X)=F$ for 
$$ X=\begin{pmatrix} 
0 &  F
 \\
-F^t  &  0  
\end{pmatrix}.  $$ 
For $h\in H$ of the form 
$$ h= \begin{pmatrix} 
A  &  0
 \\
0  &   B 
\end{pmatrix} 
$$ 
define $S_h: M(k,n-k) \to M(k,n-k)$ by $S_h(F)=AFB^{-1}$. Then 
$$T_h=J^{-1}\circ S_h \circ J.  $$
Thus, \eqref{j8e.5} follows, if we show that $|\det S_h|=1$. 
To prove this we recall that 
 $\det S_h=(\det A)^{n-k}(\det B)^{k}$ (see \cite[Chap. V]{M}) 
and hence we see that $\det S_h= 1$ or $\det S_h= (-1)^n$ 
since $\det A=\det B=1$ or $\det A=\det B=-1$. 

\end{proof} 

We can find in \cite[3.2]{KP} another way of constructing an invariant measure 
on $G_{k,n}$ based on the Haar measure on the orthogonal group ${\rm O}(n)$. 

\begin{remark}\label{re14.2}   
In Section \ref{sec2} we considered the Grassmann manifold $M=G_{k,n}$ as a 
manifold equipped with the coordinate system given in  \cite[pp. 63--65]{Bo}.  
The reason we chose the structure of $M$ as the initial definition of the 
Grassmann manifold lies in proving Lemma \ref{L2.3} through the estimates 
for the Jacobian in \eqref{eeee2.1}. 
In Proposition \ref{j8p1} we regarded the Grassmann manifold 
 as the homogeneous manifold 
$N={\rm SO}(n)/H$ and proved the existence of the invariant measure $\mu$ on 
 $N$.  From this we can see the existence of the invariant measure $\sigma$ on 
$M$ in Section \ref{sec2}. 
Here we have a more specific explanation on this.  
It is known that there is a diffeomorphism $\varphi: M \to N$  such that 
$$\varphi(g\theta)= \tau(g)\varphi(\theta), $$ 
where $g\in G:={\rm SO}(n)$ and 
$\tau(g)$ is as in Section \ref{j8intro} 
(see \cite[Theorem 9.3 in Chap. IV]{Bo},  
\cite[Theorem of \S 17 in Chap. IV]{M}). 
For a Borel set $A$ in $M$, define $\sigma(A)= \mu(\varphi(A))$. Then $\sigma$ 
is a $G$ invariant measure, since 
$$\sigma(gA)=\mu(\varphi(gA))=\mu(\tau(g)\varphi(A))=\mu(\varphi(A))
=\sigma(A).  $$  
\par 
Next, we see the local coordinates expression of $\sigma$. 
 Let $\left\{(V_a, J_a): 1\leq a \leq L_1\right\}$, $L_1=\binom{n}{k}$,  and 
$\left\{(\tau(k)W, F_{k}): k \in G \right\}$ 
 be a $C^\infty$ coordinate system of $M$  
as in \cite[pp. 63--65]{Bo} and that of $N$ 
as in Section \ref{j8Some}, respectively.  
Take a finite covering 
$\left\{(\tau(k_l)W, F_{k_l}): k_l \in G, 1\leq l\leq L_2 \right\}$ 
of $N$. 
Then 
$M=\cup_{a, l} (V_a \cap \varphi^{-1}(\tau(k_l)W))$ 
($1\leq a \leq L_1, 1\leq l\leq L_2$).  Let $E$ be a Borel set such that 
$E \subset (V_a \cap \varphi^{-1}(\tau(k_l)W))$.  We write 
$J_a(\theta)=(y_1(\theta), \dots, y_m(\theta))$, 
$F_{k_l}(p)=(x_1(p), \dots, x_m(p))$, 
$m=k(n-k)$.  Then, as in   
Section \ref{sec12}, we have 
\begin{multline*}
\sigma(E)=\mu(\varphi(E))
=\int_{\varphi(E)} |h^{(k_l)}(p)| \left|(dx_1)_p\wedge \dots \wedge (dx_m)_p
\right|
\\ 
= \int_{F_{k_l}(\varphi(E))} |h^{(k_l)}(F_{k_l}^{-1}(x))| dx_1 \dots dx_m. 
\end{multline*}
Changing variables $x=F_{k_l}\circ \varphi\circ J_a^{-1}(y)=: \psi(y)$ in the 
integral above, we have 
 \begin{multline*}
\sigma(E)
= \int_{J_a(E)} |h^{(k_l)}(\varphi\circ J_{a}^{-1}(y))| |J(\psi)(y)|
\, dy_1 \dots dy_m 
\\ 
=\int_{E} |h^{(k_l)}(\varphi(\theta))| |J(\psi)(J_a(\theta))|
\left|(dy_1)_\theta \wedge \dots \wedge (dy_m)_\theta\right|.  
\end{multline*}

\end{remark}

\end{document}